\documentclass[a4paper]{amsart}
\usepackage{amssymb,stmaryrd,float}
\usepackage[shortlabels]{enumitem}
\providecommand*\texorpdfstring[2]{#1}

\newcommand\dm[1]{\mathring{#1}}

\newcommand\wo{\mathrel{\lhd}}
\newcounter{condition}

\newtheorem{thm}{Theorem}[section]

\newtheorem*{thma}{Theorem A}
\newtheorem*{thmb}{Theorem B}

\newtheorem{lemma}[thm]{Lemma}
\newtheorem{cor}[thm]{Corollary}
\newtheorem{claim}{Claim}[thm]
\newtheorem{prop}[thm]{Proposition}
\newtheorem{fact}[thm]{Fact}

\theoremstyle{definition}
\newtheorem{defn}[thm]{Definition}
\newtheorem{definition}[thm]{Definition}
\newtheorem{notation}[thm]{Notation}
\newtheorem{example}[thm]{Example}
\newtheorem{dependencies}[claim]{Dependencies}
\newtheorem{q}[thm]{Question}

\theoremstyle{remark}
\newtheorem{remark}[thm]{Remark}
\newtheorem{remarks}[thm]{Remarks}
\newtheorem{conv}[thm]{Convention}

\newcommand*\cvec[1]{\vec{\mathcal#1}}
\newcommand\diagonal{\bigtriangleup}

\newcommand\s{\subseteq}
\newcommand\sq{\sqsubseteq}
\newcommand*\sqleft[1]{\mathrel{_{#1}{\sqsubseteq}}}
\newcommand\sqx{\sqleft{\chi}}
\newcommand\sql{\sqleft{\lambda}}
\newcommand*\conc[2]{#1{}^\curvearrowright #2}
\newcommand*\sqleftup[1]{\mathrel{^{#1}{\sqsubseteq}}}
\newcommand*\sqstarleftup[1]{\mathrel{^{#1}{\sqsubseteq^*}}}
\newcommand*\sqleftboth[2]{\mathrel{^{#1}_{#2}{\sqsubseteq}}}
\newcommand*\sqstarleftboth[2]{\mathrel{^{#1}_{#2}{\sqsubseteq^*}}}
\newcommand\br{\blacktriangleright}
\newcommand\pvec{\fbox{${\cdots}$}\hspace{1pt}}
\newcommand\defaultaction{\textsf{extend}}
\newcommand\sealantichain{\textsf{anti}}
\newcommand\sealautomorphism{\textsf{auto}}
\newcommand\free{\textsf{free}}
\newcommand\sealpath{\textsf{sealpath}}
\newcommand\symdiff{\mathbin\triangle}
\DeclareMathOperator\suc{succ}

\def\sd{\framebox[3.0mm][l]{$\diamondsuit$}\hspace{0.5mm}{}}

\renewcommand\restriction{\mathbin\upharpoonright}
\renewcommand\mid{\mathrel{|}\allowbreak}
\newcommand*\pred[1]{#1_\downarrow}
\newcommand*\cone[1]{{#1}^\uparrow}
\DeclareMathOperator{\height}{ht}
\DeclareMathOperator{\reg}{Reg}
\DeclareMathOperator{\cf}{cf}
\DeclareMathOperator{\dom}{dom}
\DeclareMathOperator{\im}{Im}
\DeclareMathOperator{\otp}{otp}
\DeclareMathOperator{\acc}{acc}
\DeclareMathOperator{\nacc}{nacc}
\DeclareMathOperator{\mup}{mup}
\DeclareMathOperator{\tr}{Tr}
\DeclareMathOperator{\p}{P}
\newcommand\on{\text{ORD}}
\newcommand*\axiomfont[1]{\textsf{\textup{#1}}}
\newcommand\zfc{\axiomfont{ZFC}}
\newcommand\gch{\axiomfont{GCH}}

\newcommand\sh{\axiomfont{SH}}
\newcommand\ch{\textup{CH}}
\newcommand\ns{\textup{NS}}
\newcommand\bd{\textup{bd}}
\newcommand\ind{\textup{ind}}
\newcommand\stree{\subsetneq}

\title[Microscopic approach, part II]{A microscopic approach to Souslin-tree construction, Part II}

\author{Ari Meir Brodsky}
\author{Assaf Rinot}
\address{Department of Mathematics, Bar-Ilan University, Ramat-Gan 5290002, Israel.}
\urladdr{http://u.math.biu.ac.il/~brodska/}
\urladdr{http://www.assafrinot.com}

\begin{document}
\begin{abstract}
In Part~I of this series, we presented the microscopic approach to Souslin-tree constructions,
and argued that all known $\diamondsuit$-based constructions of Souslin trees with various additional properties may be rendered as applications of our approach. 
In this paper, we show that constructions following the same approach may be carried out even in the absence of $\diamondsuit$. 
In particular, we obtain a new weak sufficient condition for the existence of Souslin trees at the level of a strongly inaccessible cardinal.

We also present a new construction of a Souslin tree with an ascent path,
thereby increasing the consistency strength of such a tree's nonexistence from a Mahlo cardinal to a weakly compact cardinal.

Section 2 of this paper is targeted at newcomers with minimal background. 
It offers a comprehensive exposition of the subject of constructing Souslin trees and the challenges involved.
\end{abstract}

\maketitle
\tableofcontents

\section{Introduction}

The systematic study of set-theoretic trees was pioneered by \DJ uro Kurepa in the 1930s~\cite{kurepa1935ensembles},
in the context of examining Souslin's Problem. Souslin's Problem 
goes back a century, to 1920~\cite{Souslin}, and its most succinct formulation is:
\begin{quote}
Is every linearly ordered
topological space satisfying the countable chain condition (ccc) necessarily separable?
\end{quote}
A counterexample would be called a \emph{Souslin line},\footnote{If we remove the constraint that the topology be induced by a linear order,
then there is no difficulty in obtaining a counterexample, 
such as the \emph{countable complement topology} on any uncountable set 
\cite[Counterexample~20]{steen1978counterexamples}.} 
while the conjecture that the answer is ``yes'' (meaning that a Souslin line does \emph{not} exist)
has come to be called \emph{Souslin's Hypothesis} (\sh).

In the course of attempting to prove \sh,
Kurepa showed 
in 1935~\cite{kurepa1935ensembles} that the problem can be reformulated in terms of trees,\footnote{\label{footnote:Souslin-equivalence}Kurepa states the equivalence between the topological and tree-based formulations of Souslin's Problem
(at the level of $\aleph_1$)
explicitly 
in~\cite[\S12.D.2, pp.~124--125]{kurepa1935ensembles} \cite[p.~111]{kurepa1996selected}
and \cite[Section~8, p.~134]{Kurepa-Aronszajn} \cite[p.~119]{kurepa1996selected}.
Several sources \cite[p.~1116]{rudin1969souslin}, \cite[Section~2.1, p.~202]{solovay1971iterated}, 
\cite[p.~12]{devlin1974souslin}, \cite[\S2, p.~421]{malykhin1996suslin}
attribute the reformulation to E.~W. Miller in 1943~\cite{miller1943note},
perhaps because Kurepa's thesis and early papers were written in French.
Others (\cite[p.~213]{MR1711574}, \cite[p.~3]{kanamori2011historical}, Todorcevic in~\cite[p.~9]{kurepa1996selected})
acknowledge that Miller rediscovered Kurepa's result.}
and thus 
``eliminated topological considerations from Souslin's Problem and 
reduced it to a problem of combinatorial set theory''~\cite[p.~3]{kanamori2011historical}.
Kurepa's result is that the existence of a Souslin line is equivalent to the existence of 
(what we now call) an $\aleph_1$-Souslin tree, that is,
a tree of size $\aleph_1$ that includes neither an uncountable branch nor an uncountable antichain.\footnote{Detailed definitions will be given in Section~\ref{rightway}.}

Further progress toward resolving Souslin's problem came only in the 1960s,
after the advent of the forcing technique,
when it became apparent that Souslin's problem (at the level of $\aleph_1$) is independent of \zfc:
Jech~\cite{jech1967non}, Tennenbaum~\cite{tennenbaum1968souslin}, and Jensen~\cite{jensen1968souslin}
gave consistent constructions of $\aleph_1$-Souslin trees,
while Solovay and Tennenbaum \cite{solovay1971iterated} proved the consistency of \sh.
Amazingly enough, the resolution of this single problem led to key discoveries in set theory:
various notions of trees \cite{kurepa1935ensembles},
forcing axioms and the method of iterated forcing \cite{solovay1971iterated},
the \emph{diamond} and \emph{square} principles \cite[\S5--6]{jensen1972fine},
and the theory of iteration without adding reals \cite[Chapter~VIII]{devlin1974souslin}.

Most of the early work around Souslin's problem focused on the level of $\aleph_1$,
and even to this day,
most of the standard references in set theory,
including \cite{MR3408725}, \cite{MR597342}, \cite{todorcevic1984trees}, \cite{MR1028781}, \cite{just1997discovering}, 
\cite{MR1697766}, \cite{MR1924429}, 
provide a construction of a $\kappa$-Souslin tree only for the case $\kappa=\aleph_1$.
However, Souslin's problem admits a natural generalization to higher cardinals.
Indeed, Kurepa proved the following more general equivalence:

\begin{fact}[Kurepa, \cite{kurepa1935ensembles}]\label{SH-equiv-general}
For any regular uncountable cardinal $\kappa$,
the following are equivalent:\footnote{Although Kurepa's explicit statements 
(referenced in footnote~\ref{footnote:Souslin-equivalence})
refer specifically to the level of $\aleph_1$,
he proves the result for arbitrary infinite cardinals via
the equivalence $P_2 \iff P_5$ of the Fundamental Theorem
in the Appendix~\cite[\S C.3, pp.~132--133]{kurepa1935ensembles}.}
\begin{itemize}
\item
Every tree of size $\kappa$ contains
either a branch of size $\kappa$ or an antichain of size $\kappa$
(that is, there is no $\kappa$-Souslin tree);
\item
Every linearly ordered topological space satisfying the $\kappa$-chain condition ($\kappa$\nobreakdash-cc)
has a dense subset of cardinality $<\kappa$.
\end{itemize}
\end{fact}

The preceding leads to the following definition.

\begin{definition}[{\cite[p.~292]{jensen1972fine}}]
For any regular uncountable cardinal $\kappa$,
the \emph{$\kappa$-Souslin hypothesis} ($\sh_\kappa$) 
asserts that there are no $\kappa$-Souslin trees.
\end{definition}
Jensen proved \cite[Theorem~6.2]{jensen1972fine} 
that, assuming $V=L$, for every regular uncountable cardinal $\kappa$,
$\sh_\kappa$ holds iff $\kappa$ is weakly compact.
Subsequently, many combinatorial constructions of $\kappa$-Souslin trees (for regular uncountable cardinals $\kappa$ that are not weakly compact) from axioms 
weaker than $V=L$ have appeared.\footnote{See \cite[Theorem~IV.2.4]{MR750828}, as well as Fact~1.2 of \cite{paper24}
and the historical remarks preceding it.}
However, 
the classical constructions of $\kappa$-Souslin trees generally depend on the nature of $\kappa$:
that is, on whether $\kappa$ is the successor of a regular cardinal
\cite{MR485361}, \cite{sh:e4}, \cite{MR830071}; 
the successor of a singular cardinal \cite{MR836425}, \cite[\S4]{rinot11}; 
or an inaccessible cardinal \cite{Sh:624}.

Furthermore,
the classical $\diamondsuit$-based constructions all require $\diamondsuit$ to concentrate on a
\emph{nonreflecting stationary set},
in order to ensure that sealing antichains doesn't prevent us from later building higher levels of the tree. 
Thus, classical methods cannot be applied in scenarios where all stationary sets reflect,
and thus they allow inferring
the consistency of only a Mahlo cardinal from $\gch$ and the non-existence of a higher Souslin tree.

In addition, there is a zoo of consistent constructions of $\kappa$-Souslin trees satisfying additional properties,
such as complete, regressive, rigid, homogeneous, specializable, non-specializable, admitting an ascent path, omitting an ascending path, free and uniformly coherent.
Again, construction of a $\kappa$-Souslin tree with any desired property often depends on the nature of $\kappa$,
and in some cases even depends on whether $\kappa$ is the successor of a singular cardinal 
of countable or of uncountable cofinality \cite{MR1376756}.
To obtain the additional features, constructions include extensive 
bookkeeping, counters, timers, coding and decoding, 
whose particular nature makes it difficult to transfer the process from one type of cardinal to another.

What happens if we want to replace an axiom known to imply the existence of a $\kappa_0$-Souslin tree with strong properties
by an axiom from which a plain $\kappa_1$-Souslin tree can be constructed?
Do we have to revisit each scenario and tailor each of these particular constructions 
in order to derive a tree with strong properties?

In~\cite{paper22}, which forms the starting point of this research project,
we set out to develop new foundations that enable uniform construction of $\kappa$-Souslin trees;
we introduced a single (parameterized) \emph{proxy principle} from which $\kappa$-Souslin trees 
with various additional features can be constructed, regardless of the identity of $\kappa$.
In that paper, we also built the bridge between the old and new foundations,
establishing, among other things, that all known $\diamondsuit$-based constructions of $\kappa$-Souslin trees
may be redirected through this new proxy principle. 
There was one scenario that was not covered by that paper,
namely, Jensen's construction from ${\square(E)}+{\diamondsuit(E)}$ \cite[Theorem~6.2]{jensen1972fine},
and in Subsection~\ref{section:captureJensen} of the present paper, we cover it.
This means that any $\kappa$-Souslin tree with additional features that will be shown to follow from the proxy principle
will automatically be known to hold in many unrelated models.

But the parameterized proxy principle gives us more:

$\br$ It suggests a way of calibrating the fineness of a particular class of Souslin trees,
by pinpointing the weakest vector of parameters sufficient for the proxy principle to enable construction of a member of this class.
This leads, for instance, to the understanding that uniformly coherent $>$ free $>$ specializable $>$ plain.
This is explained in Section~\ref{constructions-section} below.

$\br$ It allows comparison and amplification of previous results.

In \cite{just2001like},\cite{MR2320769}, and \cite{rinot09}, new weak forms of $\diamondsuit$ at the successor of a regular cardinal $\lambda$
were proposed and shown to entail the existence of $\lambda^+$-Souslin trees.
In this project, we put all of these principles under a single umbrella by computing the corresponding vector of parameters for which the proxy principle holds in each of the previously studied configurations.
From this and the constructions we presented in \cite{paper32}, it follows, for example, that the Gregory configuration \cite{MR485361} suffices for the construction of a specializable $\lambda^+$-Souslin tree,
and the K\"onig--Larson--Yoshinobu configuration \cite{MR2320769} suffices for the construction of a free $\lambda^+$-Souslin tree.

Moreover, pump-up lemmas such as \cite[Lemma~4.9]{paper29} and \cite[Lemma~3.8]{paper32} establish that strong instances of the proxy principle may be derived from apparently weaker ones,
leading, for example, to the existence of $\kappa$-Souslin trees with a maximal degree of completeness in unexpected scenarios.

$\br$ It allows the construction of various types of trees at a broader class of cardinals.
To give two examples:

$\br\br$ A combinatorial construction of a free $\kappa$-Souslin tree for $\kappa=\aleph_1$ may be found in \cite[Theorem~V.1]{devlin1974souslin}, \cite[Theorem~6.6]{todorcevic1984trees} and \cite[\S2.1]{AbSh:403}.
In \cite[\S6]{rinot20} and in \cite[\S4.3]{paper32}, we gave new combinatorial constructions of free $\kappa$-Souslin trees, both using the proxy principle, and therefore
they automatically apply to all regular uncountable cardinals $\kappa$, including successors of singular cardinals.

$\br\br$ A combinatorial construction of a uniformly coherent $\kappa$-Souslin tree for a successor of a regular cardinal $\kappa$ may be found in \cite[Theorem~IV.1]{devlin1974souslin}, \cite{MR1683897}, and \cite{MR830071}.
In \cite[Theorem~2.5]{paper22}, we gave a proxy-based construction of a uniformly coherent $\kappa$-Souslin tree, and therefore
it automatically applies to all regular uncountable cardinals $\kappa$, including inaccessible cardinals.

$\br$ It allows obtaining completely new types of Souslin trees.

Once we have suitable foundations, the construction of Souslin trees becomes simple, 
and it is then easier to carry out considerably more complex constructions.
For example, in \cite[Theorem~1.1]{rinot20},
we gave the first example of a Souslin tree whose reduced powers behave independently of each other;
starting from a combinatorial hypothesis that follows from ``$V=L$'', 
we constructed an ultrafilter $\mathcal U_0$ over $\aleph_0$ and
an ultrafilter $\mathcal U_1$ over $\aleph_1$
such that, for every $(i,j)\in 2\times 2$,
there exists an $\aleph_3$-Souslin tree $T$ for which
$T^{\aleph_0}/\mathcal U_0$ is $\aleph_3$-Aronszajn iff $i=1$ and 
$T^{\aleph_1}/\mathcal U_1$ is $\aleph_3$-Aronszajn iff $j=1$.

$\br$ It paves the way to finding completely new scenarios in which Souslin trees must exist,
by finding new configurations in which an instance of the proxy principle holds.
To give several examples:

$\br\br$ In \cite[Corollary~1.20]{paper22}, we constructed a model of Martin's Maximum 
in which, for every regular cardinal $\kappa>\aleph_2$, a strong instance of the proxy principle at $\kappa$
(strong enough to yield a free $\kappa$-Souslin tree) holds.

$\br\br$ In \cite[Theorem~6.3]{paper22}, we proved that the sufficient condition of Gregory for the existence of a Souslin tree
at the successor of a regular uncountable cardinal \cite{MR485361}  yields an instance of the proxy principle,
and then, in \cite[Corollary~3.4]{paper32}, this was generalized to include successors of singulars, as well.

$\br\br$  In \cite[Corollary~4.14]{paper24}, the second author proved that for every uncountable cardinal $\lambda$, 
${\gch}+{\square(\lambda^+)}$ entails an instance of the proxy principle 
sufficient for the construction of a $\cf(\lambda)$-complete $\lambda^+$-Souslin tree.
It follows that if $\gch$ holds and there are no $\aleph_2$-Souslin trees, 
then $\aleph_2$ is a weakly compact cardinal in $L$,
thus improving the lower bound obtained by Gregory 40 years earlier \cite{MR485361}.

$\br\br$ 
In \cite{Sh:176}, Shelah proved that adding a single Cohen real indirectly adds an $\aleph_1$-Souslin tree.
In the same spirit, in \cite{paper26}, 
we identified a large class of notions of forcing that, assuming a $\gch$-type hypothesis, add a very strong instance of the proxy principle at the level of $\lambda^+$.
This class includes (but is not limited to) notions of forcing for changing the cofinality of an inaccessible cardinal $\lambda$, such as Prikry, Magidor and Radin forcing.

$\br$ It gives rise to combinatorial constructions of $\kappa$-Souslin trees even in the absence of $\diamondsuit$,
which is something we did not anticipate, but is established in Sections \ref{xboxdiamond} and \ref{constructions-section} below.

$\br$ It even gives rise to results in other topics,
such as special and non-special Aronszajn trees \cite{paper29}, infinite graph theory \cite{paper28} and Ramsey theory \cite{paper45}.

\subsection{Two results of particular interest}
In the 1980s (see \cite{MR0732661}), Baumgartner proved that ${\gch}+{\square_{\aleph_1}}$ 
entails the existence of an $\aleph_2$-Souslin tree with an $\omega$-ascent path (see Definition~\ref{defascentpath} below). 
A special case of Corollary~\ref{cor713} reads as follows.

\begin{thma} ${\gch}+{\square(\aleph_2)}$ entails the existence of an $\aleph_2$-Souslin tree with an $\omega$-ascent path.
\end{thma}
\begin{remark} The significance of this improvement is that the consistency strength of the failure of $\square_{\aleph_1}$ is a Mahlo cardinal,
whereas the consistency strength of the failure of $\square(\aleph_2)$ is a weakly compact cardinal.
\end{remark}

As alluded to earlier, Jensen's construction of $\kappa$-Souslin trees in $L$ \cite[Theorem~6.2]{jensen1972fine} 
goes through the hypothesis that there exists a stationary subset $E\subset\kappa$ for which $\square(E)$ and $\diamondsuit(E)$ both hold.
Here, we obtain the same conclusion from weaker hypotheses,
which is best seen for $\kappa$ inaccessible and $E\s E^\kappa_{>\omega}$:
\begin{thmb} Suppose that $\kappa$ is a strongly inaccessible cardinal,
and there exists a sequence $\langle A_\alpha\mid \alpha\in E\rangle$ such that:
\begin{itemize}
\item $E$ is a nonreflecting stationary subset of $E^\kappa_{>\omega}$;
\item For every $\alpha\in E$, $A_\alpha$ is a cofinal subset of $\alpha$;
\item For every cofinal $B\s\kappa$, there exists $\alpha\in E$
for which $$\{\delta<\alpha\mid \min(A_\alpha\setminus(\delta+1))\in B\}$$ is stationary in $\alpha$.
\end{itemize}
Then there exists a $\kappa$-Souslin tree.
\end{thmb}

\subsection{Conventions}\label{conventionsforpaper} Throughout the paper,
$\kappa$ stands for an arbitrary regular uncountable cardinal;
$\theta,\lambda,\mu,\nu,\chi$ are (possibly finite) cardinals $\le\kappa$;
and $\xi,\sigma$ are ordinals $\le\kappa$.

\subsection{Notation}
We let $H_\kappa$ denote the collection of all sets of hereditary cardinality less than $\kappa$ (cf.~\cite[IV, $\S6$]{MR597342}).
We let $\reg(\lambda)$ denote the set of all infinite regular cardinals below $\lambda$.
We say that $\kappa$ is $({<}\chi)$-closed iff $\lambda^{<\chi}<\kappa$ for every $\lambda<\kappa$.
Denote $E^\lambda_\theta:=\{\alpha<\lambda\mid \cf(\alpha)=\theta\}$, and define $E^\lambda_{\neq\theta}$, $E^\lambda_{<\theta}$, $E^\lambda_{>\theta}$, and $E^\lambda_{\ge\theta}$ in a similar fashion.
Write $[\lambda]^\theta$ for the collection of all subsets of $\lambda$ of cardinality $\theta$, and define $[\lambda]^{<\theta}$ similarly.
Write  $\ch_\lambda$ for the assertion that $2^\lambda=\lambda^+$.

Suppose that $C$ is a set of ordinals.
Write $\acc(C):=\{\alpha\in C\mid \sup (C\cap\alpha) = \alpha>0 \}$, $\nacc(C) := C \setminus \acc(C)$,
$\acc^+(C) := \{\alpha<\sup(C)\mid \sup (C\cap\alpha) = \alpha>0 \}$.
In particular, $\acc(\kappa)$ is the set of all nonzero limit ordinals below $\kappa$.
For any  $j < \otp(C)$, denote by $C(j)$ the unique element $\delta\in C$ for which $\otp(C\cap\delta)=j$.
Write $\suc_\sigma(C) := \{ C(j+1)\mid j<\sigma\ \&\allowbreak\ j+1<\otp(C)\}$.
In particular, for all $\gamma\in C$ such that $\sup(\otp(C\setminus\gamma))\ge\sigma$,
$\suc_\sigma(C\setminus\gamma)$ consists of the next $\sigma$ many successor elements of $C$ above $\gamma$.

The class of ordinals is denoted by $\on$. 
For all $\alpha<\kappa$, $t:\alpha\rightarrow\kappa$, and $i<\kappa$, we denote by 
$\conc{t}{i}$ the unique function $t'$ extending $t$ satisfying $\dom(t')=\alpha+1$ and $t'(\alpha)=i$.

\section{How to construct a Souslin tree the right way}\label{rightway}
This section is accessible to novices with just basic background in Set Theory.
\subsection{Trees}
\label{subsection:trees}
A \emph{tree} is a partially ordered set $(T,{<_T})$ with the property that, for every $t\in T$,
the downward cone $t_\downarrow:=\{ s\in T\mid s<_T t\}$ is well-ordered by $<_T$. The \emph{height} of $t\in T$, denoted $\height(t)$,
is the order-type of $(t_\downarrow,{<_T})$. Then, for any ordinal $\alpha$, the $\alpha^{\text{th}}$ level of $(T,{<_T})$ is the set $T_\alpha:=\{t\in T\mid \height(t)=\alpha\}$;
the \emph{height of the tree} $(T,{<_T})$ is the smallest ordinal $\alpha$ such that $T_\alpha = \emptyset$.
For $X\s\on$, we write $T \restriction X := \{t \in T \mid \height(t) \in X \} = \bigcup_{\alpha\in X} T_\alpha$;
in particular, if $\alpha$ is any ordinal, then the tree $T\restriction\alpha$ has height $\leq\alpha$.
For any $s,t \in T$, we say that $s$ and $t$ are \emph{comparable} if $s <_T t$ or $t <_T s$ or $s=t$;
otherwise they are \emph{incomparable}.

There are several natural properties that the trees we construct will always satisfy.
In particular,
a tree $(T,{<_T})$ is said to be:
\begin{itemize}
\item \emph{Hausdorff} if for any limit ordinal $\alpha$ and $s,t\in T_\alpha$,
$(s_\downarrow=t_\downarrow)\implies (s=t)$;\footnote{As 0 is a limit ordinal, 
any (nonempty) Hausdorff tree is, in particular, \emph{rooted} --- that is, the level $T_0$ is a singleton, whose unique element is called the \emph{root}.}
\item \emph{normal} if for any pair of ordinals $\alpha < \beta$
and every $s \in T_\alpha$, if $T_\beta\neq\emptyset$ then there exists some $t \in T_\beta$ such that $s <_T t$;
\item \emph{ever-branching} if, for every node $s \in T$,
the upward cone $\cone{s} := \{ t \in T \mid s <_T t \}$ is \emph{not} linearly ordered by $<_T$.
\end{itemize}
\subsection{Souslin trees} 
Suppose that $(T,{<_T})$ is a tree. For any ordinal $\alpha$,
we say that a subset $B\s T$ is an \emph{$\alpha$-branch} if $(B,{<_T})$ is linearly ordered and
$\{ \height(t)\mid t\in B\}= \alpha$.
We say that $B\s T$ is a \emph{cofinal branch} if it is a $\kappa$-branch,
where $\kappa$ is the height of $(T,{<_T})$.
We say that $A\s T$ is an \emph{antichain} if any two distinct $s,t\in A$ are incomparable.

A tree $(T,{<_T})$ is a \emph{$\kappa$-tree} whenever its height is $\kappa$ 
and $|T_\alpha|<\kappa$ for all $\alpha<\kappa$.\footnote{Recall that $\kappa$ denotes a regular uncountable cardinal.}
A \emph{$\kappa$-Aronszajn tree} is a $\kappa$-tree with no cofinal branches.
A \emph{$\kappa$-Souslin tree} is a $\kappa$-Aronszajn tree that has no antichains of size $\kappa$.

A significant focus of this paper is the construction of $\kappa$-Souslin trees.
The most natural way to do this is to construct, recursively, a sequence $\langle T_\alpha\mid \alpha<\kappa\rangle$ of levels
whose union will ultimately be the desired $\kappa$-Souslin tree.
However, in order to ensure that the outcome tree will have
neither any cofinal branches nor any antichains of size $\kappa$,
we must find ways to anticipate these ``global properties'' of the tree 
when constructing each level.
The following well-known lemma (cf.~{\cite[Lemma~II.7.4]{MR597342}}) shows that if we ensure throughout the construction that our $\kappa$-tree is ever-branching
(a ``local property'', which can be ensured level by level throughout the construction),
then we can avoid the necessity of verifying that it has no cofinal branches.

\begin{lemma}\label{enough-no-antichains} 
For any ever-branching $\kappa$-tree $\mathcal T$, the following are equivalent:
\begin{itemize}
\item $\mathcal T$ is $\kappa$-Souslin;
\item $\mathcal T$ has no antichains of size $\kappa$.
\end{itemize}
\end{lemma}
\begin{proof} The forward implication is obvious. 
Next, assume that $\mathcal T=(T,{<_T})$ is an ever-branching $\kappa$-tree,
having no antichains of size $\kappa$.
Towards a contradiction, suppose that $\mathcal T$ admits a cofinal branch, say, $B$.
As $\mathcal T$ is ever-branching, for every $t\in B$, $\cone{t}$ is not linearly ordered, so that
we may fix $t'\in \cone{t}\setminus B$.
Recursively construct $B^\bullet \in [B]^\kappa$ such that,
for any two $s<_T t$ both in $B^\bullet$, $\height(s') < \height(t)$.
As $\height(t)<\height(t')$ for every $t \in B$, 
it follows that $\{ t'\mid t\in B^\bullet \}$ forms an antichain of size $\kappa$,
contradicting our hypothesis and thereby completing the proof.
\end{proof}

It follows that if we construct an ever-branching $\kappa$-tree,
our main worry is to ensure the non-existence of large antichains. 
Furthermore, the following well-known fact (cf.~{\cite[Lemma~2.4]{rinot20}}) shows that we do not
lose any opportunities by insisting that the trees we construct are normal and ever-branching. 
\begin{fact} Suppose $(T,{<_T})$ is a $\kappa$-Souslin tree.
Then there is a normal and ever-branching subtree which is again $\kappa$-Souslin.
In fact, there is a club $C\s\kappa$ such that $({T\restriction C},{<_T})$ is normal and ever-branching.
\end{fact}

Thus, the existence of a $\kappa$-Souslin tree is equivalent to the existence of a normal ever-branching one.
In fact, the same is true for $\kappa$-Aronszajn trees
(cf.~\cite[Lemmas~II.5.11--12]{MR597342}).\footnote{In a sense,
normality is exactly the portion of K\H{o}nig's Lemma that can be salvaged at the height of an arbitrary regular cardinal,
and this is what makes the problem of constructing $\kappa$-Aronszajn and $\kappa$-Souslin trees challenging.}

\subsection{Streamlined trees}
What will our trees $(T, {<_T})$ look like? What are the elements of a tree, anyway?

Formally, of course, elements of a tree can be anything we choose.
However, for all of the trees that we construct here, elements of the tree will be (transfinite) sequences of ordinals,
and the tree-order $<_T$ will be the initial-sequence ordering 
(which is the same as ordinary proper inclusion ($\subsetneq$), if we view a sequence as a function).
To ensure that the height of an element in the tree corresponds to the element's length as a sequence,
we must ensure that our collection of sequences is closed under initial segments
(``downward-closed'').

In order to formalize this intuition while retaining some flexibility, we introduce the following definition (recall that for an ordinal $\alpha$ and any set $X$, $^\alpha X$ denotes the set of functions from $\alpha$ to $X$,
and $^{<\alpha}X := \bigcup_{\beta<\alpha} {}^{\beta} X$):

\begin{definition}
A set $T$ is a \emph{streamlined tree} iff there exists some cardinal $\kappa$ such that $T\s{}^{<\kappa}H_\kappa$
and, for all $t\in T$ and $\beta<\dom(t)$, 
$t\restriction\beta\in T$.\footnote{All of the $\kappa$-trees that we actually construct
will be subsets of $^{<\kappa}\kappa$,
but we shall also consider the broader $T\s{}^{<\kappa}H_\kappa$ when analyzing \emph{derived trees}
in Subsection~\ref{subsection:free}.}
\end{definition}

We shall freely use the following basic properties, whose verification is left to the reader.

\begin{lemma}\label{streamlined-basics}
For every streamlined tree $T\s{}^{<\kappa}H_\kappa$ and every ordinal $\alpha$:
\begin{itemize}
\item $(T,{\stree})$ is a Hausdorff tree in the abstract sense of Subsection~\ref{subsection:trees}.
\item Assuming $T$ is nonempty, its root is the empty sequence, $\emptyset$.
\item For every $t \in T$, $\height(t) = \dom(t)$
and $\pred{t} = \{ t \restriction \beta \mid \beta < \dom(t) \}$.
\item $T_\alpha=T\cap{}^{\alpha}H_\kappa$. In particular, $T\restriction\alpha = T\cap{}^{<\alpha}H_\kappa$.
\item For every $t \in T$, if $\alpha<\dom(t)$, then
$t\restriction\alpha$ is the unique element of $\pred{t}\cap T_\alpha$.
\item Any $\alpha$-branch $B \s T$ can be written as
$B = \{ t \restriction \beta \mid \beta < \alpha \}$ for some 
function $t : \alpha \rightarrow H_\kappa$. 
\item For all $s,t \in T$, $s$ and $t$ are comparable iff $s \cup t \in T$. \qed
\end{itemize}
\end{lemma}

The main advantage of streamlined trees is the identification of 
a limit of an increasing sequence of nodes.
Indeed, for any  $\stree$-increasing sequence $\eta$ of nodes, say,
$\eta=\langle t_\gamma \mid \gamma < \beta\rangle$,
the unique limit of this sequence, which may or may not be a member of the tree,
is nothing but $\bigcup_{\gamma<\beta}t_\gamma$, that is, $\bigcup\im(\eta)$.

It follows that when constructing a streamlined tree,
for any limit nonzero ordinal $\alpha$ such that all of the previous levels $\langle T_\beta\mid\beta<\alpha\rangle$ have already been determined, the definition of $T_\alpha$
amounts to deciding which $\alpha$-branches will have their limits inserted into the tree.  
Equivalently, for any cofinal subset $C$ of $\alpha$,
we shall have $T_\alpha\s \{ t\in{}^\alpha H_\kappa\mid \forall\beta\in C(t\restriction\beta\in T_\beta)\}$.

But are there any disadvantages here?
It turns out that we lose no generality by insisting on constructing only streamlined trees:

\begin{lemma}\label{sllemma} Suppose that $(X,{<_X})$ is a $\kappa$-tree.  Then:
\begin{enumerate}
\item If $(X,{<_X})$ is Hausdorff, then there exists a streamlined tree $T \s {}^{<\kappa}\kappa$ 
such that $(X,{<_X})$ is order-isomorphic to $(T,{\stree})$.
\item Regardless of whether or not $(X,{<_X})$ is Hausdorff, 
there exists a streamlined tree $S \s {}^{<\kappa}\kappa$ 
such that $(X,{<_X})$ is order-isomorphic to a cofinal subset of $(S,{\stree})$ via a level-preserving map.
\end{enumerate}
\end{lemma}
\begin{proof}
As $|X_\alpha|<\kappa$ for all $\alpha<\kappa$,
we may recursively find a sequence of injections $\langle \pi_\alpha:X_\alpha\rightarrow\kappa\mid \alpha<\kappa\rangle$
such that for all $\alpha<\beta<\kappa$,
$\sup(\im(\pi_\alpha))<\min(\im(\pi_\beta))$.
Let $\pi:=\bigcup_{\alpha<\kappa}\pi_\alpha$.
Note that if $y,z\in X$ satisfy $y<_X z$, then $\pi(y)<\pi(z)$.
\begin{enumerate}
\item Suppose $(X,{<_X})$ is Hausdorff.  For all $\delta<\kappa$ and $x\in X_\delta$, the set of ordinals $$\left\lceil x \right\rceil:=\{ \pi(y)\mid y\in X, y<_X x\}$$ has order-type $\delta$,
so we may let $t_x:\delta\rightarrow \left\lceil x \right\rceil$ denote the order-preserving isomorphism.
Evidently, $T:=\{ t_x\mid x\in X\}$ is a streamlined tree,
and $x\mapsto t_x$ forms an isomorphism between $(X,{<_X})$ and $(T,{\stree})$,
where injectivity is due to the fact that $(X,{<_X})$ is Hausdorff.
\item  For all $\delta<\kappa$ and $x\in X_\delta$, the set of ordinals $$[x]=\{ \pi(y)\mid y\in X, (y<_X x\text{ or }y=x)\}$$ has order-type $\delta+1$,
so we may let $s_x:\delta+1\rightarrow [x]$ denote the order-preserving isomorphism.
Evidently, $S:=\{ s_x\restriction\beta\mid x\in X,\beta<\kappa\}$ is a streamlined tree,
and $x\mapsto s_x$ forms an isomorphism between $(X,{<_X})$ and a cofinal subset of $(S,{\stree})$
sending level $\delta$ to level $\delta+1$,
where this time injectivity is due to the fact that we have included $\pi(x)$ in $[x]$.\qedhere
\end{enumerate}
\end{proof}

Thus, the existence of a $\kappa$-Aronszajn is equivalent to the existence of a streamlined one,
and the same is true for $\kappa$-Souslin trees.

\begin{conv} We shall say that $T$ is a \emph{streamlined $\kappa$-tree} if
$T \s {}^{<\kappa}H_\kappa$ is a streamlined tree and $(T,{\stree})$ is a $\kappa$-tree.
Furthermore, whenever we say that a streamlined tree $T$ is \emph{normal}, \emph{ever-branching}, \emph{Souslin}, etc.,
we mean to refer to the tree $(T,{\stree})$.
\end{conv}
\begin{remark}\label{streamlined-kappa-tree} 
Notice that any streamlined $\kappa$-tree $T$ is a subset of $H_\kappa$ and also has cardinality $\kappa$;
thus $T$ and all of its subsets are elements of $H_{\kappa^+}$.
\end{remark}

\subsection{Completing canonical branches and sealing antichains}\label{subsection:completing+sealing}
What does it take to build a $\kappa$-Souslin tree? 
Based on our previous discussion,
we shall want to build, level by level, a normal, ever-branching, streamlined $\kappa$-tree.

When constructing the level $T_\alpha$,
ensuring that the tree remains normal amounts to ensuring that
for every $s \in T\restriction\alpha$ we insert some $t$ into $T_\alpha$ satisfying $s \stree t$.
On the other hand, as all levels must be kept of size $<\kappa$ and as we must prevent the birth of large antichains,
there will be a stationary subset $\Gamma\s\kappa$ on which,
for every $\alpha\in\Gamma$,
$T_\alpha$ necessarily
must be some proper subset of $\{ t\in{}^\alpha H_\kappa\mid \forall\beta<\alpha (t\restriction\beta\in T_\beta)\}$.

But let us point out the challenge now arising in securing normality at the level $T_\alpha$,
where $\alpha<\kappa$ is some nonzero limit ordinal.
Suppose $x\in T\restriction\alpha$; we must include a node $\mathbf{b}^\alpha_x$ in $T_\alpha$
extending $x$ (the limit of the ``canonical $\alpha$-branch for $x$''). The natural way to do so is 
to pick a club $C_\alpha$ in $\alpha$ (``a ladder climbing up to $\alpha$''),
and then recursively identify an increasing and continuous sequence $\langle x_\beta \mid \beta\in C_\alpha\rangle$ 
of nodes of $T\restriction\alpha$
comparable with $x$ and satisfying $x_\beta\in T_\beta$ for all $\beta\in C_\alpha$. 
Normality up to level $\alpha$ makes the successor step of this recursion possible;
however, when we reach a limit step $\beta$ (that is, $\beta\in\acc(C_\alpha)$), this ordinal $\beta$ may be an element of $\Gamma$,
meaning that the unique limit of our partial sequence might have been excluded from $T_\beta$.
If we are extremely unlucky, we may have constructed $T\restriction\alpha$ to be an Aronszajn tree, with no $\alpha$-branches at all!
Thus, we have to define $T_\beta$ for $\beta\in\Gamma$
in an educated way so as to avoid such unfortunate scenarios. 
In the special case where $\kappa=\lambda^+$
for a (regular) cardinal $\lambda=\lambda^{<\lambda}$,\footnote{Including, for example, the simple case $\kappa=\aleph_1$.}
one can avoid this problem by simply taking $\Gamma$ to be $E^{\lambda^+}_\lambda$
and letting each ladder have order-type $\le\lambda$.
However, in the general case, there is a need for some coherent ladder system,
as we shall see in Definition~\ref{xboxdefinitionfrom22} below.

Recalling Lemma~\ref{enough-no-antichains}, we must also ensure that the resulting tree $T$
will not have any antichains of size $\kappa$.
The number of candidates for antichains of size $\kappa$ is $|{}^{<\kappa}H_\kappa|^\kappa$,
which is bigger than $\kappa$, the length of our recursive construction.
Put differently, there are not enough stages to take care of all of the candidates for large antichains, 
if we need to deal with them one at a time! 
In contrast, assuming $\kappa^{<\kappa}=\kappa$,
the number of candidates for proper initial segments of antichains is merely $\kappa$.
The upcoming lemma reduces the problem of eliminating antichains of size $\kappa$ to a problem of addressing their proper initial segments.

\begin{definition}Suppose $T$ is a streamlined $\kappa$-tree.
An antichain $A\s T$ is said to be \emph{sealed at level $\alpha$}
iff every element of $T_\alpha$ extends some element of $A$.
\end{definition}

\begin{lemma}\label{equiv-no-antichains} Suppose $T$ is a streamlined $\kappa$-tree.
Then the following are equivalent:
\begin{enumerate}
\item $T$ has no antichains of size $\kappa$;
\item For every antichain $A \s T$, there is some ordinal $\alpha<\kappa$ such that $A \s T \restriction\alpha$;
\item For every maximal antichain $A \s T$,
there is some ordinal $\alpha<\kappa$ such that $A \cap (T\restriction\alpha)$
is sealed at level $\alpha$.
\end{enumerate}
\end{lemma}
\begin{proof}
\begin{description}
\item[(1) $\implies$ (2)]
Let $A\subseteq T$ be any given antichain.
By (1), $|A|<\kappa$, so that by regularity of $\kappa$,
we obtain $\sup \{ \dom(x) \mid x \in A\} < \kappa$, as sought.

\item[(2) $\implies$ (3)]
Given any maximal antichain $A \s T$, fix $\alpha$ as in Clause~(2),
so that $A\subseteq T\restriction\alpha$.
Let $t \in T_\alpha$ be given.
As $A$ is a maximal antichain, there must be some $s \in A$ comparable with $t$.
But $\dom(s) < \alpha = \dom(t)$,
so it follows that $t$ extends $s$.

\item[(3) $\implies$ (1)]
Using Zorn's lemma, it is easy to see that every antichain is included in a maximal antichain.
Thus, it suffices to verify that $T$ has no \emph{maximal} antichains of size $\kappa$.

Given any maximal antichain $A \s T$, fix $\alpha$ as in Clause~(3).
As $T$ is a $\kappa$-tree, $|T_\beta|<\kappa$ for every $\beta$,
so that by regularity of $\kappa$ it follows that $|T \restriction\alpha| = \sum_{\beta<\alpha}|T_\beta| <\kappa$.
Thus, it suffices to prove that $A\s(T\restriction\alpha)$.

Consider any $u \in T \restriction [\alpha,\kappa)$; we shall show that $u \notin A$.
Let $t := u\restriction\alpha$, which is an element of $T_\alpha$. 
By our choice of $\alpha$ we can fix $s \in A$ with $s \stree t$.
Altogether, $s \stree u$.
As $s$ is an element of the antichain $A$ and $u$ properly extends $s$,
we infer that $u\notin A$.\qedhere
\end{description}
\end{proof}

In our discussion of the normality requirement, 
we already agreed that at limit levels $\alpha$,
$T_\alpha$ will consist of elements $\mathbf{b}^\alpha_x$ 
extending nodes $x\in T\restriction\alpha$.
In order to accomplish Clause~(3) of the preceding,
we now need to ensure that given an antichain $A$, 
each $\mathbf{b}^\alpha_x$ extends some element of $A \cap (T\restriction\alpha)$.
For this to be possible, every $x \in T \restriction\alpha$ must be comparable with some element of $A \cap (T\restriction\alpha)$,
meaning that $A \cap (T\restriction\alpha)$ must be a \emph{maximal} antichain 
in $T\restriction\alpha$. 
How do we locate ordinals at which properties of a given structure are replicated?
This question prompts the introduction of elementary submodels.

\subsection{Elementary submodels and diamonds}

Recall that for every regular uncountable cardinal $\varkappa$,
$(H_{\varkappa}, {\in})$ models all axioms of \zfc\ except possibly for the power-set axiom.
We shall be working extensively with elementary submodels $(\mathcal M, {\in})$ of $(H_{\varkappa}, {\in})$, though, by a slight abuse of notation, we will identify these structures with their underlying sets $\mathcal M$ and $H_{\varkappa}$, omitting the mention of the $\in$-relation.

A comprehensive exposition of elementary submodels of $H_\varkappa$ may be found in \cite[Chapter~24]{just1997discovering} and \cite[Chapter~4]{holz2010introduction}. For now, we shall only need to be aware
of the following corollary of the downward L\"owenheim--Skolem theorem.

\begin{fact}\label{cofinally-many-esms} For every parameter $p\in H_{\kappa^+}$, the following set is cofinal in $\kappa$:
$$B(p):=\{ \beta<\kappa\mid \exists \mathcal M\prec H_{\kappa^+}( p\in\mathcal M\ \&\ \mathcal M\cap\kappa=\beta)\}.$$
\end{fact}
\begin{remark} It is not hard to verify that the set $B(p)$ is, in fact, a club in $\kappa$. But we shall not need that.
\end{remark}

Given a well-founded poset $\mathbb P=(P,{\lhd})$ which is a subset of $H_\kappa$ and $\beta\in B(\mathbb P)$,
for any 
$\mathcal M \prec H_{\kappa^+}$ witnessing that $\beta\in B(\mathbb P)$,
the intersection $P\cap\mathcal M$ is a subset of $P$ that we can think of as being an \emph{initial segment} of $\mathbb P$.
The following proposition shows that this is precisely the case when $\mathbb P=(T,{\stree})$
and $T$ is a streamlined $\kappa$-tree, in which case the initial segment of $\mathbb P$ determined by $\mathcal M$ is nothing but $T\restriction\beta$.
Furthermore, global properties of $(T,{\stree})$ and its derivatives are reflected down to~$\beta$:

\begin{prop}\label{motivate} Suppose that $T$ is a streamlined $\kappa$-tree,
and $\beta\in B(T)$ as witnessed by $\mathcal M\prec H_{\kappa^+}$.
Then:
\begin{enumerate}
\item $T \cap \mathcal M = T\restriction\beta$;
\item If $ A\s T$ is a maximal antichain and $ A\in\mathcal M$, then $A\cap\mathcal M = A\cap(T\restriction\beta)$ is a maximal antichain in $T\restriction\beta$;
\item If $f:T\rightarrow T$ is a nontrivial automorphism, and $f\in\mathcal M$, then $f\cap\mathcal M=f\restriction(T\restriction \beta)$ is a nontrivial automorphism of $T\restriction\beta$.
\end{enumerate}
\end{prop}
\begin{proof}  
(1) For all $\alpha<\beta$, by $\alpha,T\in \mathcal M$, we obtain $T_\alpha\in \mathcal M$,
and by $\mathcal M\models |T_\alpha|<\kappa$,
we infer that $T_\alpha\s \mathcal M$. So $T\restriction\beta\s \mathcal M$.

As $\dom(z)\in \mathcal M\cap\kappa$ for all $z \in T\cap \mathcal M$,
we conclude that $T\cap \mathcal M=T\restriction\beta$.

(2) Suppose $A \in \mathcal M$ is a maximal antichain in $T$.
Since $H_{\kappa^{+}}\models A$ is a maximal antichain in $T$,
it follows by elementarity that
$$\mathcal M \models A\text{ is a maximal antichain in } T,$$
so that in fact $A\cap\mathcal M$ is a maximal antichain in $T\cap\mathcal M$.
But $T \cap \mathcal M = T\restriction\beta$ by Clause~(1),
so that also $A\cap\mathcal M = A \cap (T\restriction\beta)$.
Altogether, we infer that
$A\cap\mathcal M = A\cap(T\restriction\beta)$ is a maximal antichain in $T\restriction\beta$,
as sought.

(3) Left to the reader.
\end{proof}

It thus follows from Fact~\ref{cofinally-many-esms} and Proposition~\ref{motivate}(2) that for any maximal antichain $A\s T$,
we can find cofinally many ordinals $\beta<\kappa$
such that $A\cap(T\restriction\beta)$ is a maximal antichain in $T\restriction\beta$.
Coming back to our previous discussion,
we see that as we build our tree, we will be able to seal maximal antichains of the form $A\cap(T\restriction\beta)$, so that the challenge boils down to predicting $A\cap(T\restriction\beta)$
for each and every maximal antichain $A$ of the eventual tree $T$.
This leads us to discussing diamonds.

The combinatorial principle $\diamondsuit(\kappa)$ was coined by Jensen in \cite[p.~293]{jensen1972fine}.
Rather than giving its original definition,
we focus here on an equivalent formulation that is motivated by Fact~\ref{cofinally-many-esms}.

\begin{fact}[{\cite[Lemma~2.2]{paper22}}]\label{def_Diamond_H_kappa}
$\diamondsuit(\kappa)$ is equivalent to the existence of a sequence $\vec{A} = \langle  A_\beta \mid \beta < \kappa \rangle$
of elements of $H_\kappa$
such that, for every parameter $p\in H_{\kappa^{+}}$  and every subset $\Omega\subseteq H_\kappa$, the following set is cofinal in $\kappa$:
$$B(\Omega,p):=\{ \beta<\kappa\mid \exists \mathcal M\prec H_{\kappa^+}(\mathcal M\cap\Omega= A_\beta\ \&\ p\in\mathcal M\ \&\ \mathcal M\cap\kappa=\beta)\}.$$
\end{fact}
\begin{remark}\label{rmk_Diamond_H_kappa} It is not hard to verify that the cofinal set $B(\Omega,p)$ is, in fact, stationary in $\kappa$.
Also note that a sequence $\vec{A}$ as above must form an enumeration (with repetition) of all elements of $H_\kappa$, thus witnessing the fact that $\diamondsuit(\kappa)$ implies
$|H_\kappa| = \kappa^{<\kappa} = \kappa$.
\end{remark}

It follows from Fact~\ref{def_Diamond_H_kappa} that $\diamondsuit(\kappa)$ provides us a way to anticipate instances of Clause~(2) of Proposition~\ref{motivate}.

\begin{lemma}[cf.~{\cite[Claim~2.3.2]{paper22}}]\label{c232}
Suppose $\diamondsuit(\kappa)$ holds, as witnessed by a sequence $\vec{A} = \langle  A_\beta \mid \beta < \kappa \rangle$ as in Fact~\ref{def_Diamond_H_kappa}.

If $A$ is a maximal antichain in a given streamlined $\kappa$-tree $T$,
then the following set is cofinal in $\kappa$:
$$B := \{ \beta <\kappa \mid  A\cap(T\restriction\beta)= A_\beta\text{ is a maximal antichain in }T\restriction\beta \}.$$
\end{lemma}

\begin{proof}
Let $p := \{T, A\}$ and $\Omega:=A$.
Recalling Remark~\ref{streamlined-kappa-tree},
we infer that $p \in H_{\kappa^+}$ and $\Omega \s H_\kappa$,
so that by our choice of $\vec{A}$,
the corresponding set $B(\Omega,p)$ of Fact~\ref{def_Diamond_H_kappa} is cofinal in $\kappa$.
To see that $B(\Omega,p) \s B$,
consider any given $\beta \in B(\Omega,p)$, as witnessed by some $\mathcal M \prec H_{\kappa^+}$.
Since $p\in\mathcal M$, by elementarity we infer that $T, A \in \mathcal M$.
By Proposition~\ref{motivate}(2),
we then deduce that $A\cap\mathcal M = A\cap(T\restriction\beta)$ is a maximal antichain in $T\restriction\beta$.
But $\mathcal M\cap A=\mathcal M\cap\Omega= A_\beta$ by our choice of $\mathcal M$, 
and it follows that $\beta \in B$, as sought.
\end{proof}

Thus, when building the tree at the outset using a fixed diamond sequence $\vec{A}$,
we take advantage of the fact that, for many ordinals $\beta$,
$A_\beta$ will be a maximal antichain in $T\restriction\beta$.

\subsection{Coherent ladder systems}\label{subsection:coherent}

We now return to a point we alluded to earlier, in Subsection~\ref{subsection:completing+sealing}.
Suppose we are building a limit level $T_\alpha$.
For $x \in T\restriction\alpha$, in order to construct $\mathbf{b}^\alpha_x$, 
the limit of the ``canonical $\alpha$-branch for $x$'', 
we want to identify an increasing and continuous sequence $\langle x_\beta \mid \beta\in C_\alpha\rangle$ 
of nodes of $T\restriction\alpha$
comparable with $x$ and satisfying $x_\beta\in T_\beta$ for all $\beta\in C_\alpha$. 
In order to continue this recursion through a limit step $\beta \in \acc(C_\alpha)$,
we need to ensure that the limit of the partial sequence so-far identified was not excluded from $T_\beta$.
We do this by insisting on a uniform method for constructing $\mathbf{b}^\alpha_x$, 
so that the limit of the partial sequence $\langle x_\beta \mid \beta \in C_\alpha\cap\beta \rangle$
is exactly $\mathbf{b}^\beta_x$, the limit of the canonical $\beta$-branch for $x$, 
which we would have inserted into $T_\beta$ when constructing that level.
This insistence suggests several requirements whenever $\beta\in\acc(C_\alpha)$:
\begin{enumerate}
\item Coherence of the ladder system:  $C_\beta= C_\alpha\cap\beta$;
\item Microscopic perspective: the identification of the node $x_{\beta'}$, for $\beta'\in C_\beta$,
must not depend on whether we are heading towards $\mathbf{b}^\beta_x$ or $\mathbf{b}^\alpha_x$;
\item Smoothness: we must never exclude any $\mathbf{b}^\beta_x$ when constructing the level $T_\beta$.
\end{enumerate}
It should be clear that if we can comply with requirements (1)--(3) above, 
then we can construct a normal ever-branching $\kappa$-tree.
But we must not forget the task of sealing antichains,
and requirement (3) appears to conflict with the need to comply with Lemma~\ref{equiv-no-antichains}(3).
How can this be resolved?\footnote{A brief comparison of the classic non-smooth approach (requiring nonreflecting stationary sets)
and the modern approach may be found on \cite[p.~1965]{paper22}. The smoothness of our approach is witnessed by Fact~\ref{chrisreflectionthm} below.}

The answer lies in the subtlety of how we seal the antichains,
more precisely, in how we decide \emph{which} maximal antichain to seal at level $T_\alpha$.
Constructing the level $T_\alpha$ will not involve consulting the set $A_\alpha$ 
given by Fact~\ref{def_Diamond_H_kappa}.
Rather, when constructing $T_\alpha$, we will seal antichains that are predicted by $A_\beta$,
for ordinals $\beta\in\nacc(C_\alpha)$.
This approach respects requirements (2) and~(3) above,
but raises the concern of whether every maximal antichain will be predicted by $A_\beta$
for enough ordinals $\beta$ with $\beta\in\nacc(C_\alpha)$.
This concern, together with requirement~(1) above, motivates the following principle:

\begin{definition}[{\cite[Definition~1.3]{paper22}}]\label{xboxdefinitionfrom22} $\boxtimes^-(\kappa)$ asserts the existence of a sequence $\vec C=\langle C_\alpha\mid\alpha<\kappa\rangle$ such that:
\begin{itemize}
\item for all $\alpha<\kappa$, $C_\alpha$ is a closed subset of $\alpha$ with $\sup(C_\alpha)=\sup(\alpha)$;
\item for all $\alpha<\kappa$ and $\beta\in\acc(C_\alpha)$, $C_\beta=C_\alpha\cap\beta$;
\item for every cofinal $B\s\kappa$, there exists an infinite ordinal $\alpha<\kappa$   such that $\sup(\nacc(C_\alpha)\cap B)=\alpha$.
\end{itemize}
\end{definition}
\begin{remarks}\label{remarks-about-xbox}
\begin{enumerate}
\item The first bullet of Definition~\ref{xboxdefinitionfrom22} implies that
$C_0 = \emptyset$, that $\max(C_{\alpha+1}) = \alpha$ for every $\alpha<\kappa$,
and that $C_\alpha$ is a club in $\alpha$ for every $\alpha\in\acc(\kappa)$.
In particular, any $\alpha$ whose existence is asserted in the last bullet
must be a nonzero limit ordinal.
\item If we omit the last bullet, or even weaken it by removing ``$\nacc$'',
then we can trivially build a witnessing sequence by setting $C_\alpha := \alpha$ for every $\alpha<\kappa$.
As we have alluded to in this section, and will see in detail in the proof of Proposition~\ref{prop23},
it is the action of $\boxtimes^-(\kappa)$ at the non-accumulation points that enables the construction of a 
$\kappa$-Souslin tree by appropriately sealing the antichains without ruining the smoothness of the 
identification of canonical $\alpha$-branches throughout the construction.
\item The last bullet implies that the sequence $\vec C$ is \emph{unthreadable},
that is, there is no club $D\subseteq\kappa$ such that $D\cap\alpha = C_\alpha$ for all $\alpha\in\acc(D)$
(see \cite[Lemma~3.2]{paper22}). This bullet should be understood as a \emph{genericity} feature of the coherent sequence (cf.~\cite[Lemma~3.20]{paper28}).
\item Notice that in our primary application of the last bullet of $\boxtimes^-(\kappa)$,
as exemplified by the upcoming proof of Proposition~\ref{prop23},
the set $B$ will be the set of ordinals where a maximal antichain is \emph{predicted},
while the ordinal $\alpha$ will give us a level $T_\alpha$ at which such antichains are \emph{sealed}.
As we shall see, separating the set of ordinals where we predict a maximal antichain from the set of ordinals where we seal 
the predicted antichain will provide a great deal of flexibility.
\item We encourage the reader who is already familiar with the diamond and club principles
to verify that $\diamondsuit(\omega_1)\implies\clubsuit(\omega_1)\implies\boxtimes^-(\omega_1)$.\footnote{See Section~\ref{sectiononclub}.}
We also mention that, by \cite[Corollary~1.10(5)]{paper22}, if $V=L$, then $\boxtimes^-(\kappa)$ holds for every (regular uncountable cardinal) $\kappa$ that is not weakly compact.
In contrast, by \cite[Example~1.26]{paper22}, after L\'evy collapsing a weakly compact cardinal to $\omega_2$,
$\diamondsuit(\omega_2)$ holds, but $\boxtimes^-(\omega_2)$ fails.
\item The minus sign in the notation $\boxtimes^-(\kappa)$ is there to distinguish the latter from the stronger principle $\boxtimes(\kappa)$ of \cite[Definition~1.4]{paper22}.
\end{enumerate}
\end{remarks}

\subsection{A simple construction}

Having developed the machinery in the preceding subsections, we are now ready to prove the following proposition.

\begin{prop}[{\cite[Proposition~2.3]{paper22}}]\label{prop23} Suppose that $\boxtimes^-(\kappa)+\diamondsuit(\kappa)$ holds. Then there exists a  $\kappa$-Souslin tree.
\end{prop}
\begin{proof} Let $\vec{C} = \langle C_\alpha\mid \alpha<\kappa\rangle$ be a witness to $\boxtimes^-(\kappa)$.
Let $\vec{A} = \langle A_\beta\mid\beta<\kappa\rangle$ be given by Fact~\ref{def_Diamond_H_kappa}.
In addition, let $\lhd$ be some well-ordering of ${}^{<\kappa}2$.

As outlined earlier, we shall recursively construct a sequence $\langle T_\alpha\mid \alpha<\kappa\rangle$ of levels
such that $T:=\bigcup_{\alpha<\kappa}T_\alpha$ will form a normal, ever-branching, streamlined $\kappa$-Souslin tree. 
Furthermore, in this construction we shall ensure that for all $\alpha<\kappa$,
$T_\alpha$ will be a subset of ${}^\alpha2$ of size $\le\max\{\aleph_0,|\alpha|\}$.\footnote{This means that $T$ will be \emph{slim}, see Definition~\ref{def63}.}

\medskip

$\br$ Of course, we begin by letting $T_0:=\{\emptyset\}$.

$\br$ Successor levels are where we will ensure that the tree is ever-branching.
The simplest way to do that is to assign two immediate successors to every node from the previous level.
That is, for every $\alpha<\kappa$, we 
let 
$$T_{\alpha+1}:=\{ \conc{t}{0}, \conc{t}{1}\mid t\in T_\alpha\}.$$

$\br$ Suppose that $\alpha\in\Gamma$, where $\Gamma:=\acc(\kappa)$, and that $\langle T_\beta\mid \beta<\alpha\rangle$ has already been defined.
Recall that $T\restriction\alpha =\bigcup_{\beta<\alpha}T_\beta$,
and that
constructing the level $T_\alpha$ involves deciding which $\alpha$-branches
through $T \restriction \alpha$ will have their limits placed into the tree.
As discussed in Subsection~\ref{subsection:completing+sealing},
we must balance the normality requirement with the need to bound the size of $T_\alpha$ 
and to seal antichains.

Normality requires that for every $x \in T\restriction\alpha$ we include in $T_\alpha$ some node extending~$x$.
As $\alpha$ is a nonzero limit ordinal, our choice of the sequence $\vec{C}$ implies that
$C_\alpha$ is a club  in $\alpha$.
Thus, relying on the fact that the tree $T\restriction\alpha$ was constructed to be normal
(in particular, it is normal at each level $T_\beta$ for $\beta\in C_\alpha$),
and recalling that $T\restriction C_\alpha =\bigcup_{\beta\in C_\alpha}T_\beta$,
the idea for ensuring normality at level $T_\alpha$ is 
to attach to each node $x \in T\restriction C_\alpha$ 
some node $\mathbf{b}^\alpha_x \in {}^\alpha2$  above it,
and then let 
\[
T_\alpha := \{ \mathbf{b}^\alpha_x \mid x \in T \restriction C_\alpha\}.\tag*{$(*)_\alpha$}
\]

Let $x\in T\restriction C_\alpha$ be arbitrary.
We want $\mathbf{b}^\alpha_x$ to be the limit of some canonical $\alpha$-branch for $x$,
that is, some $\alpha$-branch through $T\restriction\alpha$ that contains $x$.
As $\sup(C_\alpha)=\alpha$,
it makes sense to describe $\mathbf{b}^\alpha_x$ as the limit $\bigcup\im(b^\alpha_x)$ of a sequence $b^\alpha_x\in\prod_{\beta\in C_\alpha\setminus\dom(x)}T_\beta$ such that:
\begin{itemize}
\item $b^\alpha_x(\dom(x))=x$;
\item $b_x^\alpha(\beta') \stree b_x^\alpha(\beta)$ for any pair $\beta'<\beta$ of ordinals from $C_\alpha\setminus\dom(x)$;
\item $b^\alpha_x(\beta)=\bigcup\im(b^\alpha_x\restriction\beta)$ for all $\beta\in\acc(C_\alpha\setminus\dom(x))$.
\end{itemize}
We build the sequence $b^\alpha_x$ by recursion:

Let $b^\alpha_x(\dom(x)):=x$.
Next, suppose $\beta^-<\beta$ are successive points of $C_\alpha\setminus\dom(x)$, and $b^\alpha_x(\beta^-)$ has already been defined.
In order to decide $b^\alpha_x(\beta)$, we consult the following set:
$$Q^{\alpha}_{x,\beta} := \{ t\in T_\beta\mid \exists s\in A_{\beta}\ (s\cup b^\alpha_x(\beta^-))\stree t\}.$$
Now, there are the two possibilities:
\begin{itemize}
\item If $Q^{\alpha}_{x,\beta} \neq \emptyset$, then let $b^\alpha_x(\beta)$ be its $\lhd$-least element.
\item Otherwise, let $b^\alpha_x(\beta)$ be the $\lhd$-least element of $T_\beta$ that extends $b^\alpha_x(\beta^-)$.
Such an element must exist, as the level $T_\beta$ was constructed so as to preserve normality.
\end{itemize}

The following is obvious, and is aligned with the microscopic perspective described in requirement (2) of Subsection~\ref{subsection:coherent}.

\begin{dependencies}\label{depends2171}
For any two consecutive points $\beta^-<\beta$ of $\dom(b^\alpha_x)$,
the value of $b^\alpha_x(\beta)$ is completely determined by
$b^\alpha_x(\beta^-)$, $A_\beta$, and $T_\beta$.
\end{dependencies}

Finally, suppose $\beta \in \acc(C_\alpha\setminus\dom(x))$ and $b^\alpha_x\restriction\beta$ has already been defined.
As promised, we let $b^\alpha_x(\beta):=\bigcup\im(b^\alpha_x\restriction\beta)$.
It is clear that $b^\alpha_x(\beta) \in {}^\beta 2$,
but we need more than that:

\begin{claim}\label{coherence}
$b^\alpha_x (\beta) \in T_\beta$.
\end{claim}
\begin{proof}
First, note that since $\beta \in \acc(C_\alpha)$ and  $\vec C$ is a $\boxtimes^-(\kappa$)-sequence,
$C_\alpha\cap\beta=C_\beta$, so that $x \in T\restriction C_\beta$.
So, by the induction hypothesis $(*)_\beta$, we infer that $\mathbf{b}^\beta_x$ is  in $T_\beta$.
As $\mathbf{b}^\beta_x=\bigcup\im(b^\beta_x)$ and $b^\alpha_x(\beta)=\bigcup\im(b^\alpha_x\restriction\beta)$,
it thus suffices to prove that $b^\beta_x=b^\alpha_x \restriction \beta$.

From $C_\alpha\cap\beta=C_\beta$, 
we obtain $\dom(b^\beta_x) = C_\beta \setminus \dom(x) =C_\alpha \cap \beta \setminus \dom(x) = \dom(b^\alpha_x) \cap \beta$. Call the latter by $d$.
Now, we prove that, for every $\delta \in d$, $b^\beta_x(\delta)=b^\alpha_x(\delta)$. By induction:
\begin{itemize}
\item Clearly, $b^\beta_x(\min(d)) = x = b^\alpha_x(\min(d))$.
\item Suppose $\delta^-<\delta$ are successive points of $d$,
and $b^\beta_x (\delta^-) = b^\alpha_x(\delta^-)$.
Then by Dependencies~\ref{depends2171}, also $b^\beta_x(\delta) = b^\alpha_x(\delta)$.

\item For $\delta \in \acc(d)$:
If the sequences are identical up to $\delta$, then their limits must be identical. \qedhere
\end{itemize}
\end{proof}

This completes the definition of the sequence $b^\alpha_x$, and thus of its limit $\mathbf{b}^\alpha_x$,
for each $x\in T\restriction C_\alpha$.
Consequently, the level $T_\alpha$ is defined as promised in $(*)_\alpha$.

Having constructed all levels of the tree, we then let
$T := \bigcup_{\alpha < \kappa} T_\alpha$.
It is clear from the construction that $T$ is a normal, ever-branching, streamlined $\kappa$-tree.
By Lemma~\ref{enough-no-antichains}, to prove that  $T$ is $\kappa$-Souslin,
it suffices to show that it has no $\kappa$-sized antichains.
By Lemma~\ref{equiv-no-antichains}, we thus fix an arbitrary maximal antichain $A \s T$,
and argue that there is some ordinal $\alpha<\kappa$ such that $A \cap (T\restriction\alpha)$ is sealed at level $\alpha$.

To find the sought-after ordinal $\alpha$, let
$$B := \{ \beta <\kappa \mid  A\cap(T\restriction\beta)= A_\beta\text{ is a maximal antichain in }T\restriction\beta \}.$$
By Lemma~\ref{c232}, $B$ is cofinal in $\kappa$.
Thus, by our choice of the sequence $\vec{C}$, let us fix an infinite ordinal $\alpha < \kappa$ for which $\sup (\nacc(C_\alpha) \cap B) = \alpha$.
Note that by Remark~\ref{remarks-about-xbox}(1), $\alpha\in\Gamma$.

\begin{claim} Every node of $T_\alpha$ extends some element of $A \cap (T\restriction\alpha)$.
\end{claim}
\begin{proof} Let $t \in T_\alpha$ be arbitrary. As $\alpha\in\Gamma$, the construction of $T_\alpha$ entails that $t = \mathbf{b}^\alpha_x$ for some node $x \in T \restriction C_\alpha$.
Fix such an $x$. By our choice of $\alpha$, fix $\beta \in \nacc(C_\alpha) \cap B$ above $\dom(x)$.
Denote $\beta^-:=\sup(C_\alpha\cap\beta)$.
Since $\beta \in B$, we know that $A_\beta = A \cap (T \restriction \beta)$ is a maximal antichain in $T \restriction \beta$,
and hence there is some $s \in A_\beta$ comparable with $b^\alpha_x(\beta^-)$,
so that by normality of the tree, $Q^{\alpha}_{x,\beta} \neq \emptyset$.
It follows that we chose $b^\alpha_x(\beta)$ to extend some $s \in A_\beta$.
Altogether, 
\[s \stree b^\alpha_x(\beta) \stree  \bigcup\nolimits_{ \beta \in C_\alpha \setminus \dom(x) } b^\alpha_x (\beta)=\mathbf{b}^\alpha_x = t.\qedhere\]
\end{proof}
This completes the proof.
\end{proof}

Now that we have built a $\kappa$-Souslin tree from $\boxtimes^-(\kappa) + \diamondsuit(\kappa)$,
we mention various scenarios in which these hypotheses are known to be valid:

\begin{fact}\label{thm218}
$\boxtimes^-(\kappa)+\diamondsuit(\kappa)$ holds, assuming any of the following:
\begin{enumerate}
\item $\kappa$ is a regular uncountable cardinal that is not weakly compact, and $V=L$ \cite[Corollary~1.10(5)]{paper22};
\item $\kappa = \aleph_1$ and $\diamondsuit(\aleph_1)$ holds \cite[Theorem~3.6]{paper22};
\item $\kappa = \lambda^+$ for $\lambda$ uncountable, and $\square(\lambda^+)+\gch$ holds
\cite[Corollary~4.5]{paper24};
\item $\kappa = \lambda^+$ for $\lambda$ uncountable, and $\square_\lambda+\ch_\lambda$ holds
\cite[Corollary~3.9]{paper22};
\item $\kappa = \lambda^+$ for $\lambda \geq \beth_\omega$,
and $\square(\lambda^+)+\ch_\lambda$ holds \cite[Corollary~4.7]{paper24}.
\end{enumerate}
\end{fact}

It follows from Clause~(3) of the preceding that in the Harrington--Shelah
model \cite[Theorem~A]{MR783595}, $\boxtimes^-(\kappa)+\diamondsuit(\kappa)$ holds for $\kappa=\aleph_2$,
and, in addition, every stationary subset of $E^{\aleph_2}_{\aleph_0}$ reflects,
meaning we are far away from the Gregory scenario  \cite{MR485361}.
Furthermore, $\boxtimes^-(\kappa)+\diamondsuit(\kappa)$ is compatible with the reflection of \emph{all} 
stationary subsets of $\kappa$:

\begin{fact}[{\cite[Theorem~1.12]{lambie2017aronszajn}}]\label{chrisreflectionthm}
Modulo a large cardinal hypothesis,
there is a model of $\zfc+\gch$
in which $\boxtimes^-(\kappa)+\diamondsuit(\kappa)$ holds,
and 	every stationary subset of $\kappa$ reflects, where $\kappa$ can be taken to be $\aleph_{\omega+1}$,
or the first inaccessible cardinal.
\end{fact}

By Proposition~\ref{prop23}, we get a $\kappa$-Souslin tree uniformly in all of these scenarios!

After developing some more machinery in the next few sections,
we shall return in Section~\ref{constructions-section} to construct a $\kappa$-Souslin tree from hypotheses considerably weaker than the ones here.

\section{Interlude: The \texorpdfstring{$\clubsuit$}{club} principle}\label{sectiononclub}

The following principle was introduced by Ostaszewski \cite{ostaszewski1976countably} for the special case $S = \kappa=\aleph_1$.
\begin{definition}\label{def-club}
For a stationary set $S\subseteq\kappa$,
the principle $\clubsuit(S)$ asserts the existence of a sequence
$\langle X_\delta \mid \delta \in S \rangle$ such that:
\begin{enumerate}
\item for every $\delta\in S\cap\acc(\kappa)$, $X_\delta$ is a cofinal subset of $\delta$ with order-type $\cf(\delta)$;
\item for every cofinal subset $X\s \kappa$, the following set is stationary:
$$\{ \delta \in S \mid X_\delta\subseteq X \}.$$
\end{enumerate}
\end{definition}

As $\clubsuit(\omega_1)$ entails $\boxtimes^-(\omega_1)$, it is worth spending some time to present some of the techniques involved in manipulating and improving the former.

\begin{definition} For any two sets of ordinals $A$ and $B$, we say that $A$ is \emph{$B$-separated}
iff for every pair $\alpha<\alpha'$ of ordinals from $A$, there exists $\beta\in B$ with $\alpha<\beta<\alpha'$.
\end{definition}

\begin{lemma}\label{thinning-out} For any two cofinal subsets $A,B$ of some limit nonzero ordinal $\delta$,
there exists a cofinal subset $A'\s A$ such that $A'$ is $B$-separated.
\end{lemma}
\begin{proof} Let $\langle \delta_i\mid i<\cf(\delta)\rangle$ be a strictly increasing sequence of ordinals converging to $\delta$.
Recursively construct a sequence $\langle (\alpha_i,\beta_i)\mid i<\cf(\delta)\rangle$ such that,
for all $i<j<\cf(\delta)$:
\begin{itemize}
\item $\alpha_i\in A$,
\item $\beta_i\in B$, and
\item $\delta_i<\alpha_i<\beta_i<\alpha_{i+1}\le\alpha_j$.
\end{itemize}
Evidently, $A':=\{\alpha_i\mid i<\cf(\delta)\}$ is as sought.
\end{proof}

\begin{cor}\label{equivofclub} Suppose $S\s\kappa$ is stationary.
Then $\clubsuit(S)$
is equivalent to the existence of a sequence
$\langle A_\delta \mid \delta \in S \rangle$ such that, for every cofinal subset $A\s \kappa$, there exists a nonzero $\delta\in S$ such that $A_\delta\subseteq A\cap\delta$ and $\sup(A_\delta)=\delta$.
\end{cor}
\begin{proof} We focus on the nontrivial (that is, backward) implication.
Let $\vec A=\langle A_\delta \mid \delta \in S \rangle$ be as above.
For every $\delta\in S\cap\acc(\kappa)$, if $A_\delta$ happens to be a cofinal subset of $\delta$,
then let $X_\delta$ be a cofinal subset of $A_\delta$ of order-type $\cf(\delta)$;
otherwise, let $X_\delta$ be an arbitrary cofinal subset of $\delta$ of order-type $\cf(\delta)$.
For every $\delta\in S\setminus\acc(\kappa)$, just let $X_\delta := \emptyset$.

To see that $\langle X_\delta\mid\delta\in S\rangle$ is a $\clubsuit(S)$-sequence,
fix an arbitrary cofinal subset $X\s\kappa$ and a club $B\s\kappa$;
we must find $\delta\in S\cap B$ with $X_\delta\s X$. 

By Lemma~\ref{thinning-out}, let $A$ be a cofinal subset of $X$ that is $B$-separated.
By the choice of $\vec A$, let us fix a nonzero ordinal $\delta\in S$ such that $A_\delta\subseteq A\cap\delta$ and $\sup(A_\delta)=\delta$.
In particular $\sup(A\cap\delta) = \delta$, and so by $B$-separation, also $\sup(B\cap\delta) = \delta$.
But $B$ is closed, so that $\delta\in B$. In addition, $X_\delta\s A_\delta\s A\s X$, as sought.
\end{proof}

\begin{lemma}\label{infinite-club-matrix} Suppose that $\kappa^\theta=\kappa$, $S\s\kappa$, and $\clubsuit(S)$ holds.
Then there exists a matrix $\langle X_\delta^\tau\mid \delta\in S,\ \tau\le\theta\rangle$ such that, 
for every sequence $\langle X^\tau\mid \tau\le\theta\rangle$ of cofinal subsets of $\kappa$, 
there exist stationarily many $\delta\in S$, such that, for all $\tau\le\theta$,
$X^\tau_\delta\s X^\tau\cap\delta$ and $\sup(X^\tau_\delta)=\delta$.
\end{lemma}
\begin{proof} Let $\langle A_\delta\mid\delta\in S\rangle$ be a $\clubsuit(S)$-sequence.
Fix an enumeration $\langle f_\alpha\mid\alpha<\kappa\rangle$ of ${}^{\theta+1}\kappa$.
Fix a club $D$ such that for all $\delta\in D$ and $\alpha<\delta$, $\sup(\im(f_\alpha))<\delta$.
For all $\delta\in S$ and $\tau\le\theta$, let $X^\tau_\delta:=\{ f_\alpha(\tau)\mid \alpha\in A_\delta\}$.
To see that $\langle X_\delta^\tau\mid \delta\in S, \tau\le\theta\rangle$ is as sought,
fix an arbitrary sequence $\langle X^\tau\mid \tau\le\theta\rangle$ of cofinal subsets of $\kappa$.
For every $\iota<\kappa$, let $X^\tau(\iota)$ denote the unique element $\gamma\in X^\tau$ such that $\otp(X^\tau\cap\gamma)=\iota$.
Define $g:\kappa\rightarrow\kappa$ by stipulating:
$$g(\iota):=\min\{\alpha<\kappa\mid \forall \tau\le\theta(f_\alpha(\tau)=X^\tau(\iota))\}.$$
Notice that $g$ is injective, so that $\im(g)$ is a cofinal subset of $\kappa$.
Fix a cofinal subset $A$ of $\im(g)$ such that, for any pair $\alpha<\alpha'$ of ordinals from $A$, we have $\alpha<\min_{\tau\le\theta}f_{\alpha'}(\tau)$.
Consider the stationary set:
$$S':=\{ \delta\in S\cap D \cap\acc(\kappa) \mid A_\delta\s A\}.$$
Let $\delta\in S'$ and $\tau\le\theta$. We claim that $X^\tau_\delta\s X^\tau\cap\delta$ and $\sup(X^\tau_\delta)=\delta$. To see this, let $\alpha\in A_\delta$ be arbitrary.
As $A_\delta\s A\s\im(g)$ and $\alpha\in\delta\in D$, this means that $f_\alpha(\tau)\in X^\tau\cap\delta$.
Finally, as $\sup(A_\delta)=\delta$ and $\alpha<f_{\alpha'}(\tau)$ for any pair $\alpha<\alpha'$ of ordinals from $A$, 
we infer that $\sup(X^\tau_\delta)=\delta$.
\end{proof}

It follows from Remark~\ref{rmk_Diamond_H_kappa} and Corollary~\ref{equivofclub} that $\diamondsuit(\kappa) \implies \clubsuit(\kappa)$.
More generally, we have the following.

\begin{fact}[Devlin, {\cite[p.~507]{ostaszewski1976countably}}]\label{clubvsdiamond} For every stationary $S\s\kappa$, the following are equivalent:
\begin{itemize}
\item $\diamondsuit(S)$;
\item $\clubsuit(S)$ and $\kappa^{<\kappa}=\kappa$.
\end{itemize}
\end{fact}

In \cite{MR0491194}, Devlin proved that if $\diamondsuit(S)$ holds for a stationary subset $S\s\kappa$,
then there exists a partition $\langle S_\iota \mid \iota<\kappa \rangle$ of $S$
such that $\diamondsuit(S_\iota)$ holds for every $\iota<\kappa$. 
His proof makes essential use of the consequence $\kappa^{<\kappa}=\kappa$ of $\diamondsuit(S)$.
We now generalize Devlin's theorem and show that its analogue is valid for the weaker principle $\clubsuit$
(even in the absence of $\kappa^{<\kappa}=\kappa$),
along the way, giving a proof that applies to other variants of $\clubsuit$ and $\diamondsuit$.
\begin{thm}\label{partition-clubsuit}
Suppose that $\clubsuit(S)$ holds for a stationary subset $S\s\kappa$. 
Then there exists a partition $\langle S_\iota \mid \iota<\kappa \rangle$ of $S$
such that $\clubsuit(S_\iota)$ holds for every $\iota<\kappa$.
\end{thm}
\begin{proof} Fix a sequence $\vec{X} = \langle X_\delta\mid \delta\in S\rangle$ witnessing  $\clubsuit(S)$.
Fix a bijection $\pi:\kappa\leftrightarrow\kappa\times\kappa$.
Define $h:S\rightarrow\kappa$ by letting $h(\delta):=0$ for $\delta\in S\setminus\acc(\kappa)$, and, for every $\delta\in S\cap\acc(\kappa)$, 
$$h(\delta):=\min\{\iota<\kappa\mid  (\kappa\times\{\iota\})\cap \pi[X_\delta]\neq\emptyset\}.$$
For every $\iota<\kappa$, let $S_\iota := \{ \delta\in S\mid h(\delta) = \iota \}$,
so that $\langle S_\iota \mid \iota<\kappa \rangle$ is a partition of $S$.

Next, for every $\delta\in S\cap\acc(\kappa)$, let $$A_\delta:=\{ \gamma<\delta\mid (\gamma,h(\delta))\in\pi[ X_\delta]\},$$
while for $\delta\in S\setminus\acc(\kappa)$, just let $A_\delta := \emptyset$.

\begin{claim} Let $A\s\kappa$ be cofinal. For every $\iota<\kappa$, 
there exists a nonzero $\delta\in S_\iota$ such that $A_\delta\s A$ and $\sup(A_\delta)=\delta$.
\end{claim}
\begin{proof} Let $\iota<\kappa$ be arbitrary. Put $X:=\pi^{-1}[A\times\{\iota\}]$.
As $\pi$ is bijective, $|X|=|A|=\kappa$, so that $X$ is cofinal in $\kappa$.
Consider the club $B:=\{\beta<\kappa\mid \pi[\beta] = \beta\times\beta\}$.
By Lemma~\ref{thinning-out}, fix a cofinal $X'\s X$ that is $B$-separated.
Now, by the choice of $\vec X$, the following set is stationary:
$$G(X'):=\{ \delta\in S\mid X_\delta\s X'\}.$$

Fix $\delta\in G(X')\cap\acc(\kappa)$ above $\iota$.
As $X_\delta\s X'$, we infer $\pi[X_\delta]\s\pi[X']\s \pi[X]=A\times\{\iota\}$, so that $h(\delta)=\iota$.
It thus follows that $\delta\in S_\iota$ and $A_\delta\s A$.
Finally, to see that $\sup(A_\delta)=\delta$, let $\alpha<\delta$ be arbitrary,
and we shall find $\gamma\in A_\delta$ above $\alpha$.
As $\sup(X_\delta)=\delta$, we may assume that $\alpha\in X_\delta\setminus\iota$.
Let $\alpha':=\min(X_\delta\setminus(\alpha+1))$.
As $X_\delta\s X'$ and the latter is $B$-separated, let us also fix $\beta\in B$ with $\alpha<\beta<\alpha'$.
Since $\pi[X_\delta]\s A\times\{\iota\}$, we can fix $\gamma$ such that $\pi(\alpha')=(\gamma,\iota)$.
Since $\pi[\beta]=\beta\times\beta$ while $\iota\leq\alpha<\beta<\alpha'$, it follows that $\gamma\ge\beta$.
As $\sup(X_\delta)=\delta$ and $X_\delta\s X'$, we infer that $\sup(X'\cap\delta)=\delta$, 
so that by $B$-separation of $X'$, we also obtain $\sup(B\cap\delta)=\delta$.
But $B$ is closed, so that $\delta\in B$, meaning that $\pi[\delta]=\delta\times\delta$.
As $\alpha'<\delta$, it follows that $\gamma<\delta$.
Altogether,
$\gamma$ is an element of $A_\delta$ above $\alpha$.
\end{proof}

As $A_\delta\s A$ for all $\delta \in S$,
it now follows from Corollary~\ref{equivofclub}
that $\clubsuit(S_\iota)$ holds for all $\iota<\kappa$.
\end{proof}

When replacing sets of ordinals by sets of sets of ordinals,
it is natural to replace the $\sup$ measure by a $\sup$-over-$\min$  measure which we call $\mup$:

\begin{definition}\label{defmup} For every $\mathcal A\s\mathcal P(\kappa)$,
let $\mup(\mathcal A):=\sup\{ \min(a)\mid a\in\mathcal A, a\neq\emptyset\}$.
\end{definition}

A minor variation of the proof of Theorem~\ref{partition-clubsuit} establishes yet 
another equivalence, which will be utilized in deriving Theorem~\ref{get-sigma-finite} below.

\begin{lemma}\label{club-equivalents}
For every stationary $S\subseteq\kappa$, $\clubsuit(S)$ holds iff
there exists a sequence $\langle \mathcal X_\delta \mid \delta \in S \rangle$ such that:
\begin{enumerate}
\item for every $\delta\in S\cap\acc(\kappa)$,  $\mathcal X_\delta\s [\delta]^{<\omega}$ with $\mup(\mathcal X_\delta)=\delta$;
\item for every $\mathcal X\s[\kappa]^{<\omega}$ with $\mup(\mathcal X)=\kappa$, the following set is stationary:
$$\{ \delta \in S \mid \mathcal X_\delta\subseteq \mathcal X \}.$$
\end{enumerate}
\end{lemma}
\begin{proof} We focus on the nontrivial (that is, forward) implication.
Fix a sequence $\vec{A} = \langle A_\delta\mid \delta\in S\rangle$ witnessing  $\clubsuit(S)$.
Fix a bijection $\pi:\kappa\leftrightarrow[\kappa]^{<\omega}$.
For every $\delta\in S$, let 
$$\mathcal X_\delta:=\begin{cases}
\pi[A_\delta],&\text{if } \delta\in\acc(\kappa)\ \&\ \pi[A_\delta]\s[\delta]^{<\omega}\ \&\ \mup(\pi[A_\delta])=\delta;\\
[\delta]^{<\omega},&\text{otherwise}.
\end{cases}$$
To see that $\langle \mathcal X_\delta\mid \delta\in S\rangle$ satisfies Clause~(2) of the lemma,
fix an arbitrary $\mathcal X\s[\kappa]^{<\omega}$ with $\mup(\mathcal X)=\kappa$.
By thinning $\mathcal X$ out, we may assume that $\emptyset\notin\mathcal X$ and that $x\mapsto\min(x)$ is injective over $\mathcal X$.
Now, let $A:=\pi^{-1}[\mathcal X]$, so that $|A|=\kappa$.
Consider the following set:
$$B:=\{\beta<\kappa\mid \pi[\beta]=[\beta]^{<\omega}\ \&\ \forall x\in\mathcal X(\min(x)<\beta\implies\max(x)<\beta)\}.$$
It is not hard to see that for $p:=\{\pi,\mathcal X\}$,
our set $B$ covers the set $B(p)$ of Fact~\ref{cofinally-many-esms}, so that $B$ is cofinal in $\kappa$.
By Lemma~\ref{thinning-out}, fix a cofinal $A'\s A$ that is $B$-separated.
Now, by the choice of $\vec A$, 
there are stationarily many $\delta\in S\cap\acc(\kappa)$ for which $A_\delta\s A'$.
Fix such a $\delta$, and we shall show that $\mathcal X_\delta\s\mathcal X$.

As $A_\delta\s A'$, we infer $\pi[A_\delta]\s\pi[A']\s \pi[A]=\mathcal X$.
By definition of $\mathcal X_\delta$, then, it suffices to prove that $\pi[A_\delta]\s[\delta]^{<\omega}$ and $\mup(\pi[A_\delta])=\delta$.
To prove the former:
As $\sup(A_\delta)=\delta$ and $A_\delta\s A'$, we infer that $\sup(A'\cap\delta)=\delta$, 
so that by $B$-separation of $A'$, we also obtain $\sup(B\cap\delta)=\delta$.
But $B$ is closed, so that $\delta\in B$, 
and as $A_\delta\s\delta$ we infer that $\pi[A_\delta]\s\pi[\delta]=[\delta]^{<\omega}$.
To see that $\mup(\pi[A_\delta])=\delta$,
let $\alpha<\delta$ be arbitrary,
and we shall find $\alpha'\in A_\delta$ with $\min(\pi(\alpha'))>\alpha$.
As $\sup(A_\delta)=\delta$, we may assume that $\alpha\in A_\delta$.
Let $\alpha':=\min(A_\delta\setminus(\alpha+1))$.
As $A_\delta\s A'$ and the latter is $B$-separated, let us also fix $\beta\in B$ with $\alpha<\beta<\alpha'$.
Since $\pi[\beta]=[\beta]^{<\omega}$, we know that $\max(\pi(\alpha'))\ge\beta$.
Since $\beta\in B$ and $\pi(\alpha') \in \pi[A_\delta] \s\mathcal{X}$,
we infer that 
also $\min(\pi(\alpha'))\geq\beta>\alpha$, 
as sought.
\end{proof}

\section{A generalization of \texorpdfstring{$\boxtimes^-(\kappa)$}{xbox}}
In this section and the next one, 
we shall present generalizations of the concepts that arose in Section~\ref{rightway}.
Here, we present a principle 
$\p_\xi^-(\kappa, \mu, \mathcal R, \theta, \mathcal S,  \nu,\sigma)$ 
that generalizes the ladder system principle $\boxtimes^-(\kappa)$.
Then, in the next section, we shall present $\p_\xi^\bullet(\kappa, \mu, \mathcal R, \theta, \mathcal S,  \nu)$, 
which serves as a generalization and weakening of the conjunction  ${\boxtimes^-(\kappa)}+{\diamondsuit(\kappa)}$.

\subsection{Ladder systems}
We assume the reader is comfortable with Conventions~\ref{conventionsforpaper} from Page~\pageref{conventionsforpaper}.
\begin{definition}[\cite{paper29}] \label{def-K-kappa} Let $\mathcal K(\kappa):=\{ x\in\mathcal P(\kappa)\mid x\neq\emptyset\ \&\ \acc^+(x)\s x\ \&\allowbreak\ \sup(x) \notin x\}$
denote the collection of all nonempty $x \subseteq \kappa$ such that $x$ is a club subset of $\sup(x)$.

For each $C\in\mathcal K(\kappa)$, denote  $\alpha_C:=\sup(C)$
\end{definition}

For a binary relation $\mathcal R$  over $\mathcal K(\kappa)$, and a nonempty collection $\mathcal S$ of stationary subsets of $\kappa$,
we shall define a principle $\p^-_\xi(\kappa, \mu, \mathcal R, \theta, \mathcal S,  \nu,\sigma)$  in two stages.
In the first stage, we focus on the first four parameters.

\begin{defn}\label{defn-p-sequence-first-stage} 
We say that $\langle\mathcal C_\alpha\mid\alpha<\kappa\rangle$ is a $\p^-_\xi(\kappa, \mu, \mathcal R, \ldots)$-sequence iff, for every $\alpha\in\acc(\kappa)$, all of the following hold:
\begin{itemize}
\item $\mathcal C_\alpha\s\{ C\in \mathcal K(\kappa)\mid \otp(C)\le\xi\ \&\ \alpha_C=\alpha\}$;
\item $0<|\mathcal C_\alpha|<\mu$;
\item for all $C \in \mathcal C_\alpha$ and $\bar\alpha \in \acc(C)$, there exists $D \in \mathcal C_{\bar\alpha}$ with $D \mathrel{\mathcal R} C$.
\end{itemize}
\end{defn}
\begin{conv}\label{conventionxi}
If we omit the subscript $\xi$, then we mean that $\xi:=\kappa$.
\end{conv}
\begin{conv}\label{default-outside-acc} 
We shall always assume
that $\mathcal C_0:=\{\emptyset\}$ and $\mathcal C_{\alpha+1}:=\{\{\alpha\}\}$ for all $\alpha<\kappa$.
Likewise, whenever we construct a  $\p^-_\xi(\kappa, \mu, \mathcal R, \ldots)$-sequence $\langle \mathcal D_\alpha\mid\alpha<\kappa\rangle$,
we shall never bother to define $\mathcal D_0$ and $\mathcal D_{\alpha+1}$ for $\alpha<\kappa$.
\end{conv}

\begin{example}\label{example53}
The binary relations over $\mathcal K(\kappa)$ that fit as the parameter $\mathcal R$
should be understood as \emph{coherence} relations.
The basic example is the \emph{end-extension} relation, $\sq$,
where, for $C,D\in\mathcal K(\kappa)$, we define $C \sqsubseteq D$ iff $C = D \cap \alpha_C$.
More nuanced binary relations over $\mathcal K(\kappa)$ are obtained by modifying the $\sq$ relation as follows:
\begin{itemize}
\item We define $C \sq^* D$ iff there exists $\gamma < \alpha_C$ such that ${C\setminus\gamma}\sq{D\setminus\gamma}$;
\item For $\mathcal R \in \{ {\sq}, {\sq^*} \}$,
we define
$C \mathrel{_{\chi}{\mathcal R}}D$ iff ((${C}\mathrel{\mathcal{R}}{D}$) or ($\cf(\alpha_C)<\chi$));
\item For $\mathcal{R} \in \{ {\sq}, {\sq^*} \}$,
we define $C \mathrel{\mathcal{R}_\chi} D$ iff ((${C} \mathrel{\mathcal{R}} {D}$) or ($\otp(D)<\chi$ and $\nacc(D)$ consists only of successor ordinals));
\item For any binary relation $\mathcal R$ over $\mathcal K(\kappa)$ and any class $\Omega\s\on$, we define
$C\mathrel{^{\Omega}{\mathcal R}}D$ iff ((${C}\mathrel{\mathcal{R}}{D}$) and ($\alpha_C\notin\Omega$)).
\end{itemize}
\end{example}

$\p^-(\kappa,2,{\sq},\dots)$-
and $\p^-_\lambda(\lambda^+,2,{\sqleft{\lambda}},\dots)$-sequences may be constructed in $\zfc$,
but there are stronger variations.
For instance, Jensen's axiom $\square_\lambda$ (resp.~$\square_\lambda^*$) is equivalent to the existence of a $\p^-_\lambda(\lambda^+,2,{\sq},\ldots)$-sequence
(resp.~$\p^-_\lambda(\lambda^+,\lambda^+,{\sq},\ldots)$-sequence).
More examples in this spirit may be found in \cite{paper22}.

\begin{conv}\label{convention-mu-infty}
We may put ``$\infty$'' in place of $\mu$ in Definition~\ref{defn-p-sequence-first-stage},
in which case we mean that $|\mathcal C_\alpha|\le|\alpha|$ for every nonzero $\alpha<\kappa$.
\end{conv}

\begin{remark}\label{remark-only-two-relations}
The relation $\sq$ coincides with $\sqleftboth{\Omega}{\chi}$ for $(\Omega,\chi):=(\emptyset,0)$,
as well as  $(\Omega,\chi):=(\emptyset,\omega)$.
Note that if $\kappa=\lambda^+$ is a successor cardinal,
then for every relation $\mathcal R \in \{ {\sqleftboth{\Omega}{\chi}}, {\sqstarleftboth{\Omega}{\chi}} \}$
with $\Omega\cap\lambda=\emptyset$,
any $\p_\xi^-(\kappa,\kappa, \mathcal R, \ldots)$-sequence
may be improved into a $\p_\xi^-(\kappa, \infty, \mathcal R, \ldots)$-sequence
while preserving its further crucial features. The simple proof may be found in the construction before Claim~3.4.6 of~\cite{paper26}.
\end{remark}

Let us stress that a study of coherence relations weaker than $\sq$ is necessary. 
For instance, unlike coherent square sequences that are refuted by large cardinals,
$\sq_\chi$-coherent square sequences provide an effective means to obtain optimal incompactness results above large cardinals (cf.~\cite{paper28}), 
as well as $\kappa$-Souslin trees in a model in which \emph{all $\kappa$-Aronszajn trees are nonspecial} (cf.~\cite[Corollary~1.20, Examples 1.26 and~1.27]{paper22}).

\medskip

In \cite{paper29}, we introduced the following definition as a tool for manipulating and improving ladder systems.

\begin{defn}[{\cite[Definition~1.8]{paper29}}]\label{def-pp}
A function $\Phi:\mathcal K(\kappa)\rightarrow\mathcal K(\kappa)$ is a \emph{postprocessing function}
iff for every $C\in\mathcal K(\kappa)$:
\begin{itemize}
\item $\sup(\Phi(C)) = \sup(C)$;
\item $\acc(\Phi(C)) \s \acc(C)$;
\item $\Phi(C)\cap\bar\alpha=\Phi(C\cap\bar\alpha)$ for every $\bar\alpha\in\acc(\Phi(C))$.
\end{itemize}

If, in addition, $\min(\Phi(C))=\min(C)$ (resp.~$\acc(\Phi(C))=\acc(C)$) for every $C\in\mathcal K(\kappa)$,
then $\Phi$ is said to be \emph{$\min$-preserving} (resp.~\emph{$\acc$-preserving}).
\end{defn}

The point is that whenever $\langle\mathcal C_\alpha\mid\alpha<\kappa\rangle$ is a $\p^-_\xi(\kappa, \mu, {\sqleftboth{\Omega}{\chi}}, \ldots)$-sequence, 
then for every postprocessing function $\Phi$, by setting $\mathcal D_\alpha:=\{ \Phi(C)\mid C\in\mathcal C_\alpha\}$,
we get that $\langle\mathcal D_\alpha\mid\alpha<\kappa\rangle$ is yet again a $\p^-_\xi(\kappa, \mu, {{\sqleftboth{\Omega}{\chi}}}, \ldots)$-sequence.\footnote{In fact,
postprocessing functions can be viewed as actions on square-like sequences, see Notation~2.15, Lemma~2.16 and Lemma~4.17 of \cite{paper29}. Also, please keep Convention~\ref{default-outside-acc} in mind.}
We now present a sufficient condition for preserving all binary relations from Example~\ref{example53}.

\begin{lemma}\label{pp-from-micro}
Suppose that $\mathfrak x=\langle x_{\gamma,\beta}\mid \gamma<\beta<\kappa\rangle$ is a triangular array of nonempty finite sets such that, for all $\gamma<\beta<\kappa$,
$x_{\gamma,\beta}\s(\gamma,\beta]$, and  if $\beta$ is a successor ordinal, then $x_{\gamma,\beta}=\{\beta\}$.
Define a corresponding function $\Phi_{\mathfrak x}:\mathcal K(\kappa)\rightarrow\mathcal K(\kappa)$ via:
$$\Phi_{\mathfrak x}(C):=\{\min(C)\}\cup\bigcup\bigl\{ x_{\gamma,\beta}\bigm| \gamma\in C, \beta=\min(C\setminus(\gamma+1))\bigr\}\cup\acc(C).$$

Then $\Phi_{\mathfrak x}$ is a $\min$-preserving, $\acc$-preserving postprocessing function. Furthermore,
for every $\p_\xi^-(\kappa, \mu, \mathcal R, \ldots)$-sequence with $\mathcal R$ taken from Example~\ref{example53},
if we set $\mathcal D_\alpha:=\{ \Phi_{\mathfrak x}(C)\mid C\in\mathcal C_\alpha\}$,
then $\langle\mathcal D_\alpha\mid\alpha<\kappa\rangle$ is yet again a $\p^-_\xi(\kappa, \mu, \mathcal R, \ldots)$-sequence.
\end{lemma}
\begin{proof} Left to the reader (cf.~\cite[Lemma~2.8]{paper29}).
\end{proof}

\subsection{Ladder systems with hitting features}
We now arrive at the second stage of the definition of the proxy principle.

\begin{definition}[Proxy principle]\label{pminus}
$\p_\xi^-(\kappa, \mu, \mathcal R, \theta, \mathcal S,  \nu,\sigma)$
asserts the existence 
of a $\p_\xi^-(\kappa, \mu, \mathcal R, \ldots)$-sequence $\cvec{C} = \langle \mathcal C_\alpha \mid \alpha < \kappa \rangle$
satisfying the following \emph{hitting} feature.

For every sequence $\langle B_i \mid i < \theta \rangle$ of cofinal subsets of $\kappa$,
and every $S \in \mathcal S$, there exist stationarily many $\alpha \in S$ such that: 
\begin{enumerate}
\item $| \mathcal C_\alpha| < \nu$; and
\item for all $C \in \mathcal C_\alpha$ and $i < \min\{\alpha, \theta\}$,
\begin{equation}\label{hittingeqn}\tag{$\star$}
\sup\{ \gamma\in C \mid \suc_\sigma (C \setminus \gamma) \subseteq B_i \} = \alpha.
\end{equation}
\end{enumerate}
\end{definition}

\begin{remark}\label{monotonicity}
The reader can verify that the proxy principle satisfies monotonicity properties with respect to most of its parameters:
Any sequence witnessing
$\p_\xi^-(\kappa, \mu, \mathcal R, \theta, \mathcal S,  \nu,\sigma)$
remains a witness to the principle if any of $\xi$, $\mu$, or $\nu$ are increased;
if $\theta$ or $\sigma$ are decreased; if $\mathcal R$ is weakened;
if $\mathcal S$ is shrunk;
or if any element of $\mathcal S$ is expanded.
Furthermore, increasing $\chi$ or shrinking $\Omega$
both weaken the relations $\sqleftboth{\Omega}{\chi}$, $\sqstarleftboth{\Omega}{\chi}$, $\sqleftup{\Omega}_{\chi}$,
and $^{\Omega}{\sq}^*_\chi$.
Note also that 
$\p_\xi^-(\kappa,\mu, {^{\Omega}\mathcal{R}_{\chi}},\ldots)$
entails 
$\p_\xi^-(\kappa,\mu, {^\Omega_\chi{\mathcal{R}}},\ldots)$.
These monotonicity properties will be used freely without explanation as the need arises.
\end{remark}
\begin{remark}
In the special case $\sigma=1$, Equation~\eqref{hittingeqn} above is 
equivalent to the assertion that 
$\sup(\nacc(C) \cap B_i) = \alpha$.
Thus, applying an argument just as in the proof of Corollary~\ref{equivofclub}, 
we infer that the principle $\boxtimes^-(\kappa)$ of Definition~\ref{xboxdefinitionfrom22}
is equivalent to the instance $\p^-(\kappa,2,{\sq},1,\{\kappa\},2,1)$.
\end{remark}

\begin{conv}\label{conv<theta}
In Definition~\ref{pminus},
by putting ``${<}\theta$'' in place of $\theta$, we mean that $\cvec{C}$ 
simultaneously witnesses $\p_\xi^-(\kappa, \mu, \mathcal R, \theta', \mathcal S,  \nu,\sigma)$
for all $\theta'<\theta$.
\end{conv}

\begin{conv}\label{convention-sigma<infty}\label{conv<sigma}
In Definition~\ref{pminus}, by putting ``${<}\sigma$'' in place of $\sigma$, we mean that $\cvec{C}$ 
simultaneously witnesses $\p_\xi^-(\kappa, \mu, \mathcal R, \theta, \mathcal S,  \nu,\sigma')$
for all $\sigma'<\sigma$.

We may also put ``${<}\infty$'' in place of $\sigma$,
in which case we mean to replace the assertion of Equation~\eqref{hittingeqn}  of Definition~\ref{pminus} by:
$$\forall\sigma<\otp(C)~\sup\{ \gamma \in C \mid \suc_\sigma (C \setminus \gamma) \subseteq B_i \} = \alpha.$$
\end{conv}

\begin{thm}\label{thm416} Suppose  $\mathcal R$ is taken from Example~\ref{example53}. 
Then all of the following are equivalent:
\begin{itemize}
\item[(i)] $\p_\xi^-(\kappa,\mu,\mathcal R, \theta, \mathcal S, \nu, {<}\omega)$;
\item[(ii)] $\p_\xi^-(\kappa,\mu,\mathcal R, \theta, \mathcal S, \nu, 2)$;
\item[(iii)] There exists a $\p_\xi^-(\kappa,\mu,\mathcal R, \ldots)$-sequence $\langle \mathcal C_\alpha\mid\alpha<\kappa\rangle$
satisfying the following.
For every $S \in \mathcal S$ 
and 
every sequence $\langle B_i \mid i < \theta \rangle$ of cofinal subsets of $\kappa$,
there exist stationarily many $\alpha \in S$ with $| \mathcal C_\alpha| < \nu$ 
such that, for all $C \in \mathcal C_\alpha$ and $i < \min\{\alpha, \theta\}$,
$$\sup\{\delta\in B_i\cap\alpha\mid \min(C\setminus(\delta+1))\in B_i\}=\alpha;$$
\item[(iv)] There exists a $\p_\xi^-(\kappa, \mu, \mathcal R, \ldots)$-sequence $\langle \mathcal C_\alpha \mid \alpha < \kappa \rangle$
satisfying the following.
For every $S \in \mathcal S$, 
every sequence $\langle B_i \mid i < \theta \rangle$ of cofinal subsets of $\kappa$,
and every every club $D\s\kappa$, 
there exists $\alpha \in S\cap\acc(\kappa)$ with $| \mathcal C_\alpha| < \nu$
such that, for all $C \in \mathcal C_\alpha$ and $i < \min\{\alpha, \theta\}$,
$$\sup\{\delta\in D\cap\alpha\mid \min(C\setminus(\delta+1))\in B_i\}=\alpha;$$
\item[(v)] There exists a $\p_\xi^-(\kappa,\mu,\mathcal R, \ldots)$-sequence $\langle \mathcal C_\alpha\mid\alpha<\kappa\rangle$
satisfying the following.
For every $S \in \mathcal S$ and
every sequence $\langle \mathcal B_i \mid i < \theta \rangle$ with $\mathcal B_i\s[\kappa]^{<\omega}$ and $\mup(\mathcal B_i)=\kappa$ for all $i<\theta$,\footnote{For the definition of $\mup$, see Definition~\ref{defmup}.}
there exist stationarily many $\alpha \in S$ with $| \mathcal C_\alpha| < \nu$
such that, for all $C \in \mathcal C_\alpha$, $i < \min\{\alpha, \theta\}$ and $\epsilon<\alpha$,
there exist $\gamma,\beta$ with $\epsilon\le\gamma<\beta<\alpha$
for which $C\cap(\gamma,\beta)$ is in $\mathcal B_i$;
\item[(vi)] There exists a $\p_\xi^-(\kappa,\mu,\mathcal R, \ldots)$-sequence $\langle \mathcal C_\alpha\mid\alpha<\kappa\rangle$
satisfying the following.
For every $S \in \mathcal S$, every sequence $\langle B_i \mid i < \theta \rangle$ of cofinal subsets of $\kappa$, and every $n<\omega$, there exist stationarily many $\alpha \in S$ 
with $| \mathcal C_\alpha| < \nu$
such that, for all $C \in \mathcal C_\alpha$ and $i < \min\{\alpha, \theta\}$,
$$\sup\{ \gamma\in C \mid \suc_n (C \setminus \gamma)= \suc_n (B_i \setminus \gamma) \} = \alpha.$$
\end{itemize}
\end{thm}
\begin{proof} It is clear that any sequence witnessing (i) will witness (ii),
any sequence witnessing (ii) will witness (iii), and any sequence witnessing (vi) will witness (i).
Next, to see that any sequence witnessing (iii)
will witness (iv), suppose that we are given a club $D$ and a sequence $\langle B_i\mid i<\theta\rangle$ of cofinal subsets of $\kappa$.
For each $i<\theta$, by Lemma~\ref{thinning-out}, fix a cofinal subset $B'_i$ of $B_i$ that is $D$-separated.
Then, for every $\alpha\in\acc(\kappa)$, every club $C$ in $\alpha$, and every $i<\theta$, if 
$\sup\{\delta\in B_i'\cap\alpha\mid \min(C\setminus(\delta+1))\in B_i'\}=\alpha$, then
$\sup\{\delta\in D\cap\alpha\mid \min(C\setminus(\delta+1))\in B_i\}=\alpha$.

(iv) $\implies$ (v): Fix an injection $\psi:[\kappa]^{<\omega}\rightarrow\acc(\kappa)$ such that,
for all $x\in[\kappa]^{<\omega}$, $\psi(x)>\sup(x)$. Denote $\Omega:=\im(\psi)$,
and let $\pi:\Omega\rightarrow[\kappa]^{<\omega}$ denote the inverse of $\psi$.

Fix a triangular array $\mathfrak x=\langle x_{\gamma,\beta}\mid \gamma<\beta<\kappa\rangle$,
such that for all $\gamma<\beta<\kappa$:
\begin{itemize}
\item if $\beta\in\Omega$, then $x_{\gamma,\beta}=(\pi(\beta)\setminus(\gamma+1))\cup\{\beta\}$;
\item if $\beta\notin\Omega$, then $x_{\gamma,\beta}=\{\beta\}$.
\end{itemize} 

Evidently, $x_{\gamma,\beta}\s(\gamma,\beta]$ for all $\gamma<\beta<\kappa$.
Thus, consider the corresponding postprocessing function $\Phi_{\mathfrak x}$ from  Lemma~\ref{pp-from-micro}.
Let $\cvec C=\langle \mathcal C_\alpha\mid \alpha<\kappa\rangle$ be a sequence as in (iv).
For every $\alpha\in\acc(\kappa)$, let $\mathcal C^\bullet_\alpha:=\{\Phi_{\mathfrak x}(C)\mid C\in\mathcal C_\alpha\}$,
so that $\cvec{C}^\bullet := \langle \mathcal C^\bullet_\alpha\mid\alpha<\kappa\rangle$ is a $\p_\xi^-(\kappa,\mu,\mathcal R, \ldots)$-sequence.
The following claim shows that $\cvec{C}^\bullet$ witnesses~(v):

\begin{claim}\label{claim4161} Suppose $S\in\mathcal S$, $\langle \mathcal B_i \mid i < \theta \rangle$ is a sequence with $\mathcal B_i\s[\kappa]^{<\omega}$ and $\mup(\mathcal B_i)=\kappa$ for all $i<\theta$,
and $D\s\kappa$ is a club.
Then there exists $\alpha \in S \cap D$ such that:
\begin{enumerate}
\item $| \mathcal C^\bullet_\alpha| < \nu$; and
\item for all $C^\bullet \in \mathcal C^\bullet_\alpha$, $i < \min\{\alpha, \theta\}$ and $\epsilon<\alpha$,
there exist $\gamma,\beta$ with $\epsilon\le\gamma<\beta<\alpha$
for which $C^\bullet\cap(\gamma,\beta)$ is in $\mathcal B_i$.
\end{enumerate}
\end{claim}
\begin{proof} 
For all $i<\theta$ and $\gamma<\kappa$, fix $b_i^\gamma\in\mathcal B_i$ with $\min(b_i^\gamma)>\gamma$.
Fix a club $D'\s D$ such that, for all $\delta\in D'$:
\begin{itemize}
\item $\psi``[\delta]^{<\omega}\s\delta$;
\item for all $\gamma<\delta$ and $i<\min\{\theta,\delta\}$, $b_i^\gamma\s\delta$.
\end{itemize}
For each $i<\theta$, let $$B_i:=\{\psi(\{\gamma\}\cup b_i^\gamma)\mid  \gamma\in D'\}.$$

Now, by the choice of $\cvec C$,
fix $\alpha\in S\cap\acc(\kappa)$ such that:
\begin{itemize}
\item $|\mathcal C_\alpha|<\nu$, and 
\item for all $i < \min\{\alpha, \theta\}$ and $C\in\mathcal C_\alpha$,
$$\sup\{\delta\in D'\cap\alpha\mid \min(C\setminus(\delta+1))\in B_i\}=\alpha.$$
\end{itemize}
In particular, $\alpha\in\acc^+(D') \s D' \s D$.
Furthermore, $|\mathcal C^\bullet_\alpha|\le|\mathcal C_\alpha|<\nu$.
Finally, let $i < \min\{\alpha, \theta\}$ and $C^\bullet\in\mathcal C^\bullet_\alpha$ be arbitrary. Fix $C\in\mathcal C_\alpha$ such that $C^\bullet=\Phi_{\mathfrak x}(C)$.
Let 
$$\Delta_i := \{\delta\in D'\cap\alpha \mid \min(C\setminus(\delta+1))\in B_i\ \&\ \max\{i,\min(C)\}<\delta\}.$$
By the choice of $\alpha$, $\Delta_i$ is cofinal in $\alpha$.
Let $\delta\in\Delta_i$ be arbitrary. Set $\beta:=\min(C\setminus(\delta+1))$ and $\beta^-:=\sup(C\cap\beta)$.
As $\beta\in B_i$, let us fix $\gamma\in D'$ such that $\pi(\beta)=\{\gamma\}\cup b_i^\gamma$.
Evidently, $x_{\gamma,\beta}=b_i^\gamma\cup\{\beta\}$.
As $\psi``[\delta]^{<\omega}\s\delta$ and $\beta>\delta$, we infer that $\pi(\beta)\notin[\delta]^{<\omega}$, so that $b_i^\gamma\nsubseteq\delta$.
As $\delta\in D'$ and $i<\min\{\theta,\delta\}$, this means that $\gamma\ge\delta$.
So, $\min(\pi(\beta))=\gamma\ge\delta\ge\beta^-$, and hence $x_{\beta^-,\beta}\cap(\gamma,\beta)=b^\gamma_i$.
It thus follows from the definition of $\Phi_{\mathfrak x}$ that
$C^\bullet\cap(\gamma,\beta) = (x_{\beta^-,\beta}\cap(\gamma,\beta)) \in\mathcal B_i$, as sought.
\end{proof}

Finally, to see that any sequence witnessing (v) will witness (vi), 
let $\langle B_i\mid i<\theta\rangle$ be an arbitrary sequence of cofinal subsets of $\kappa$, 
and let $n<\omega$ be arbitrary.
For each $i<\theta$, set $\mathcal B_i:=\{ \suc_{n+1}(B_i\setminus\epsilon)\mid \epsilon<\kappa\}$. 
Then, for every $\alpha\in\acc(\kappa)$, every club $C$ in $\alpha$, every $i<\theta$, 
and every $\gamma,\beta$ such that $C\cap(\gamma,\beta)$ is in $\mathcal B_i$,
setting $\gamma':=\min(C\cap(\gamma,\beta))$, we obtain $\suc_n (C \setminus \gamma')=\suc_n (B_i \setminus \gamma')$.
\end{proof}

It is clear that a sequence witnessing any one of the clauses of Theorem~\ref{thm416}
will also witness $\p_\xi^-(\kappa,\mu,\mathcal R, \theta, \mathcal S, \nu, 1)$.
Clause~(1) of the following Theorem shows that in the presence of $\clubsuit(\kappa)$,
the principles are equivalent.

\begin{thm}\label{get-sigma-finite}
Suppose $\clubsuit(\kappa)$ and $\p_\xi^-(\kappa,\mu,\mathcal R, \theta, \mathcal S, \nu, 1)$ both hold,
with  $\mathcal R$ taken from Example~\ref{example53}.
Then:
\begin{enumerate}
\item $\p_\xi^-(\kappa,\mu,\mathcal R, \theta, \mathcal S, \nu, {<}\omega)$ holds;
\item If $0<\theta<\omega$, then $\p_\xi^-(\kappa,\mu,\mathcal R, {<}\omega, \mathcal S, \nu, {<}\omega)$ holds.
\end{enumerate}
\end{thm}
\begin{proof}
(1) Denote $\Omega:=\acc(\kappa)$. 
By $\clubsuit(\kappa)$ and Lemma~\ref{club-equivalents}, we may fix a sequence $\cvec{X} = \langle \mathcal X_\beta \mid \beta\in\Omega\rangle$ such that:
\begin{enumerate}[(a)]
\item for every $\beta\in \Omega$,  $\mathcal X_\beta\s [\beta]^{<\omega}$ with $\mup(\mathcal X_\beta)=\beta$;
\item for every $\mathcal X\s[\kappa]^{<\omega}$ with $\mup(\mathcal X)=\kappa$, 
there are stationarily many $\beta\in\Omega$ for which $\mathcal X_\beta\subseteq \mathcal X$.
\end{enumerate}

Fix a triangular array $\mathfrak x=\langle x_{\gamma,\beta}\mid \gamma<\beta<\kappa\rangle$ 
such that for all $\gamma<\beta<\kappa$:
\begin{itemize}
\item if $\beta\in\Omega$, then $x_{\gamma,\beta}\in\mathcal X_\beta\setminus\{\emptyset\}$ with $\min(x_{\gamma,\beta})>\gamma$;
\item if $\beta\notin\Omega$, then $x_{\gamma,\beta}=\{\beta\}$.
\end{itemize} 

Now, consider the corresponding postprocessing function $\Phi_{\mathfrak x}$ from  Lemma~\ref{pp-from-micro}.
Let $\cvec{C} = \langle \mathcal C_\alpha\mid \alpha<\kappa\rangle$ be a sequence witnessing $\p_\xi^-(\kappa,\mu,\mathcal R, \theta, \mathcal S, \nu, 1)$.
For every $\alpha\in\acc(\kappa)$, let $\mathcal D_\alpha:=\{\Phi_{\mathfrak x}(C)\mid C\in\mathcal C_\alpha\}$,
so that $\cvec{D} := \langle \mathcal D_\alpha\mid\alpha<\kappa\rangle$ is a $\p_\xi^-(\kappa,\mu,\mathcal R, \ldots)$-sequence.
To see that $\p_\xi^-(\kappa,\mu,\mathcal R, \theta, \mathcal S, \nu, {<}\omega)$ holds,
we shall verify that $\cvec{D}$ witnesses Clause~(v) of Theorem~\ref{thm416}:

\begin{claim}\label{claim41411} Suppose $S\in\mathcal S$ and $\langle \mathcal B_i \mid i < \theta \rangle$ is a sequence with $\mathcal B_i\s[\kappa]^{<\omega}$ and $\mup(\mathcal B_i)=\kappa$ for all $i<\theta$.
Then there exist stationarily many $\alpha \in S$ such that:
\begin{enumerate}
\item $| \mathcal D_\alpha| < \nu$; and
\item for all $D \in \mathcal D_\alpha$, $i < \min\{\alpha, \theta\}$ and $\epsilon<\alpha$,
there exist $\gamma,\beta$ with $\epsilon\le\gamma<\beta<\alpha$
for which $D\cap(\gamma,\beta)$ is in $\mathcal B_i$.
\end{enumerate}
\end{claim}
\begin{proof} 
For every $i<\theta$, we know from our choice of $\cvec{X}$ that $B_i:=\{\beta\in\Omega\mid \mathcal X_\beta\subseteq \mathcal B_i\}$ is stationary.
Thus, by our choice of $\cvec{C}$, the following set is stationary:
\[
S' := \{ \alpha\in S \mid
|\mathcal C_\alpha|<\nu \ \&\ \forall i < \min\{\alpha, \theta\} \forall C\in\mathcal C_\alpha
[\sup(\nacc(C)\cap B_i) = \alpha] \}.
\]
Now, let $\alpha\in S'$ be arbitrary. We have  $|\mathcal D_\alpha|\le|\mathcal C_\alpha|<\nu$.
Finally, let $i < \min\{\alpha, \theta\}$ and $D\in\mathcal D_\alpha$ be arbitrary.
Pick $C\in\mathcal C_\alpha$ such that $D=\Phi_{\mathfrak x}(C)$.
For all $\beta\in\nacc(C)\cap B_i$ above $\min(C)$, we have $\mathcal X_\beta\s\mathcal B_i$ so that for $\gamma:=\sup(C\cap\beta)$,
we obtain $x_{\gamma,\beta}\in\mathcal B_i$ and $x_{\gamma,\beta}\s(\gamma,\beta)$,
and we infer from the definition of $\Phi_{\mathfrak{x}}$ that 
$D\cap(\gamma,\beta)=x_{\gamma,\beta}\in\mathcal B_i$.
\end{proof}

Thus, $\p_\xi^-(\kappa,\mu,\mathcal R, \theta, \mathcal S, \nu, {<}\omega)$ follows, by the implication $(v)\implies(i)$ of Theorem~\ref{thm416}.

(2) Left to the reader to show that a sequence as in Clause~(v) of Theorem~\ref{thm416} with $\theta=1$
also witnesses $\p_\xi^-(\kappa,\mu,\mathcal R, {<}\omega, \mathcal S, \nu, {<}\omega)$.
\end{proof}

\begin{remark}
Clause~(1) of the preceding is used in the proofs of Theorems 3.11 and 3.12 of \cite{paper28},
and also, implicitly, in the proof of Corollary~3.15(2) of the same paper.
\end{remark}

In \cite[p.~1953]{paper22}, we wrote that if we omit $\sigma$, then we mean that $\sigma=1$,
but given Theorems \ref{thm416} and~\ref{get-sigma-finite}, we now upgrade it as follows.

\begin{conv}\label{omitsigma} Whenever we omit $\sigma$, we mean that we put ``${<}\omega$'' in place of $\sigma$.
\end{conv}

\subsection{Jensen's classical result and nonreflecting sets}\label{section:captureJensen}

In~\cite[Theorem~6.2]{jensen1972fine}, Jensen showed that, assuming $V=L$,
there is a $\kappa$-Souslin tree for every regular uncountable cardinal $\kappa$ that is not weakly compact.
The proof of that theorem actually derives a $\kappa$-Souslin tree from the 
existence of a stationary set $E\subset\kappa$ for which
$\square(E)+\diamondsuit(E)$ holds,\footnote{See \cite[Theorem~IV.2.4]{MR750828} and \cite[Lemma~11.68]{MR3243739}.}
where $\square(E)$ is the principle asserting the existence of a $\p^-(\kappa, 2, {\sqleftup{E}}, \dots)$-sequence,
and should not be confused with the weaker principle $\square(\kappa)$.
Here, we show how to redirect this result through our framework:\footnote{By Theorem~\ref{basicthm} together with Proposition~\ref{P->Pbullet}, the instances obtained in Corollary~\ref{corollaryalajensen} suffice for the construction of a $\kappa$-Souslin tree}
\begin{cor}\label{corollaryalajensen} Suppose $\square(E)+\diamondsuit(E)$ holds for some stationary $E\subset\kappa$. Then:
\begin{enumerate}
\item $\p^-(\kappa, 2, {\sq^*}, 1,\{E\},2)$;
\item If $\kappa>\aleph_1$, then $\p^-(\kappa,2,{\sq^*},1,\{S\},2)$ holds for every stationary $S\s\kappa$.
\end{enumerate}
\end{cor}
\begin{proof} 
(1) By Corollary~\ref{cor-KS93} below,
there exists a stationary $\Omega\s E$ such that $\p^-(\kappa, 2, \allowbreak{\sqstarleftboth{\Omega}{\chi}}, 1,\{\Omega\},2)$ holds for $\chi:=\omega$.
In particular, $\p^-(\kappa, 2, {\sq^*},1,\{E\},2)$ holds.

(2) This follows from Clause~(1), \cite[Theorem~4.13]{paper24}, and Theorem~\ref{get-sigma-finite}.
\end{proof}

We now present the main technical lemma of this subsection.
In order to enable applications of this lemma to scenarios beyond the scope of this paper
(see~\cite{paper45}),
we temporarily suspend our usual convention requiring $\mu\leq\kappa$.
Thus, for the rest of this subsection only,
a $\p^-(\kappa, \kappa^+, \mathcal{R}, \dots)$-sequence 
is a sequence just as in Definition~\ref{defn-p-sequence-first-stage},
with the second bullet replaced by $0<|\mathcal C_\alpha|\leq\kappa$.

\begin{lemma}\label{from-KS93}
Suppose $\kappa^{\aleph_0}=\kappa$, $S\s\kappa$ is stationary, $\clubsuit(S)$ holds, 
and $\cvec{C}^0 = \langle \mathcal C^0_\alpha \mid \alpha<\kappa \rangle$ is a $\p^-(\kappa, \kappa^+, {\sqstarleftboth{S}{\chi}},\ldots)$-sequence.
Then there exist:
\begin{enumerate}
\item a subset $\Omega\subseteq S\cap\acc(\kappa)$;
\item a sequence $\langle X_\beta \mid \beta \in \Omega \rangle$, where each $X_\beta$ is a cofinal subset of $\beta$;
\item a $\p^-(\kappa, \kappa^+, {\sqstarleftboth{\Omega}{\chi}},\ldots)$-sequence, $\langle \mathcal C_\alpha \mid \alpha<\kappa \rangle$, with $|\mathcal C_\alpha|\le|\mathcal C_\alpha^0|$ for all $\alpha<\kappa$,
\setcounter{condition}{\value{enumi}}
\end{enumerate}
together satisfying the following hitting feature:
\begin{enumerate}
\setcounter{enumi}{\value{condition}}
\item
for every cofinal $X\s\kappa$, there exists some $\alpha\in\Omega$
such that $\mathcal C_\alpha$ is a singleton, say, $\mathcal C_\alpha = \{C_\alpha\}$, and
\[
\sup \{\beta\in\nacc(C_\alpha)\cap \Omega\mid X_\beta\s X\} =\alpha.
\]
\end{enumerate}
\end{lemma}
\begin{proof}
Following Convention~\ref{default-outside-acc}, we assume that $\mathcal C^0_{\alpha+1} = \{\{\alpha\}\}$ for every $\alpha<\kappa$.
Furthermore, by $\sqstarleftboth{S}{\chi}$-coherence of the sequence $\cvec{C}^0$,
we may assume that $\left|\mathcal C^0_\alpha\right| = 1$ for every $\alpha \in S$.

Let $\langle X_\beta^n\mid \beta\in S, n\le\omega\rangle$ be given by Lemma~\ref{infinite-club-matrix} using $\theta:=\omega$.
Without loss of generality, for all $\beta\in S\cap\acc(\kappa)$ and $n\le\omega$, $X_\beta^n$ is a cofinal subset of $\beta$.

Suppose first that, for every cofinal $X\s\kappa$, the set
$$B_X:=\{ \beta\in S\cap E^\kappa_\omega\mid X_\beta^\omega\s X\}$$ is stationary.
In this case, let $\Omega:=S\cap E^\kappa_\omega$,
define $\langle X_\beta\mid \beta\in\Omega\rangle$ via $X_\beta:=X_\beta^\omega$,
and define a sequence $\cvec{C} = \langle \mathcal C_\alpha \mid \alpha<\kappa \rangle$ 
by letting $\mathcal C_\alpha := \mathcal C^0_\alpha$ for every $\alpha\notin\Omega$,
and $\mathcal C_\alpha :=  \{X^\omega_\alpha\}$ for every $\alpha\in\Omega$.
It is clear that $\cvec{C}$ is a $\p^-(\kappa, \kappa^+, {\sqstarleftboth{S}{\chi}},\ldots)$-sequence,
satisfying $|\mathcal C_\alpha|\le|\mathcal C_\alpha^0|$ for all $\alpha<\kappa$.
In particular, it is a $\p^-(\kappa, \kappa^+, {\sqstarleftboth{\Omega}{\chi}},\ldots)$-sequence.
To see that Clause~(4) is satisfied, let $X\s\kappa$ be an arbitrary cofinal set.
As $B_X$ is stationary, it is in particular cofinal, so that $B_{B_X}$ is stationary as well.
Pick $\alpha\in B_{B_X}$. Then $\alpha\in\Omega$,
$\mathcal C_\alpha=\{ X_\alpha^\omega\}$, and $X_\alpha^\omega\s B_X$. 
So, if $C_\alpha$ denotes the unique element of $\mathcal C_\alpha$, then every $\beta\in\nacc(C_\alpha)$
is an element of $B_X$, so that $\beta\in S\cap E^\kappa_\omega=\Omega$ and $X_\beta=X_\beta^\omega\s X$, as sought.

From now on, suppose that there exists a cofinal $X^\omega\s\kappa$ for which $B_{X^\omega}$ is nonstationary.
Denote $\Omega^0 := S \cap E^\kappa_{>\omega}$.
 It follows that, for every sequence $\langle X^n\mid n<\omega\rangle$ of cofinal subsets of $\kappa$,
there exist stationarily many $\beta\in \Omega^0$, such that, for all $n<\omega$,
$X^n_\beta\s X^n$. In particular, $\kappa\ge\aleph_2$.
The proof in this case is based on the arguments of \cite[Theorem~3]{KjSh:449} and \cite[Claim~3.8.1]{paper32}.

We have already defined $\Omega^0$ and $\cvec{C}^0$.
We now attempt to construct a sequence 
$\langle (\Omega^{n+1}, \cvec{C}^{n+1}, X^n) \mid n < \omega \rangle$
satisfying the following properties for all $n < \omega$:  
\begin{enumerate}[(a)]
\item\label{clausea} $\Omega^{n+1} \subseteq \Omega^n$; 
\item\label{clauseb} 
$\cvec{C}^n = \langle \mathcal C^n_\alpha \mid \alpha<\kappa \rangle = \langle \{C^n_{\alpha,\iota} \mid \iota<\kappa\} \mid \alpha<\kappa \rangle$ is a $\p^-(\kappa, \kappa^+, {\sqstarleftboth{\Omega^n}{\chi}},\ldots)$-sequence,\footnote{The enumeration $\{C^n_{\alpha,\iota} \mid \iota<\kappa\}$ of $\mathcal C^n_\alpha$ need not be injective.}
such that $\mathcal{C}^n_{\alpha+1} = \{\{\alpha\}\}$ for every $\alpha<\kappa$,
and $\left|\mathcal{C}^n_\alpha\right| = 1$ for every $\alpha\in S$;
\item\label{clausec} 
$C^{n}_{\alpha,\iota} \s C^{n+1}_{\alpha,\iota}$ for every $\alpha<\kappa$ and every $\iota<\kappa$;
\item\label{anothernewclause}
for every $\alpha<\kappa$ and all $\iota,\iota'<\kappa$,
if $C^n_{\alpha,\iota} = C^n_{\alpha,\iota'}$, then $C^{n+1}_{\alpha,\iota} = C^{n+1}_{\alpha,\iota'}$;
\item\label{claused} $X^n$ is a cofinal subset of $\kappa$;
\item\label{clausee} $\sup(C^n_{\alpha,0} \cap \Omega^{n+1}) < \alpha$ for every $\alpha \in \Omega^n$;
\item\label{clausef} 
for every $\alpha\in\acc(\kappa)$, all $\iota,\iota'<\kappa$, every $\bar\alpha\in\acc(C^n_{\alpha,\iota'})$,
and every $\eta\in C^n_{\bar\alpha,\iota}\cap C^n_{\alpha,\iota'}$, 
if $C^n_{\bar\alpha,\iota}\setminus\eta=C^n_{\alpha,\iota'}\cap[\eta,\bar\alpha)$
then $C^{n+1}_{\bar\alpha,\iota}\setminus\eta= C^{n+1}_{\alpha,\iota'}\cap[\eta,{\bar\alpha})$;
\item\label{clauseg} 
$\nacc(C^n_{\alpha,\iota}) \cap \Omega^{n+1} \subseteq \nacc(C^{n+1}_{\alpha,\iota})$ for all $\alpha,\iota<\kappa$;
\item\label{clauseh} 
for all $\alpha,\iota<\kappa$ and any two consecutive points $\beta^-<\beta$ of $C^n_{\alpha,\iota}$,
either $C^{n+1}_{\alpha,\iota} \cap (\beta^-,\beta) = \emptyset$ or
$C^{n+1}_{\alpha,\iota} \cap [\beta^-,\beta) = C^{n+1}_{\beta,0} \setminus\beta^-$.\footnote{We do not assume here that $\beta$ is a limit ordinal.}
\end{enumerate}

We proceed by recursion.
Fix $n<\omega$ for which $\Omega^n$ and $\cvec{C}^n$
have already been successfully constructed.
By \ref{clausea} and~\ref{clauseb}, $\langle X^n_\beta \mid\beta \in \Omega^n \rangle$ and $\cvec{C}^n$
satisfy properties (1)--(3) of the conclusion of the lemma.
If they satisfy property~(4) as well, then the lemma is proven and we abandon the recursive construction at this point.

Otherwise, let us pick a cofinal subset $X^n\s\kappa$ witnessing the failure of~(4).
That is,
$\sup(\nacc(C^n_{\alpha,0})\cap \Omega^{n+1})<\alpha$ for every $\alpha\in \Omega^n$, 
where we define:
$$\Omega^{n+1}:=\{\beta\in \Omega^n \mid X^n_\beta\s X^n\}.$$

Define $\cvec{C}^{n+1} = \langle \mathcal C^{n+1}_\alpha\mid\alpha<\kappa\rangle$ 
by recursion over $\alpha<\kappa$, as follows. Let $\mathcal C^{n+1}_0:=\{\emptyset\}$, and $\mathcal C^{n+1}_{\alpha+1}:=\{\{\alpha\}\}$ for every $\alpha<\kappa$.
Now, if $\alpha\in\acc(\kappa)$ and $\langle \mathcal C^{n+1}_{\bar\alpha}\mid {\bar\alpha}<\alpha\rangle$ has already been defined, then, for every $\iota<\kappa$, let:
\[
C^{n+1}_{\alpha,\iota}:=C^n_{\alpha,\iota}\cup \bigcup \bigl\{C^{n+1}_{\beta,0}\setminus\sup(C^n_{\alpha,\iota}\cap\beta)\bigm| \beta\in\nacc(C^n_{\alpha,\iota})\setminus \Omega^{n+1} \bigr\}.\tag{$*$}
\]

\begin{claim}\label{c251}
$\Omega^{n+1}$, $\cvec{C}^{n+1}$, and $X^n$
satisfy properties \ref{clausea}--\ref{clauseh} of the recursion.
\end{claim}
\begin{proof}
\ref{clausea}, \ref{clausec}, \ref{anothernewclause}, \ref{claused}, \ref{clausef}, \ref{clauseg}, and \ref{clauseh}
are easily verified from the construction.

\ref{clausee} 
Consider any given $\alpha \in \Omega^n$.
By $\sqstarleftboth{\Omega^n}{\chi}$-coherence of $\cvec{C}^n$,
we obtain $\acc(C^n_{\alpha,0}) \cap \Omega^n = \emptyset$,
so that certainly $\acc(C^n_{\alpha,0}) \cap \Omega^{n+1} = \emptyset$.
Thus $C^n_{\alpha,0} \cap \Omega^{n+1} = \nacc(C^n_{\alpha,0}) \cap \Omega^{n+1}$,
and it follows from our choice of $X^n$ and $D^n$ that
$$\sup(C^n_{\alpha,0} \cap \Omega^{n+1}) = \sup(\nacc(C^n_{\alpha,0}) \cap \Omega^{n+1}) < \alpha.$$

\ref{clauseb}
It is clear from the construction that $C^{n+1}_{\alpha,\iota}$ is a club subset of $\alpha$ for every limit ordinal $\alpha<\kappa$,
that $\mathcal{C}^{n+1}_{\alpha+1} = \{\{\alpha\}\}$ for every $\alpha<\kappa$,
and (using~\ref{anothernewclause}) that 
$0<\left|\mathcal{C}^{n+1}_\alpha\right|\le \left|\mathcal{C}^{n}_\alpha\right|\le \left|\mathcal{C}^{0}_\alpha\right|$ for every $\alpha<\kappa$ and
$\left|\mathcal{C}^{n+1}_\alpha\right| = 1$ for every $\alpha\in S$.

To verify $\sqstarleftboth{\Omega^{n+1}}{\chi}$-coherence of the sequence $\cvec{C}^{n+1}$,
consider any given $\alpha \in \acc(\kappa)$, $\iota<\kappa$, and
$\bar\alpha\in\acc(C^{n+1}_{\alpha,\iota})$;
we shall show that $\bar\alpha\notin \Omega^{n+1}$, and also find some $\iota'<\kappa$ such that $C^{n+1}_{\bar\alpha,\iota'} \sqleft{\chi}^* C^{n+1}_{\alpha,\iota}$,
by considering three cases.
We may assume, as an induction hypothesis, that the restriction of $\cvec{C}^{n+1}$ up to $\alpha$, that is,
$\langle \mathcal C^{n+1}_\beta \mid \beta<\alpha \rangle$,
is already known to be $\sqstarleftboth{\Omega^{n+1}}{\chi}$-coherent.

\begin{itemize}
\item[$\br$]
Suppose $\bar\alpha \in \acc(C^n_{\alpha,\iota})$.
Then by $\sqstarleftboth{\Omega^n}{\chi}$-coherence of $\cvec{C}^n$,
it follows that $\bar\alpha\notin\Omega^n$ and also we can choose some $\iota'<\kappa$ such that $C^n_{\bar\alpha,\iota'} \sqleft{\chi}^* C^n_{\alpha,\iota}$.
As $\Omega^{n+1}\s \Omega^n$, we obtain $\bar\alpha\notin\Omega^{n+1}$.
If $\cf(\bar\alpha)<\chi$, then automatically $C^{n+1}_{\bar\alpha,0} \sqleft{\chi}^* C^{n+1}_{\alpha,\iota}$.
Otherwise, we have $C^n_{\bar\alpha,\iota'} \sq^* C^n_{\alpha,\iota}$, 
and we fix $\eta \in C^n_{\bar\alpha,\iota'} \cap C^n_{\alpha,\iota}$ such that 
$C^n_{\bar\alpha,\iota'} \setminus \eta = C^n_{\alpha,\iota} \cap [\eta,\bar\alpha)$.
Then, by Clause~\ref{clausef}, we obtain
$C^{n+1}_{\bar\alpha,\iota'} \setminus \eta = C^{n+1}_{\alpha,\iota} \cap [\eta,\bar\alpha)$,
so that $C^{n+1}_{\bar\alpha,\iota'} \sq^* C^{n+1}_{\alpha,\iota}$.

\item[$\br$]
Suppose $\bar\alpha \in \nacc(C^n_{\alpha,\iota})$.
As we have assumed that $\bar\alpha \in \acc(C^{n+1}_{\alpha,\iota})$,
we obtain from Clause~\ref{clauseg} that $\bar\alpha\notin \Omega^{n+1}$.
Furthermore, it follows from Clause~\ref{clauseh} 
(with $\beta := \bar\alpha$ and $\beta^- := \sup(C^n_{\alpha,\iota}\cap\beta)$)
that $C^{n+1}_{\bar\alpha,0} \sq^* C^{n+1}_{\alpha,\iota}$.

\item[$\br$]
The remaining case is $\bar\alpha \notin C^n_{\alpha,\iota}$.
Let $\beta^-<\beta$ be the two consecutive points of $C^n_{\alpha,\iota}$ such that $\beta^-<\bar\alpha<\beta$.
In particular, $\beta \in \nacc(C^n_{\alpha,\iota})$.
Since $\bar\alpha \in \acc(C^{n+1}_{\alpha,\iota}) \cap (\beta^-,\beta)$,
it follows from Clause~\ref{clauseh} that
$C^{n+1}_{\beta,0}\setminus\beta^- = C^{n+1}_{\alpha,\iota} \cap [\beta^-,\beta)$,
so that $\bar\alpha \in \acc(C^{n+1}_{\beta,0} \setminus \beta^-)$.
As $\beta<\alpha$ and $\bar\alpha \in \acc(C^{n+1}_{\beta,0})$,
by induction hypothesis we can choose some $\iota'<\kappa$ such that
$C^{n+1}_{\bar\alpha,\iota'} \sqstarleftboth{\Omega^{n+1}}{\chi} C^{n+1}_{\beta,0}$.
Thus, in particular, $\bar\alpha\notin \Omega^{n+1}$ and $C^{n+1}_{\bar\alpha,\iota'} \sqleft{\chi}^* C^{n+1}_{\beta,0}$.
If $\cf(\bar\alpha)<\chi$, then automatically $C^{n+1}_{\bar\alpha,0} \sqleft{\chi}^* C^{n+1}_{\alpha,\iota}$.
Otherwise, we have $C^{n+1}_{\bar\alpha,\iota'} \sq^* C^{n+1}_{\beta,0}$, 
and as $\beta^-<\bar\alpha$, it follows that $C^{n+1}_{\bar\alpha,\iota'} \sq^* C^{n+1}_{\alpha,\iota}$,
as sought.
\qedhere
\end{itemize}
\end{proof}

If we reach the end of the above recursive process, then we are altogether equipped with a sequence
$\langle (\cvec{C}^n,\Omega^n,X^n)\mid n<\omega\rangle$,
from which we shall derive a contradiction.
By the choice of $X^\omega$, the following set must be stationary:
$$\Omega^\bullet :=\bigcap_{n<\omega} \Omega^n=\bigl\{ \beta\in \Omega^0\bigm| \forall n<\omega (X^n_\beta\s X^n)\bigr\}.$$
Thus, $\acc^+(\Omega^\bullet)$ is a club in $\kappa$, and we may
pick $\alpha\in \Omega^\bullet\cap\acc^+(\Omega^\bullet)$.
For every $n<\omega$, put $\alpha_n:=\sup(C^n_{\alpha,0}\cap \Omega^{n+1})$,
which by Clause~\ref{clausee} is $<\alpha$.
As $\alpha \in \Omega^\bullet \subseteq \Omega^0 \subseteq E^\kappa_{>\omega}$, 
we infer that $\alpha^*:=\sup_{n<\omega}\alpha_n$ is smaller than $\alpha$.
Clearly,
\[
C^n_{\alpha,0} \cap \Omega^{n+1} \cap (\alpha^*, \alpha) =  \emptyset \text{ for every } n<\omega.
\tag{$**$}
\]

Since $\Omega^\bullet \cap \alpha$ is cofinal in $\alpha$,
let us pick $\beta\in \Omega^\bullet \cap(\alpha^*,\alpha)$.

For all $n<\omega$, let $\beta_n:=\min(C^n_{\alpha,0}\setminus\beta)$.
As $\{ C^n_{\alpha,0}\mid n<\omega\}$ is a $\subseteq$-increasing chain, 
it follows that $\langle \beta_n\mid n<\omega\rangle$ is a $\leq$-decreasing sequence of ordinals, 
and hence stabilizes. Fix $n<\omega$ such that $\beta_n=\beta_{n+1}$.
Since $\beta \in \Omega^\bullet \subseteq \Omega^{n+1}$,
($**$) gives $\beta \notin C^n_{\alpha,0}$, so that $\beta_n > \beta$.
In particular, $\beta_n = \min(C^n_{\alpha,0} \setminus (\beta+1))$,
so that $\beta_n\in\nacc(C^n_{\alpha,0})$.
By ($**$) again and $\alpha > \beta_n > \beta > \alpha^*$,
it follows that $\beta_n\in\nacc(C^n_{\alpha,0})\setminus \Omega^{n+1}$.
Thus we infer from ($*$) that
$$C^{n+1}_{\alpha,0}\cap\bigl[\sup(C^n_{\alpha,0}\cap\beta_n),\beta_n\bigr)=C^{n+1}_{\beta_n,0}\cap\bigl[\sup(C^n_{\alpha,0}\cap\beta_n),\beta_n\bigr),$$
and from $\beta_n=\min(C^n_{\alpha,0}\setminus\beta)$, we obtain $\sup(C^n_{\alpha,0}\cap\beta_n)\le\beta$, and hence
$$C^{n+1}_{\alpha,0}\cap[\beta,\beta_n)=C^{n+1}_{\beta_n,0}\cap[\beta,\beta_n).$$

Since $\beta < \beta_n$, it follows that
$\beta_{n+1} = \min (C^{n+1}_{\alpha,0} \setminus \beta) = \min(C^{n+1}_{\beta_n,0} \setminus \beta) < \beta_n$, contradicting the choice of $n$
and completing the proof.
\end{proof}
\begin{remark}\label{remark420} It follows from Lemma~\ref{thinning-out} that, in Clause~(4) of the preceding, 
the existence of an $\alpha$ is equivalent to the existence of stationarily many such $\alpha$.
\end{remark}

\begin{cor}\label{cor-KS93}
Suppose $\kappa^{\aleph_0}=\kappa$, and $\clubsuit(E)$ holds for a stationary $E\s\kappa$.
If there exists a $\p^-(\kappa, \mu, {\sqstarleftboth{E}{\chi}},\ldots)$-sequence,
then there exists a stationary $\Omega\s E$ such that $\p^-(\kappa, \mu, {\sqstarleftboth{\Omega}{\chi}},1,\{\Omega\},2)$ holds.
\end{cor}
\begin{proof}  Let $\Omega\subseteq E\cap\acc(\kappa)$,
$\langle X_\beta \mid \beta \in \Omega \rangle$,
and $\langle \mathcal C_\alpha \mid \alpha<\kappa \rangle$ be given by Lemma~\ref{from-KS93}, using $S:=E$.
For every $\beta\in\Omega$, let $\mathcal X_\beta:=\{\{\delta\}\mid \delta\in X_\beta\}$.

Set $\nu:=2$. By Clause~(4) of Lemma~\ref{from-KS93} together with Remark~\ref{remark420},
for every cofinal $X\s\kappa$, letting $\mathcal X:=\{\{\delta\}\mid \delta\in X\}$
and $B_0:=\{\beta\in\Omega\mid \mathcal X_\beta\s \mathcal X\}$,
the set $S'$ of all $\alpha\in\Omega$ such that:
\begin{itemize}
\item $|\mathcal C_\alpha|<\nu$, and
\item for every $C\in\mathcal C_\alpha$, $\sup(\nacc(C)\cap B_0)=\alpha$
\end{itemize}
is stationary.
The rest of the proof is now identical to that of Theorem~\ref{get-sigma-finite}(1),
using $\mathcal R:={\sqstarleftboth{\Omega}{\chi}}$ and $\theta:=1$.
\end{proof}

Suppose $\kappa>\aleph_1$.
Recall that a subset $\Omega\s\kappa$ is said to be \emph{nonreflecting}
iff $\tr(\Omega):=\{\alpha<\kappa\mid \cf(\alpha)>\omega\ \&\ \Omega\cap\alpha\text{ is stationary in } \alpha \}$ is empty.
In particular, 
any $\Omega\s\kappa$ for which there exists a $\p^-(\kappa, \mu, {}^\Omega\mathcal{R}, \dots)$-sequence (regardless of the value of $\mu$)
must be a nonreflecting set,
so that a set $S$ as in the hypothesis of Lemma~\ref{from-KS93} must be a nonreflecting stationary set.
In \cite{paper32}, more applications of the existence of nonreflecting stationary sets were presented:

\begin{fact}[{\cite[Theorem~A]{paper32}}] Suppose that $\ch_\lambda$ holds for a regular uncountable cardinal $\lambda=2^{<\lambda}$,
and there exists a nonreflecting stationary subset of $E^{\lambda^+}_{\neq\lambda}$. Then $\p^-_\lambda(\lambda^+,\lambda^+,{\sq},{<}\lambda,\allowbreak\{\lambda^+\},2,{<}\lambda)$
and $\p^-(\lambda^+,\lambda^+,{\sq^*},1,\allowbreak\{E^{\lambda^+}_\lambda\},2,1)$ both hold.
\end{fact}

\begin{fact}[{\cite[Theorem~B]{paper32}}]  Suppose that $\square^*_\lambda+\ch_\lambda$ holds for a singular cardinal $\lambda=2^{<\lambda}$,
and either of the following two conditions holds:
\begin{enumerate}
\item there exists a nonreflecting stationary subset of $E^{\lambda^+}_{\neq\cf(\lambda)}$; or
\item there exists a regressive function $f:E^{\lambda^+}_{\cf(\lambda)}\rightarrow\lambda^+$ such that $f^{-1}\{i\}$ is nonreflecting for each $i<\lambda^+$.
\end{enumerate}

Then $\p^-_{\lambda^2}(\lambda^+,\lambda^+,{\sq},\lambda^+,\{\lambda^+\},2,{<}\lambda)$ holds. 
\end{fact}

Our next task is to prove analogous results for strongly inaccessible cardinals.
For this, let us first recall a result from \cite{paper29}.

\begin{fact}[Special case of {\cite[Lemma~4.9]{paper29}}]\label{blowup}  Suppose that:
\begin{enumerate}[(a)]
\item $\diamondsuit(\kappa)$ holds;
\item $\cvec{C}=\langle \mathcal C_\alpha\mid \alpha<\kappa\rangle$ is a $\p^-(\kappa,\kappa,{\sqleftup{\Omega}}, \ldots)$-sequence for some fixed subset $\Omega\s\kappa\setminus\{\omega\}$;
\item $\vec C=\langle C_\alpha \mid \alpha\in\acc(\kappa)\rangle$ is an element of $\prod_{\alpha\in\acc(\kappa)}\mathcal C_\alpha$;
\item\label{hyp-c-nacc} for every cofinal $B\s\kappa$  and every $\Lambda' < \kappa$, the following is stationary in $\kappa$:
\[
\{\alpha\in\kappa\setminus\Omega\mid \exists C \in \mathcal C_\alpha [\min(C)=\min(B), \Lambda' \le\otp(C),\nacc(C)\s B]\}.
\]
\end{enumerate}

Then there exists a $\p^-(\kappa,\kappa,{\sqleftup{\Omega}},\ldots)$-sequence, $\cvec{D}=\langle \mathcal D_\alpha\mid\alpha<\kappa\rangle$, 
with an element $\vec D=\langle D_\alpha\mid\alpha\in\acc(\kappa)\rangle$ of $\prod_{\alpha\in\acc(\kappa)}\mathcal D_\alpha$,
satisfying the following properties:
\begin{enumerate}
\item $|\mathcal D_\alpha| \leq |\mathcal C_\alpha|$ for all $\alpha < \kappa$;
\item For every cofinal $A\s\kappa$, there exists a stationary $B\s\kappa$ for which the set
$$\{\alpha\in\acc(\kappa)\mid \otp(D_\alpha)=\alpha, \nacc(D_\alpha)\s A\}$$
covers the set
$$\{\alpha\in\acc(\kappa)\mid \min(C_\alpha)=\min(B), \otp(C_{\alpha})=\cf(\alpha), \nacc(C_{\alpha})\s B \}.$$
\end{enumerate}
\end{fact}

Clause~(2) of the next theorem is used in the proof of~\cite[Corollary~4.13]{paper32}.

\begin{thm}\label{thm43} Suppose that $S$ is a nonreflecting stationary subset of a strongly inaccessible cardinal $\kappa$ and $\diamondsuit(S)$ holds.
Then:
\begin{enumerate}
\item $\p^-(\kappa,\kappa,{\sqleftup{S}},1,\{S\},2,\kappa)$ holds;
\item $\p^-(\kappa,\kappa,{\sqleftup{S}},\kappa,\{S\},2,{<}\infty)$ holds.
\end{enumerate}
\end{thm}
\begin{proof} 
Fix a $\diamondsuit(S)$-sequence $\vec A=\langle A_\alpha\mid\alpha\in S\rangle$.
As $S$ is nonreflecting, for every $\alpha\in\acc(\kappa)$, we may pick a closed and cofinal subset $B_\alpha$ of $\alpha$ with $\otp(B_\alpha)=\cf(\alpha)$ and $\acc(B_\alpha)\cap S=\emptyset$.
Now, define a sequence $\vec C=\langle C_\alpha\mid\alpha\in\acc(\kappa)\rangle$, as follows.
\begin{itemize}
\item[$\br$] If $\alpha\in S$ and $A_\alpha$ is a cofinal subset of $\alpha$,
then by Lemma~\ref{thinning-out}, fix a cofinal subset $A'_\alpha$ of $A_\alpha$
that is $B_\alpha$-separated, and then let $C_\alpha$ be the closure of $(A'_\alpha\cup\{\min(A_\alpha)\})$ in $\alpha$;

\item[$\br$] Otherwise, let $C_\alpha:=B_\alpha$.
\end{itemize}

Note that, for every $\alpha\in\acc(\kappa)$, $\acc(C_\alpha)\s\acc(B_\alpha)$, so that $\otp(C_\alpha)=\cf(\alpha)$ and $\acc(C_\alpha)\cap S=\emptyset$.

Next, recalling Convention~\ref{default-outside-acc}, we define a sequence $\cvec{C}=\langle \mathcal C_\alpha\mid\alpha<\kappa\rangle$, as follows.
For $\alpha\in S$, we let $\mathcal C_\alpha:=\{C_\alpha\}$,
and for $\alpha\in\acc(\kappa)\setminus S$, we let $\mathcal C_\alpha$ consist of all closed and cofinal subsets $C$ of $\alpha$ such that $\acc(C)\cap S=\emptyset$.

Evidently, $\vec C$ is an element of $\prod_{\alpha\in\acc(\kappa)}\mathcal C_\alpha$.
In particular, and as $\kappa$ is strongly inaccessible, it follows that $\cvec{C}$ is a $\p^-(\kappa,\kappa,{\sqleftup{\Omega}}, \ldots)$-sequence for $\Omega:=S\setminus\{\omega\}$.

\begin{claim} For every cofinal $B\s\kappa$ and $\Lambda'<\kappa$, the following sets are stationary in $\kappa$:
\begin{enumerate}[(i)]
\item $\{\alpha\in \Omega\mid \min(C_\alpha)=\min(B), \otp(C_{\alpha})=\cf(\alpha), \nacc(C_{\alpha})\s B \}$, and
\item $\{\alpha\in\kappa\setminus\Omega\mid \exists C \in \mathcal C_\alpha [\min(C)=\min(B), \Lambda' \le\otp(C),\nacc(C)\s B]\}$.
\end{enumerate}
\end{claim} 
\begin{proof}  Let $B$ be an arbitrary cofinal set.

(i) Consider the club $D:=\acc^+(B\setminus\omega)$.
By the choice of $\vec A$, $S':=\{\alpha\in S\cap D\mid B\cap\alpha=A_\alpha\}$ is a stationary subset of $\Omega$.
For every $\alpha\in S'$, the definition of $C_\alpha$ implies that $\otp(C_\alpha)=\cf(\alpha)$, $\nacc(C_\alpha)\s A_\alpha\s B$ and $\min(C_\alpha)=\min(A_\alpha)=\min(B)$.

(ii) Let $\Lambda'<\kappa$ be arbitrary. Let $D$ be an arbitrary club. 
By shrinking $D$, we may assume that $D\s\acc^+(B)$.
Set $\theta:=\max\{\Lambda',\aleph_0\}$.
As $\kappa$ is a limit cardinal, we may fix $\delta\in\acc(D)$ with $\cf(\delta)=\theta^{++}$.
Pick $\alpha\in\acc(C_\delta)\cap D$ with $\cf(\alpha)=\theta^+$.
As $\vec C$ is a $\p^-(\kappa,\kappa,{\sqleftup{\Omega}}, \ldots)$-sequence, $\alpha\notin\Omega$.
Finally, since $\alpha\in D\s\acc^+(B)$, we infer that $B\cap\alpha$ is cofinal in $\alpha$.
Let $A'$ be a cofinal subset of $A:=B\cap\alpha$ that is $B_\alpha$-separated.
Let $C$ be the closure of $(A'\cup\{\min(A)\})$ in $\alpha$.
Then $C\in\mathcal C_\alpha$, $\min(C)=\min(A)=\min(B)$, $\Lambda'<\theta^+=\otp(C)$ and $\nacc(C)\s A\s B$.
\end{proof}

By Clause~(i) of the preceding claim, $\cvec{C}$ witnesses $\p^-(\kappa,\kappa,{\sqleftup{\Omega}},1,\{\Omega\},2,\kappa)$.
As $S\symdiff\Omega$ is finite, it follows that $\p^-(\kappa,\kappa,{\sqleftup{S}},1,\{S\},2,\kappa)$ holds, as well.

Next, by appealing to Fact~\ref{blowup} with $\cvec{C}$ and $\vec C$, we obtain a $\p^-(\kappa,\kappa,{\sqleftup{\Omega}},\ldots)$-sequence $\cvec{D}=\langle \mathcal D_\alpha\mid\alpha<\kappa\rangle$, 
and an element $\vec D=\langle D_\alpha\mid\alpha\in\acc(\kappa)\rangle$ of $\prod_{\alpha\in\acc(\kappa)}\mathcal D_\alpha$,
satisfying the following properties:
\begin{enumerate}
\item $|\mathcal D_\alpha| \leq |\mathcal C_\alpha|$ for all $\alpha < \kappa$;
\item For every cofinal $G\s\kappa$, the following set is stationary:
$$\{\alpha\in \Omega\mid \otp(D_\alpha)=\alpha, \nacc(D_\alpha)\s G\}.$$
\end{enumerate}
It thus follows from \cite[Lemma~3.7]{paper29} and \cite[Lemma~3.14]{paper32} that $\p^-(\kappa,\kappa,{\sqleftup{\Omega}},\allowbreak\kappa,\{\Omega\},2,{<}\infty)$ holds.
As $S\symdiff\Omega$ is finite, it follows that  $\p^-(\kappa,\kappa,{\sqleftup{S}},\kappa,\{S\},2,{<}\infty)$ hold, as well.
\end{proof}

\begin{cor}\label{cor29} Suppose that $E$ is a nonreflecting stationary subset of a strongly inaccessible cardinal $\kappa$ and $\diamondsuit(E)$ holds. 
Then, for every nonreflecting stationary $S\s\kappa$, $\p^-(\kappa,\kappa,{\sq^*},1,\{S\},2)$ holds.
\end{cor}
\begin{proof} 
Let $S\s\kappa$ be stationary and nonreflecting.
Then $E \cup S$ is also stationary and nonreflecting and $\diamondsuit(E \cup S)$ holds, so that
by Theorem~\ref{thm43}(1), $\p^-(\kappa,\kappa,\allowbreak{\sqleftup{E\cup S}},1,\{E\cup S\},2,\kappa)$ holds.
In particular, we may fix a sequence $\cvec{C} = \langle \mathcal{C}_\alpha \mid \alpha<\kappa \rangle$
witnessing $\p^-(\kappa,\kappa,{\sqleftup{S}},1,\{\kappa\},2,1)$.
By $\sqleftup{S}$-coherence of $\cvec{C}$,
we may assume that $\left|\mathcal{C}_\alpha\right| = 1$ for every $\alpha \in S$.
So, by \cite[Lemma 3.8]{paper32} with $(\mu,\nu) := (\kappa, 2)$, $\p^-(\kappa,\kappa,{\sq^*},1,\{S\},2,1)$ holds.
Then, by Theorem~\ref{get-sigma-finite}(1), $\p^-(\kappa,\kappa,{\sq^*},1,\{S\},2)$ holds, as well.
\end{proof}

Our next result about strongly inaccessible cardinals forms the core of Theorem~B:
\begin{thm}\label{cor430} Suppose that $\kappa$ is a strongly inaccessible cardinal,
and there exists a sequence $\langle A_\alpha\mid \alpha\in S\rangle$ such that:
\begin{itemize}
\item $S$ is a nonreflecting stationary subset of $E^\kappa_{>\omega}$;
\item For every $\alpha\in S$, $A_\alpha$ is a cofinal subset of $\alpha$;
\item For every cofinal $B\s\kappa$, there exists $\alpha\in S$
for which 
$$\{\delta<\alpha\mid \min(A_\alpha\setminus(\delta+1))\in B\}$$
is stationary in $\alpha$.
\end{itemize}
Then $\p^-(\kappa,\kappa,{\sqleftup{S}}, 1,\{S\},2)$ holds.
\end{thm}
\begin{proof} 
As $S$ is nonreflecting, for every $\alpha\in S$, we may pick a club subset $B_\alpha$ of $\alpha$ with $\acc(B_\alpha)\cap S=\emptyset$.
Now, for every $\alpha\in S$, let $\mathcal C_\alpha:=\{C_\alpha\}$, where $C_\alpha$ denotes the closure in $\alpha$ of $\{\min(A_\alpha\setminus(\delta+1))\mid \delta\in B_\alpha\}$.
For every $\alpha\in\acc(\kappa)\setminus S$, let $\mathcal C_\alpha$ consist of all clubs $C$ in $\alpha$ such that $\acc(C)\cap S=\emptyset$.
As $S$ is nonreflecting, $\kappa$ is strongly inaccessible, and $\acc(C_\alpha)\s\acc(B_\alpha)$ for every $\alpha \in S$,
it follows that $\cvec{C} := \langle \mathcal C_\alpha\mid\alpha<\kappa\rangle$ is a $\p^-(\kappa,\kappa,\sqleftup{S},\dots)$-sequence.
To see that $\cvec{C}$ witnesses
Clause~(iv) of Theorem~\ref{thm416} with $(\xi,\mu,\mathcal R,\theta,\mathcal S,\nu):=(\kappa,\kappa,{\sqleftup{S}},1,\{S\},2)$,
consider any given club $D\s\kappa$ and cofinal $B\s\kappa$.
By Lemma~\ref{thinning-out}, fix a cofinal subset $B'$ of $B$ that is $D$-separated.
By the hypothesis, fix $\alpha \in S$ such that
$\{\delta<\alpha\mid \min(A_\alpha\setminus(\delta+1))\in B'\}$
is stationary in $\alpha$.
In particular, $\sup(B'\cap\alpha) = \alpha$, so that also $\sup(D\cap\alpha) = \alpha$,
and it follows that $D\cap\alpha$ is a club in $\alpha$, and so is $D \cap B_\alpha$.
Thus, the set $\{\delta \in D \cap B_\alpha\mid \min(A_\alpha\setminus(\delta+1))\in B'\}$
is stationary in $\alpha$.
As $\alpha \in S$, we infer that $\mathcal{C}_\alpha = \{C_\alpha\}$
(in particular, $|\mathcal{C}_\alpha| = 1$),
and for every $\delta \in B_\alpha$, $\min(C_\alpha\setminus(\delta+1)) = \min(A_\alpha\setminus(\delta+1))$,
yielding, in particular,
$\sup\{\delta\in D\cap\alpha\mid \min(C_\alpha\setminus(\delta+1))\in B\}=\alpha$.
Thus, by Theorem~\ref{thm416}(i), we obtain
$\p^-(\kappa,\kappa,{\sqleftup{S}},1,\{S\},2, {<}\omega)$, as sought.
\end{proof}

The proof of Lemma~\ref{from-KS93} makes clear that 
the $\diamondsuit(\kappa)$ hypothesis in \cite[Lemma 3.8]{paper32} may be reduced to $\clubsuit(\kappa)$ together with
$\kappa^{\aleph_1}=\kappa$, or together with $\kappa^{\aleph_0}=\kappa$ provided that $S$ concentrates on points of uncountable cofinality.
In addition, the proof of Theorem~\ref{get-sigma-finite}(1) goes through also for $\mu=\kappa^+$. Therefore, we arrive at the following conclusion.
\begin{lemma}\label{shiftcofinality}
Suppose that:
\begin{itemize}
\item $\mu\leq\kappa^+$ and $\nu<\kappa$ are cardinals;
\item $\cvec C=\langle\mathcal C_\alpha\mid\alpha<\kappa\rangle$ witnesses $\p^-(\kappa,\mu, {\sq^*},1,\{\kappa\},\mu,1)$;
\item $S$ is some stationary subset of $\{\alpha\in E^\kappa_{>\omega}\mid |\mathcal C_\alpha|<\nu\}$;
\item $\kappa^{\aleph_0}=\kappa$ and $\clubsuit(\kappa)$ holds.
\end{itemize}
Then $\p^-(\kappa,\mu,{\sq^*},1,\{S\},\nu)$ holds.\qed
\end{lemma}

\begin{cor}\label{cor426} Suppose $\kappa^{\aleph_0}=\kappa$, and $\clubsuit(E)$ holds for a nonreflecting stationary $E\s\kappa$.
\begin{enumerate}
\item For every $\mu\le\kappa^+$ and every stationary $S\s E^\kappa_{>\omega}$ such that there is a $\p^-(\kappa, \mu, {\sqstarleftup{E\cup S}},\ldots)$-sequence,
$\p^-(\kappa,\mu,{\sq^*},1,\{S\},2)$ holds;
\item For every nonreflecting stationary $S\s E^\kappa_{>\omega}$, $\p^-(\kappa,\kappa^+,{\sq^*},1,\{S\},2)$ holds.
\end{enumerate}
\end{cor}
\begin{proof} We shall only prove Clause~(1), as Clause~(2) easily follows from it.

Suppose that $S\s E^\kappa_{>\omega}$ is stationary, 
and there exists a $\p^-(\kappa, \mu, {\sqstarleftup{E\cup S}},\ldots)$-sequence,
$\langle \mathcal C^0_\alpha \mid \alpha<\kappa \rangle$.
Clearly, we may assume that $|\mathcal C_\alpha^0|=1$ for all $\alpha\in E\cup S$.

By the same proof of Corollary~\ref{cor-KS93},
there exists a stationary $\Omega\s E$ and a $\p^-(\kappa, \mu, {\sqstarleftup{\Omega}},1,\{\Omega\},2,1)$-sequence
$\vec{\mathcal C}=\langle \mathcal C_\alpha \mid \alpha<\kappa \rangle$ with $|\mathcal C_\alpha|\le|\mathcal C_\alpha^0|$ for all $\alpha<\kappa$.
In particular, $\vec{\mathcal C}$ is a $\p^-(\kappa, \mu, {\sq^*},1,\{\kappa\},\mu,1)$-sequence
with $S\s\{\alpha\in E^\kappa_{>\omega}\mid |\mathcal C_\alpha|<2\}$.
Then, by Lemma~\ref{shiftcofinality},  $\p^-(\kappa,\mu,{\sq^*},1,\{S'\},2)$ holds.
\end{proof}

We conclude this subsection, pointing out a few additional connections between instances of $\p^-(\kappa,\ldots)$ and the existence of nonreflecting stationary subsets of $\kappa$.
\begin{prop}\label{nonreflectcons} Assume $\chi\in\reg(\kappa)$, $\min\{\theta,\sigma\}>0$ and $\mathcal R$ as in Example~\ref{example53}.
\begin{enumerate}
\item Let $\xi<\kappa$.
For any $S\s E^\kappa_{\ge\chi}$, $\p^-_\xi(\kappa,\kappa,{\sqx},1,\{S\},2,1)$ entails the existence of a nonreflecting stationary subset of $S$.
In addition, for any $S\s\kappa$, $\p^-_\xi(\kappa,\kappa,{\sq_\chi},1,\{S\},2,1)$ entails the existence of stationary subset $\Omega\s S$
such that, for every $\alpha\in E^\kappa_{\ge\max\{\chi,\omega_1\}}$, $\Omega\cap\alpha$ is nonstationary in $\alpha$.
\item For any $\sigma\in\reg(\kappa)$, $\p^-(\kappa,2,{\sqleft{\sigma}},1,\{E^\kappa_{\ge\sigma}\},2,\sigma)$ 
entails the existence of a nonreflecting stationary subset of $E^\kappa_\sigma$. 
\item $\p^-_\xi(\kappa,\mu,{^{\Omega}{\mathcal R}},\theta,\{\Omega\},\nu,\sigma)$ implies that $\Omega$ is a nonreflecting 
stationary subset of $\kappa$ and that 
$\p^-_\xi(\kappa,\mu,{^{\Omega}{\mathcal R}},\theta,\{\Omega\},\nu',\sigma)$ holds with $\nu'=2$.
\item If $\kappa=\lambda^+$ for a regular cardinal $\lambda$,
then $\p^-_\lambda(\kappa,\mu,\mathcal R,\theta,\mathcal S,\nu,\sigma)$ 
is equivalent to $\p^-(\kappa,\mu,{^{E^\kappa_\lambda}{\mathcal R}},\theta,\mathcal S,\nu,\sigma)$.
\item If $\kappa=\lambda^+$ for a regular cardinal $\lambda$,
then, for every $\mathcal S\s\mathcal P(E^{\kappa}_\lambda)$,
$\p^-_\lambda(\kappa,\mu,\mathcal R,\allowbreak\theta,\mathcal S,\nu,\sigma)$ 
is equivalent to $\p^-_\lambda(\kappa,\mu,\mathcal R,\theta,\allowbreak\mathcal S,2,\sigma)$.
\item If $\kappa=\lambda^+$ for an infinite cardinal $\lambda$,
then, for every $\theta\ge\lambda$,
$\p^-_\lambda(\kappa,\mu,\mathcal R,\theta,\{\kappa\},\allowbreak\nu,\sigma)$ 
is equivalent to $\p^-_\lambda(\kappa,\mu,\mathcal R,\theta,\allowbreak\{E^\kappa_{\cf(\lambda)}\},2,\sigma)$.
\end{enumerate}
\end{prop}
\begin{proof} (1) The easy argument may be extracted from the proof of $(2)\implies(1)$ of \cite[Corollary~3.4]{paper32}.

(2) The proof is a straight-forward generalization of the proof of 
\cite[Theorem~4.1]{lambie2017aronszajn}.

(3) From $\p^-_\xi(\kappa,{\cdots},\theta,\{\Omega\},\cdot,\sigma)$  with $\min\{\theta,\sigma\}>0$, we infer that $\Omega$ is stationary.
Now, for every $\alpha\in E^\kappa_{>\omega}$, we may pick a club $C\in\mathcal C_\alpha$,
and so by ${^{\Omega}{\mathcal R}}$, we know that the club $\acc(C)$ is disjoint from $\Omega$, so that $\Omega\cap\alpha$ is nonstationary in $\alpha$.
In effect, for every $\alpha\in \Omega$, we may replace $\mathcal C_\alpha$ by some singleton subset of itself,
and then see that $\p^-_\xi(\kappa,\mu,{^{\Omega}{\mathcal R}},\theta,\{\Omega\},2,\sigma)$ holds, as well.

(4) This is obvious.

(5) By Clauses (3) and (4).

(6) By Clause~(3) and \cite[Remarks~3.22]{paper32}.
\end{proof}

\subsection{Hitting on a club}
Let $\ns_\kappa$ denote the nonstationary ideal over $\kappa$,
so that
$\ns^+_\kappa$ forms the collection of all stationary subsets of $\kappa$.
For every stationary $T \subseteq \kappa$,
one denotes ${\ns^+_\kappa \restriction T} := \ns^+_\kappa \cap \mathcal P(T)$.

The following result shows that hitting on a club enables increasing the number of sets being hit simultaneously to the maximal possible value.
It is used in the proof of~\cite[Corollary~3.6]{paper26},
and in the justification of~\cite[Remark~ii.\ following Definition~3.3]{paper26}
and~\cite[Theorem~1.10(2)]{lambie2017aronszajn}.

\begin{lemma}\label{lemma916} Suppose that
$\mathcal S= {\ns^+_\kappa \restriction T}$ for some stationary $T\s\kappa$.
Then
$\p_\xi^-(\kappa,\mu,\mathcal R,1,\mathcal S,\nu,\sigma)$ is equivalent to $\p_\xi^-(\kappa,\mu,\mathcal R,\theta,\mathcal S,\nu,\sigma)$ with $\theta=\kappa$.
\end{lemma}
\begin{proof} We focus on the forward implication.
Fix a sequence $\cvec{C} = \langle \mathcal C_\alpha \mid \alpha<\kappa \rangle$
witnessing $\p_\xi^-(\kappa,\mu,\mathcal R,1,\mathcal S,\nu,\sigma)$.
In particular, for every cofinal $B\s\kappa$,
if we let $G(B)$ denote the set of all $\alpha<\kappa$ such that
$| \mathcal C_\alpha| < \nu$, and for all $C \in \mathcal C_\alpha$,
$$\sup\{ \gamma\in C \mid \suc_\sigma (C \setminus \gamma) \subseteq B \} = \alpha,$$
then, for every $S\in\mathcal S$, $G(B)\cap S$ is stationary. Recalling that $\mathcal S=\{S\s T\mid T\text{ is stationary}\}$, this means that $T\setminus G(B)$ is nonstationary.

\begin{claim} $\cvec{C}$ witnesses
$\p_\xi^-(\kappa,\mu,\mathcal R,\kappa,\mathcal S,\nu,\sigma)$, as well.
\end{claim}
\begin{proof}
Suppose that $\langle B_i\mid i<\kappa\rangle$ is a given sequence of cofinal subsets of $\kappa$.
For every $i<\kappa$, let us fix a club $D_i\s\kappa$ such that $D_i\cap T\s G(B_i)$.
Consider the club  $D:=\acc(\kappa)\cap(\diagonal_{i<\kappa}D_i)$. Let $\alpha\in D\cap T$ be arbitrary. By $\alpha\in D_0\cap T$, we infer $|\mathcal C_\alpha|<\nu$.
Let $C\in\mathcal C_\alpha$ be arbitrary. Given $i<\alpha$, we have $\alpha\in D_i\cap T$, and hence
$$\sup\{ \gamma \in C \mid \suc_\sigma (C \setminus \gamma) \subseteq B_i \} = \alpha,$$
as required.
\end{proof}
This completes the proof.
\end{proof}

\begin{remark}\label{remark-P*}
Recalling Conventions \ref{convention-mu-infty} and \ref{convention-sigma<infty},
the principle $\p^*(T,\xi)$ of \cite[Definition~3.3]{paper26} is nothing but
$\p^-_\xi(\kappa, \infty, {\sq}, 1, {\ns^+_\kappa \restriction T}, 2, {<}\infty)$.
In particular, $\p^*(T,\kappa)$ entails 
$\p^-(\kappa, \kappa, {\sq}, \kappa, \{\kappa\}, 2, 1)$,
which, by the results of Subsection~\ref{subsection:free}, entails the existence of a free $\kappa$-Souslin tree.
This justifies remark~(ii) after Definition 3.3 of \cite{paper26}.
\end{remark}

Consider a regular infinite cardinal $\lambda$ and some stationary set $S \subseteq E^{\lambda^+}_\lambda$.
Similar to the fact that $\diamondsuit(S)$ implies $\p^-_\lambda(\lambda^+, 2, {\sqleft{\lambda}}, 1, \{S\}, 2, \sigma)$
(cf.~\cite[Theorem~5.1(2)]{paper22}),
it is the case that $\diamondsuit^*(S)$ implies $\p^-_\lambda(\lambda^+, 2, {\sqleft{\lambda}}, 1, {\ns^+_{\lambda^+} \restriction S}, 2, \sigma)$.\footnote{For the definition of $\diamondsuit^*(S)$, see \cite[Definition~1.3]{rinot_s01}.}
To prove this, let us first dispose of the following:

\begin{prop}\label{combine-predictions-to-club}
Suppose $\alpha$ is some limit ordinal, and $\mathcal A$ is a collection of $\cf(\alpha)$ many cofinal subsets of $\alpha$.
Then there exists a club $C$ in $\alpha$ of order-type $\cf(\alpha)$ satisfying the following.
For all $A \in \mathcal A$ and $\sigma<\cf(\alpha)$:
$$\sup\{\gamma\in C\mid \suc_\sigma(C\setminus\gamma)\s A\} = \alpha.$$
\end{prop}
\begin{proof} Let $\lambda:=\cf(\alpha)$. Let $\langle \alpha_i\mid i<\lambda\rangle$ be the strictly increasing enumeration of a club in $\alpha$.
Fix a surjection $g:\lambda\rightarrow\mathcal A$ 
such that for all $\sigma<\lambda$ and $A\in\mathcal A$, the set $\{k<\lambda\mid g``(k,k+\sigma)=\{A\}\}$ is cofinal in $\lambda$.
Now, recursively construct a strictly increasing and continuous sequence $\langle \gamma_i \mid i<\lambda\rangle$ such that,
for all $i<\lambda$, $\gamma_i>\alpha_i$ and $\gamma_{i+1}\in g(i)$.
Evidently, $C:=\{\gamma_i\mid i<\lambda\}$ is as sought.
\end{proof}

Clause~(2) of the next theorem is used in the proof of~\cite[Proposition~3.10]{paper26}.

\begin{thm}\label{diamond*-to-P*}
Suppose $\lambda$ is any regular infinite cardinal,
and $S\subseteq E^{\lambda^+}_\lambda$ is a stationary subset such that $\diamondsuit^*(S)$ holds.
Set $\Omega := E^{\lambda^+}_\lambda$. Then:
\begin{enumerate}
\item $\p^-_\lambda(\lambda^+, 2, {\sqleftboth{\Omega}{\lambda}}, 1, {\ns^+_{\lambda^+} \restriction S}, 2, {<}\infty)$ holds.
\item If, in addition, $\lambda^{<\lambda}=\lambda$,
then $\p^-_\lambda(\lambda^+, \infty, {\sqleftup{\Omega}}, 1, {\ns^+_{\lambda^+} \restriction S}, 2, {<}\infty)$ holds.
\end{enumerate}
\end{thm}
\begin{proof} (1) Fix a sequence $\cvec{A} = \langle \mathcal A_\alpha \mid \alpha \in S \rangle$ witnessing $\diamondsuit^*(S)$.
Without loss of generality, we may assume that, for all $\alpha\in S$,
$\mathcal A_\alpha':=\{ A\in\mathcal A_\alpha\mid \sup(A)=\alpha\}$ is of size $\lambda$.
We define a sequence $\langle C_\alpha \mid \alpha<\lambda^+ \rangle$, as follows:

$\br$ For each $\alpha \in \lambda^+ \setminus S$,
pick a closed subset $C_\alpha$ of $\alpha$ with $\sup(C_\alpha)=\sup(\alpha)$ and $\otp(C_\alpha)=\cf(\alpha)$.

$\br$ For each $\alpha \in S$,
let $C_\alpha$ be given by Proposition~\ref{combine-predictions-to-club} when fed with $\mathcal A_\alpha'$.

Now, set $\cvec{C}=\langle \mathcal C_\alpha\mid\alpha<\lambda^+\rangle$
where $\mathcal C_\alpha:=\{C_\alpha\}$ for all $\alpha<\lambda^+$.
Then $\cvec C$ is a $\p_\lambda^-(\lambda^+, 2, {\sqleftboth{\Omega}{\lambda}}, \ldots)$-sequence,
since, for every $\alpha\in\acc(\lambda^+)$, $C_\alpha$ is a club subset of $\alpha$
of order-type $\leq\lambda$, so that $\sqleftboth{\Omega}{\lambda}$-coherence is satisfied vacuously.

To see that  $\cvec{C}$ witnesses $\p_\lambda^-(\lambda^+, 2, {\sqleftboth{\Omega}{\lambda}}, 1, {\ns^+_{\lambda^+} \restriction S}, 2, {<}\infty)$,
consider any cofinal set $A \subseteq\lambda^+$.
By our choice of $\cvec{A}$, we can fix a club $D \subseteq\lambda^+$ such that, for every $\alpha\in D\cap S$,
$A\cap\alpha \in \mathcal A_\alpha$.
Consider the club $E:=\acc^+(A)\cap D$. For any $\alpha \in E \cap S$, we obtain $A\cap\alpha \in \mathcal A_\alpha'$,
so that by our choice of $C_\alpha$ it follows that, for all $\sigma<\lambda$,
$$\sup\{\gamma\in C\mid \suc_\sigma(C\setminus\gamma)\s A\} = \alpha,$$
as sought.

(2) Let $\cvec{C}=\langle \mathcal C_\alpha\mid\alpha<\lambda^+\rangle$ be given by Clause~(1).
Define $\cvec D=\langle \mathcal D_\alpha\mid\alpha<\lambda^+\rangle$ as follows.

$\br$ For all $\alpha\le \lambda$, let $\mathcal D_\alpha:=\{\alpha\}$.

$\br$ For all $\alpha\in E^{\lambda^+}_{<\lambda}\setminus\lambda$,
let $\mathcal D_\alpha$ be the collection of all clubs $d$ in $\alpha$ such that $\min(d)\ge\lambda$
and $|d|<\lambda$.

$\br$ For all $\alpha\in E^{\lambda^+}_\lambda\setminus\{\lambda\}$, let $\mathcal D_\alpha:=\{ C\setminus\lambda\mid C\in\mathcal C_\alpha\}$.

Then $\cvec D$ witnesses $\p^-_\lambda(\lambda^+, \infty, {\sqleftup{\Omega}}, 1, {\ns^+_{\lambda^+} \restriction S}, 2, {<}\infty)$.\qedhere
\end{proof}

\subsection{Improving the coherence}

In this subsection, we show that we can improve the coherence of a sequence witnessing the proxy principle
by inserting all necessary initial segments of the clubs into the relevant collections,
as long as we allow the sequence to grow wide enough.
The main result here is the following:

\begin{cor}\label{get-full-coherence} 
$\p_\xi^-(\kappa, \kappa, {\sq^*}, \theta, \mathcal S,  \kappa,\sigma)$ is equivalent to 
$\p_\xi^-(\kappa, \kappa, {\sq}, \theta, \mathcal S, \kappa,\sigma)$.
\end{cor}
\begin{proof} By Theorem~\ref{get-full-coherence-from-sqleft*} below, using $(\Omega,\chi,\nu):=(\emptyset,0,\kappa)$
and Remark~\ref{remark-nu*}
\end{proof}

The following notion will allow us to state the main result of this subsection in its most general form.

\begin{definition}\label{convention-nu*}
Let $\p^-_\xi(\kappa, \mu, \mathcal R, \theta, \mathcal S,  \nu^*)$
be the assertion obtained by replacing Clause~(1) of Definition~\ref{pminus} by the weaker property:
\begin{enumerate}
\item[($1^*$)] there exists $\mathcal C\in[\mathcal C_\alpha]^{<\nu}$ such that, for all $C\in\mathcal C_\alpha$,
there is $C'\in\mathcal C$ with $\sup(C\symdiff C')<\alpha$.
\end{enumerate}
\end{definition}
\begin{remark}\label{remark-nu*}
Clearly, for any vector of parameters \pvec, we have:\[
\p^-_\xi(\kappa, \mu,\pvec,  \nu)\implies
\p^-_\xi(\kappa, \mu,\pvec,  \nu^*)\implies
\p^-_\xi(\kappa, \mu,\pvec,  \mu).
\]
\end{remark}

Note that setting $(\Omega,\theta,\nu,\sigma) := (\emptyset,1,\kappa,1)$ and $\mathcal S$ a singleton in the following Theorem
provides the justification for \cite[Remark~(i) following Definition~2.1]{paper26}.

\begin{thm}\label{get-full-coherence-from-sqleft*}
Suppose $\kappa$ is $({<}\chi)$-closed and
$\mathcal S \subseteq \mathcal P(E^\kappa_{\geq\chi})$.

Then
$\p_\xi^-(\kappa, \kappa, {\sqstarleftboth{\Omega}{\chi}}, \theta, \mathcal S,  \nu,\sigma)$ is equivalent to
$\p_\xi^-(\kappa, \kappa, {\sqleftup{\Omega}}, \theta, \mathcal S, \nu^*,\sigma)$.
\end{thm}
\begin{proof}
($\impliedby$) Fix a sequence $\langle \mathcal C_\alpha \mid \alpha<\kappa \rangle$
witnessing 
$\p^-_\xi(\kappa, \kappa, {\sqleftup{\Omega}}, \theta, \mathcal S,  \nu^*, \sigma)$.
For every $\alpha\in\acc(\kappa)$,
pick a subset $\mathcal{C}^\bullet_\alpha\s\mathcal{C}_\alpha$ of minimal cardinality 
such that, for all $C\in\mathcal{C}_\alpha$,
there is $C'\in\mathcal{C}^\bullet_\alpha$ with $\sup(C\symdiff C')<\alpha$.

Recalling Convention~\ref{default-outside-acc},
the sequence $\langle \mathcal C^\bullet_\alpha \mid \alpha<\kappa \rangle$ witnesses
$\p^-_\xi(\kappa, \kappa, {\sqstarleftup{\Omega}}, \theta, \mathcal S,\allowbreak\nu, \sigma)$,
as seen by noting the following:
\begin{itemize}
\item For any $\alpha\in\acc(\kappa)$, $C\in\mathcal{C}^\bullet_\alpha$, and $\bar\alpha\in\acc(C)$,
we may pick $D\in\mathcal C_{\bar\alpha}$ with $D\sqleftup{\Omega}C$,
and then pick $D'\in\mathcal C_{\bar\alpha}^\bullet$ with $\sup(D\symdiff D')<\bar\alpha$.
Thus $\sup((C\cap\bar\alpha)\symdiff D')<\bar\alpha$,
so that $D' \sqstarleftup{\Omega} C$, as sought.
\item By Definition~\ref{convention-nu*},
$|\mathcal{C}^\bullet_\alpha|<\nu$ wherever the hitting takes place.
\end{itemize}

($\implies$)
Fix a sequence $\cvec{C} = \langle \mathcal C_\delta \mid \delta<\kappa \rangle$
witnessing $\p_\xi^-(\kappa, \kappa, {\sqstarleftboth{\Omega}{\chi}}, \theta, \mathcal S,  \nu,\sigma)$.
Denote $\Sigma := \bigcup_{\delta\in\acc(\kappa)} \mathcal C_\delta$,
and for every $\alpha<\kappa$,
denote $\mathcal A_\alpha := \{C\cap\alpha \mid C\in \Sigma \}$
and $\mathcal G_\alpha := \{c \in \mathcal A_\alpha \mid \sup(c)=\alpha\}$.
Recalling Convention~\ref{default-outside-acc},
let $\cvec{D} := \langle \mathcal D_\alpha \mid \alpha<\kappa \rangle$
be determined by setting 
$\mathcal D_\alpha := \mathcal G_\alpha$
for every $\alpha\in\acc(\kappa)$
(cf.~{\cite[Notation~2.4]{paper32}}).
We shall show that $\cvec{D}$ witnesses $\p_\xi^-(\kappa, \kappa, {\sqleftup{\Omega}}, \theta, \mathcal S, \nu^*,\sigma)$.

\begin{claim}\label{G-is-A+C}
Suppose $\alpha \in E^\kappa_{\geq\chi}$.
Then every element $c \in \mathcal G_\alpha$
can be written as
$c = (c\cap\beta) \cup (D\setminus\beta)$
for some $D \in \mathcal C_\alpha$ and $\beta<\alpha$, where $c\cap\beta \in \mathcal A_\beta$.
\end{claim}
\begin{proof} Consider any given $c \in \mathcal G_\alpha$.
Fix $C\in \Sigma$ such that $c = C\cap\alpha$.
As $c \in \mathcal G_\alpha$, we infer that $\sup(C\cap\alpha) = \alpha$,
so that $\alpha \in \acc(C) \cup \{\sup(C)\}$.
If $\alpha\in\acc(C)$,
then by $\sqleft{\chi}^*$-coherence and the fact that $\cf(\alpha)\geq\chi$,
we can fix some $D \in \mathcal C_\alpha$ such that $D \sq^* C$.
Otherwise, $\alpha = \sup(C)$, so that in fact $c = C \in \mathcal C_\alpha$;
in this case we set $D := C$.
In either case, set $\beta := \min \{\beta \mid D\setminus\beta \sq C\setminus\beta\}$.
As $D \sq^* C$, it is clear that $\beta$ is an ordinal $<\alpha$.
Furthermore, as $\sup(D) = \sup(c) = \alpha$ and $c=C\cap\alpha$,
it follows that $D\setminus\beta = c\setminus\beta$.
Finally, notice that $c\cap\beta = C\cap\beta \in \mathcal A_{\beta}$,
and it is clear that $c = (c\cap\beta) \cup (D\setminus\beta)$.
\end{proof}

\begin{claim}\label{P-get-sq-wide}
$\cvec{D}$ is a $\p^-_\xi(\kappa, \kappa, {\sqleftup{\Omega}}, \dots)$-sequence
satisfying 
$\mathcal C_\alpha \subseteq \mathcal D_\alpha$ for every $\alpha \in \acc(\kappa)$, and
$\mathcal D_\alpha = \mathcal C_\alpha$ for every $\alpha \in \Omega\cap\acc(\kappa)$.
\end{claim}
\begin{proof} Consider any given $\alpha\in\acc(\kappa)$.
It is clear that $\mathcal C_\alpha \subseteq \mathcal D_\alpha$, 
so that, in particular, $\mathcal D_\alpha \neq\emptyset$.
Now, fix any given $c \in \mathcal D_\alpha$.
By definition of $\mathcal D_\alpha$, 
we can fix some $C \in \Sigma$ such that $c = C\cap\alpha$ and $\sup(c)=\alpha$.
In particular, $c \in \mathcal K(\kappa)$ and $\alpha_c = \alpha$.
Furthermore, $\otp(c)\leq\otp(C)\leq\xi$.
If, in addition, $\alpha\in\Omega$,
then by $\sqstarleftboth{\Omega}{\chi}$-coherence of $\cvec{C}$
we cannot have $\alpha \in \acc(C)$, 
so the only way to have $\sup(C\cap\alpha) = \alpha$ is if in fact $\alpha = \sup(C)$,
meaning that $c = C \in \mathcal C_\alpha$.

Next, we verify $\sqleftup{\Omega}$-coherence:
Consider arbitrary $\alpha \in \acc(\kappa)$, $c \in \mathcal D_\alpha$, and $\bar\alpha \in \acc(c)$.
Pick $C \in \Sigma$ such that $c = C \cap \alpha$.
Then $\bar\alpha\in\acc(C)$, 
so that by $\sqstarleftboth{\Omega}{\chi}$-coherence of $\cvec{C}$ we obtain $\bar\alpha\notin\Omega$.
Clearly, $c\cap\bar\alpha = C\cap\bar\alpha\in\mathcal A_{\bar\alpha}$.
By $\bar\alpha\in\acc(C)$, we moreover obtain $C\cap\bar\alpha \in \mathcal D_{\bar\alpha}$,
and hence $c\cap\bar\alpha\in \mathcal D_{\bar\alpha}$.

It remains to show that $|\mathcal D_\alpha| <\kappa$ for every $\alpha\in\acc(\kappa)$.
In fact, we shall prove, by induction on $\alpha$,
the stronger result that $|\mathcal A_\alpha|<\kappa$ for every ordinal $\alpha<\kappa$.
Thus, fix a given ordinal $\alpha<\kappa$,
and assume that $|\mathcal A_\gamma|<\kappa$ 
for every ordinal $\gamma<\alpha$;
we shall first show that $|\mathcal G_\alpha|<\kappa$,
and then argue further that $|\mathcal A_\alpha|<\kappa$.

To see that $|\mathcal G_\alpha|<\kappa$, we consider several cases:
\begin{itemize}
\item[$\br$] If $\alpha=0$, then $\mathcal G_\alpha = \{\emptyset\}$ is a singleton.
\item[$\br$] If $\alpha$ is a successor ordinal, then $\mathcal G_\alpha$ is empty.

Thus we may assume that $\alpha\in\acc(\kappa)$.

\item[$\br$] Suppose $\alpha\in E^\kappa_{<\chi}$. Set $\eta := \cf(\alpha)$, and
fix an increasing sequence $\langle \gamma_i \mid i<\eta \rangle$ of ordinals converging to $\alpha$.
Define a function
$\varphi : \mathcal G_\alpha \to \prod_{i<\eta} \mathcal A_{\gamma_i}$
by setting $\varphi(c) := \langle c\cap\gamma_i \mid i<\eta \rangle$ for every $c \in \mathcal G_\alpha$.
For every $i<\eta$, we have $|\mathcal A_{\gamma_i}|<\kappa$ by the induction hypothesis,
since $\gamma_i<\alpha$.
Let $\lambda := \sup_{i<\eta} |\mathcal A_{\gamma_i}|$.
We infer from the regularity of $\kappa$ that $\lambda<\kappa$,
and then since $\eta<\chi$, it follows from $({<}\chi)$-closedness of $\kappa$ that
$\left|\im(\varphi)\right| \leq \lambda^\eta <\kappa$.
As $c = \bigcup_{i<\eta} (c\cap\gamma_i)$ for every $c \in \mathcal G_\alpha$, 
it is clear that $\varphi$ is injective.
Altogether, $|\mathcal G_\alpha|<\kappa$, as sought.
\item[$\br$] Finally, suppose $\alpha\in E^\kappa_{\geq\chi}$.
In this case, define a function
$$\varphi:\mathcal G_\alpha\to\bigcup_{\beta<\alpha}\mathcal A_\beta\times\alpha\times\mathcal C_\alpha$$
by setting $\varphi(c) := (c\cap\beta_c, \beta_c, D_c)$,
where the representation $c = (c\cap\beta_c) \cup (D_c\setminus\beta_c)$
is given by Claim~\ref{G-is-A+C}.
As $|\mathcal A_\beta|<\kappa$ for every ordinal $\beta<\alpha$, 
and $|\mathcal C_\alpha|<\kappa$ by our choice of $\cvec{C}$,
it follows from regularity of $\kappa$ that $\left|\im(\varphi)\right|<\kappa$.
Furthermore, it is clear that $\varphi$ is injective.
Altogether, $|\mathcal G_\alpha|<\kappa$, as sought.
\end{itemize}

In all cases we have shown that $|\mathcal G_\alpha|<\kappa$.
Of course, $\mathcal G_\beta \s \mathcal A_\beta$ for every ordinal $\beta<\kappa$,
so that our induction hypothesis implies that 
$| \mathcal G_\beta| <\kappa$ for every ordinal $\beta<\alpha$.
We use these facts as we continue to show that $|\mathcal A_\alpha|<\kappa$.

Notice that every element $c \in \mathcal A_\alpha$ can be written as
$c = (c\cap\beta_c) \cup (c\setminus\beta_c)$,
where $\beta_c$ is some (possibly 0) limit ordinal $\leq\alpha$,
$c\cap\beta_c \in \mathcal G_{\beta_c}$,
and $c\setminus\beta_c$ is finite.
To see this,
consider any given $c \in \mathcal A_\alpha$.
Define $\beta_c := \min\{\beta \mid \left|c\setminus\beta\right| < \aleph_0 \}$.
As $c\subseteq\alpha$, it is clear that $\beta_c$ is an ordinal $\leq\alpha$.
By minimality of $\beta_c$,
it follows that $\beta_c$ is a limit ordinal (possibly 0) and $\sup(c\cap\beta_c) = \beta_c$.
Fix $C\in \Sigma$ such that $c = C\cap\alpha$.
Then $c\cap\beta_c = C\cap\beta_c$, so that, in particular,
$c\cap\beta_c \in \mathcal G_{\beta_c}$.
It is clear that $c = (c\cap\beta_c) \cup (c\setminus\beta_c)$,
and by our definition of $\beta_c$ it is clear that $c\setminus\beta_c$ is finite.

Thus, to show that $|\mathcal A_\alpha| <\kappa$, define a function 
\[
\pi : \mathcal A_\alpha \to 
\bigcup\bigl\{ \mathcal G_\beta \times [\alpha]^{<\aleph_0} \bigm| {\beta \in \acc(\alpha+1) \cup\{0\}} \bigr\}
\]
by setting $\pi(c) := (c\cap\beta_c, c\setminus\beta_c)$,
where the representation $c = (c\cap\beta_c) \cup (c\setminus\beta_c)$ 
is the one given above. 
Since $|\alpha|<\kappa$ and $\kappa$ is infinite,
we infer that $|[\alpha]^{<\aleph_0}| <\kappa$.
Then, by regularity of $\kappa$ and the fact that 
$| \mathcal G_\beta| <\kappa$ for every ordinal $\beta\leq\alpha$,
it follows that $\left|\im(\pi)\right|<\kappa$.
As $c = (c\cap\beta_c) \cup (c\setminus\beta_c)$ for every $c \in \mathcal A_\alpha$, 
it is clear that $\pi$ is injective.
Altogether, $|\mathcal A_\alpha|<\kappa$, as sought.
\end{proof}

\begin{claim} Suppose $\langle B_i \mid i < \theta \rangle$ is a sequence of cofinal subsets of $\kappa$,
and $S \in \mathcal S$. Then there exist stationarily many $\alpha \in S$ such that:
\begin{enumerate}
\item there exists $\mathcal C\in[\mathcal D_\alpha]^{<\nu}$ such that, for all $D\in\mathcal D_\alpha$,
there is $C\in\mathcal C$ with $\sup(D\symdiff C)<\alpha$;
\item for all $D \in \mathcal D_\alpha$ and $i < \min\{\alpha, \theta\}$,
$\sup\{ \gamma\in D \mid \suc_\sigma (D \setminus \gamma) \subseteq B_i \} = \alpha$.
\end{enumerate}
\end{claim}
\begin{proof}
Let $S^*$ denote the set of all $\alpha\in S\cap\acc(\kappa)$ such that $|\mathcal C_\alpha|<\nu$ and, for all $C\in\mathcal C_\alpha$ and $i<\min\{\alpha,\theta\}$, 
$\sup\{ \gamma\in C \mid \suc_\sigma (C \setminus \gamma) \subseteq B_i \} = \alpha$.
By the choice of  $\cvec{C}$, $S^*$ is a stationary subset of $E^\kappa_{\geq\chi}$.
Consider any given $\alpha\in S^*$.
Then, by Claims \ref{P-get-sq-wide} and \ref{G-is-A+C}, $\mathcal C:=\mathcal C_\alpha$ witnesses Clause~(1).
Now, given $D\in\mathcal D_\alpha$ and $i<\min\{\alpha,\theta\}$, we first pick $C\in\mathcal C_\alpha$
such that $\sup(D\symdiff C)<\alpha$. As $\alpha\in S^*$, we know that 
$\sup\{ \gamma\in C \mid \suc_\sigma (C \setminus \gamma) \subseteq B_i \} = \alpha$.
As $\sup(D\symdiff C)<\alpha$, it is also the case that 
$\sup\{ \gamma\in D \mid \suc_\sigma (D \setminus \gamma) \subseteq B_i \} = \alpha$.
\end{proof}
Thus, $\cvec D$ witnesses $\p_\xi^-(\kappa, \kappa, {\sqleftup{\Omega}}, \theta, \mathcal S, \nu^*,\sigma)$.
\end{proof}

\begin{cor} $\p_\xi^-(\kappa, \kappa, {\sqstarleftup{\Omega}}, \theta, \{\Omega\},  \kappa,\sigma)$ is equivalent to 
$\p_\xi^-(\kappa, \kappa, {\sqleftup{\Omega}}, \theta, \{\Omega\}, 2,\sigma)$.
\end{cor}
\begin{proof} Apply Theorem~\ref{get-full-coherence-from-sqleft*} followed by Proposition~\ref{nonreflectcons}(3).
\end{proof}

\subsection{Indexed ladders}

In this subsection, we consider indexed $\p^-_\xi(\kappa, \mu^+, {\sq}, \dots)$-sequences
for infinite cardinals $\mu<\kappa$.
Our proof of Theorem~A will go through the following concept.

\begin{defn}[{\cite[Definition~6.3]{narrow_systems}}]
$\square^{\ind}(\kappa, \mu)$ asserts the existence of a matrix $\langle C_{\alpha, i} \mid \alpha < \kappa, ~ i(\alpha) \leq i < \mu \rangle$ such that for all $\alpha\in\acc(\kappa)$:
\begin{itemize}
\item $i(\alpha) < \mu$, and, for all $i\in[i(\alpha),\mu)$, $C_{\alpha,i}$ is a club in $\alpha$;
\item for all $i\in[i(\alpha),\mu)$ and $\bar\alpha\in \acc(C_{\alpha, i})$, $i\ge i(\bar\alpha)$ and $C_{\bar\alpha,i}\sq C_{\alpha,i}$;
\item $\langle C_{\alpha, i} \mid i(\alpha) \leq i < \mu \rangle$ is $\subseteq$-increasing with $\acc(\alpha)=\bigcup_{i\in[i(\alpha),\mu)}\acc(C_{\alpha,i})$,
\end{itemize}
and such that, for every club $D$ in $\kappa$, there exists $\alpha \in \acc(D)$ such that $D \cap \alpha \neq C_{\alpha, i}$ for all $i\in[i(\alpha),\mu)$.
\end{defn}

Motivated by the preceding definition, we introduce the following
strengthening of $\p^-_\xi(\kappa, \mu^+, {\sq}, \ldots)$ (compare with Definition~\ref{defn-p-sequence-first-stage}).

\begin{defn}\label{indexedP}
We say that $\langle\mathcal C_\alpha\mid\alpha<\kappa\rangle$ is a $\p^-_\xi(\kappa, \mu^{\ind}, {\sq}, \ldots)$-sequence iff
there exists a sequence $\langle i(\alpha)\mid \alpha<\kappa\rangle$ of ordinals in $\mu$, such that, for every $\alpha\in\acc(\kappa)$, all of the following hold:
\begin{itemize}
\item $\mathcal C_\alpha\s\{ C\in \mathcal K(\kappa)\mid \otp(C)\le\xi\ \&\ \alpha_C=\alpha\}$;
\item there exists a canonical enumeration $\langle C_{\alpha,i}\mid i(\alpha)\le i<\mu\rangle$ of $\mathcal C_\alpha$
(possibly with repetition);
\item for all $i\in[i(\alpha),\mu)$ and $\bar\alpha\in\acc({C}_{\alpha,i})$, $i\ge i(\bar\alpha)$ and $C_{\bar\alpha,i}\sq C_{\alpha,i}$;
\item $\langle \acc({C}_{\alpha,i})\mid i(\alpha)\le i<\mu\rangle$
is $\s$-increasing with $\acc(\alpha) = \bigcup_{i\in[i(\alpha),\mu)}\acc(C_{\alpha,i})$.
\end{itemize}
\end{defn}

We will need an adaptation of our wide-club-guessing lemma \cite[Lemma~2.5]{paper29} to the context of indexed sequences.

\begin{lemma}\label{Addingguessing} Suppose that $\aleph_1\le\mu^+<\kappa$, and $\langle\mathcal C_\alpha\mid\alpha<\kappa\rangle$ 
is a $\p^-_\xi(\kappa, \mu^{\ind}, {\sq}, \ldots)$-sequence satisfying that, for every club $D\s\kappa$, there exists $\alpha\in\acc(D)$ with $D\cap\alpha\notin\mathcal C_\alpha$.
Suppose $S\s\kappa$ is stationary.

Then there is a $\p^-_\xi(\kappa, \mu^{\ind}, {\sq}, \ldots)$-sequence $\langle\mathcal C^\bullet_\alpha\mid\alpha<\kappa\rangle$
with the additional property that for every club $E\s\kappa$, there exist stationarily many $\alpha\in S$
such that, for all $C\in\mathcal C^\bullet_\alpha$, $\sup(\nacc(C)\cap E)=\alpha$.
\end{lemma}
\begin{proof} Fix a sequence $\langle i(\alpha)\mid \alpha<\kappa\rangle$ of ordinals in $\mu$ as in Definition~\ref{indexedP}.
For every club $D\s\kappa$, we define a map $\Phi_D : \mathcal K(\kappa) \to \mathcal K(\kappa)$ via
\[\Phi_D(x):=\begin{cases}
x \setminus \sup(D \cap \sup(x)),\\
&\hfill\text{if } \sup(D \cap \sup(x)) < \sup(x); \\
\makebox[150pt][l]{$\{ \sup(D\cap\eta)\mid \eta\in x\ \&\ \eta>\min(D)\}$,} \\
&\hfill\text{if } \sup(D \cap \sup(x)) = \sup(x).
\end{cases}\]
By \cite[Lemma~2.2]{paper29}, $\Phi_D$ is a postprocessing function.
Furthermore, by \cite[Claim~2.5.1]{paper29}, 
we may fix a club $D\s\kappa$ such that, 
for every club $E\s\kappa$, there exists $\alpha\in S$ with $\sup(\nacc(\Phi_D(C))\cap E)=\alpha$ for all $C\in\mathcal C_\alpha$.
Let $\pi:\kappa\leftrightarrow D$ denote the order-preserving bijection.
Now, for every $\alpha\in\acc(\kappa)$, let $j(\alpha):=i(\pi(\alpha))$ and set $$C^\bullet_{\alpha,i}:=\pi^{-1}[\Phi_D(C_{\pi(\alpha),i})]$$
for every $i\in[j(\alpha),\mu)$. Put $\mathcal C_\alpha^\bullet:=\{ C^\bullet_{\alpha,i}\mid j(\alpha)\le i<\mu\}$.
\begin{claim} Let $\alpha\in\acc(\kappa)$ and $i\in[j(\alpha),\mu)$. Then:
\begin{enumerate}
\item $C^\bullet_{\alpha,i}$ is a club in $\alpha$ of order-type $\le\xi$;
\item for all $\bar\alpha\in\acc(C^\bullet_{\alpha,i})$, $i\ge j(\bar\alpha)$ and $C^\bullet_{\bar\alpha,i}\sq C^\bullet_{\alpha,i}$.
\end{enumerate}
\end{claim}
\begin{proof} (1) As $\alpha\in\acc(\kappa)$, $\pi(\alpha)\in \acc(D)$,
and hence the definition of $\Phi_D$ implies that $\Phi_D(C_{\pi(\alpha),i})$ is a club in $\pi(\alpha)$ which is a subset of $D$.
So $C^\bullet_{\alpha,i}=\pi^{-1}[\Phi_D(C_{\pi(\alpha),i})]$ is a club in $\alpha$ of order-type $\otp(C_{\pi(\alpha),i})\le\xi$.

(2) Suppose  $\bar\alpha\in\acc(C^\bullet_{\alpha,i})$.
Then $\bar\alpha\in\acc(\pi^{-1}[\Phi_D(C_{\pi(\alpha),i})])$,
so that $\pi(\bar\alpha)\in\acc(\Phi_D(C_{\pi(\alpha),i}))=\acc(C_{\pi(\alpha),i})\cap\acc(D)$.
In particular, $i\ge i(\pi(\bar\alpha))=j(\bar\alpha)$ and $C_{\pi(\bar\alpha),i}\sq C_{\pi(\alpha),i}$.
As $\Phi_D$ is a postprocessing function, $\Phi_D(C_{\pi(\bar\alpha),i})\sq\Phi_D(C_{\pi(\alpha),i})$, and hence $C_{\bar\alpha,i}^\bullet\sq C_{\alpha,i}^\bullet$.
\end{proof}
\begin{claim} Let $\alpha\in\acc(\kappa)$.
Then $\langle \acc(C^\bullet_{\alpha,i})\mid j(\alpha)\le i<\mu\rangle$
is $\s$-increasing with $\acc(\alpha) = \bigcup_{i\in[j(\alpha),\mu)}\acc(C^\bullet_{\alpha,i})$.
\end{claim}
\begin{proof} It is clear that for any two clubs $C,C'$ in $\pi(\alpha)$,
$C\s C'\implies \Phi_D(C)\s\Phi_D(C')$. 
Consequently,  $\langle \acc(C^\bullet_{\alpha,i})\mid j(\alpha)\le i<\mu\rangle$ is $\s$-increasing.
Finally, for any $\bar\alpha\in\acc(\alpha)$, as $\pi(\bar\alpha)\in\acc(\pi(\alpha))$,
we may find a large enough $i\ge i(\pi(\alpha))=j(\alpha)$ such that $\pi(\bar\alpha)\in\acc(C_{\pi(\alpha),i})$.
As $\pi(\bar\alpha)\in\acc(D)$, altogether, 
$\pi(\bar\alpha)\in\acc(C_{\pi(\alpha),i})\cap\acc(D)=\acc(\Phi_D(C_{\pi(\alpha),i}))$,
where the last equality is due to the definition of $\Phi_D$ and fact that $\pi(\alpha)\in\acc(D)$.
Finally, as $\Phi_D(C_{\pi(\alpha),i})\s\im(\pi)$, it follows that $\bar\alpha$ is an accumulation point of $\pi^{-1}[\Phi_D(C_{\pi(\alpha),i})]=C^\bullet_{\alpha,i}$.
\end{proof}
\begin{claim} Let $E\s\kappa$ be a club. Then there exist stationarily many $\alpha\in S$
such that, for all $C\in\mathcal C^\bullet_\alpha$, $\sup(\nacc(C)\cap E)=\alpha$.
\end{claim}
\begin{proof} Let $B$ be an arbitrary club. We need to find $\alpha\in B\cap S$
such that, for all $C\in\mathcal C^\bullet_\alpha$, $\sup(\nacc(C)\cap E)=\alpha$.

As $E^\bullet:=\{\beta\in B\cap E\mid \pi(\beta)=\beta\}$ is a club,
by the choice of $D$, we may now fix $\alpha\in S$ with $\sup(\nacc(\Phi_D(C))\cap E^\bullet)=\alpha$ for all $C\in\mathcal C_\alpha$.
In particular, $\alpha\in\acc(E^\bullet)$, so that $\alpha\in B\cap S$ and $\pi(\alpha)=\alpha$.
Consequently, $\mathcal C_\alpha^\bullet=\{ \pi^{-1}[\Phi_D(C)]\mid C\in\mathcal C_\alpha\}$.
Let $C\in\mathcal C_\alpha$ be arbitrary. 
Put $c:=\nacc(\Phi_D(C))\cap E^\bullet$. Then $c$ is a subset of $\nacc(\Phi_D(C))\cap D\cap E$,
consisting of fixed-points of $\pi$.
As $\sup(c)=\alpha$, it follows that $\pi^{-1}[c]=c$ is a cofinal subset of  $\nacc(\pi^{-1}[\Phi_D(C)])$.
So, $c \s \nacc(\pi^{-1}[\Phi_D(C)])\cap E$ and hence the latter is cofinal in $\alpha$, as sought.
\end{proof}
This completes the proof.
\end{proof}

If $\cvec{C} = \langle \mathcal{C}_\alpha \mid \alpha<\kappa \rangle$ is a $\p^-(\kappa, \mu^{\ind}, {\sq}, \dots)$-sequence
such that $S:=\{ \alpha<\kappa\mid |\mathcal{C}_\alpha|<\aleph_0 \}$ is stationary,
then we may find some $i<\mu$ for which $A:=\{\alpha\in S\cap\acc(\kappa)\mid C_{\alpha,i}=\max(\mathcal C_\alpha,{\s})\}$ is stationary.
For every $\alpha\in A$, $A\cap\alpha\s\acc(\alpha)\s \acc(C_{\alpha,i})$, and hence $\langle C_{\alpha,i}\mid \alpha\in A\rangle$ is an $\sq$-chain
converging to some club $D$ that forms a thread through $\cvec C$ (cf.~Remark~\ref{remarks-about-xbox}(3)).
In particular, $\p^-(\kappa, \mu^{\ind}, {\sq}, 1, \{\kappa\}, \nu)$ is inconsistent for $\nu\leq\aleph_0$.
By contrast, we have the following indexed variant of~\cite[Theorem~3.5]{paper37}:\footnote{The proofs are different, since here there is no guarantee that $\diamondsuit^*(E^{\lambda^+}_{<\lambda})$ holds.} 
\begin{thm}\label{mainindex} Suppose $\square^{\ind}(\kappa, \mu)$ holds for an infinite cardinal $\mu<\kappa$.

If $\kappa=\lambda^+=2^\lambda$, then for any regular cardinal $\chi\ge\mu$ such that $\lambda^\chi=\lambda$,
$\p^-(\kappa,\mu^{\ind},{\sq},1,\{E^\kappa_\chi\},\mu^+)$ holds.
\end{thm}
\begin{proof} Suppose $\kappa=\lambda^+$ and $\chi\ge\mu$ is a regular cardinal such that $\lambda^\chi=\lambda$.
In particular, $\chi<\lambda$, so that $\mu^+\le\chi^+<\kappa$. Now, as $\square^{\ind}(\kappa,\mu)$ holds,
the hypotheses of Lemma~\ref{Addingguessing} are satisfied with $S := E^\kappa_\chi$.
Consequently, we may fix a $\p^-(\kappa, \mu^{\ind}, {\sq}, \ldots)$-sequence $\cvec C=\langle\mathcal C_\alpha\mid\alpha<\kappa\rangle$
with the additional property that, for every club $D\s\kappa$, there exist stationarily many $\alpha\in E^\kappa_\chi$
such that $\sup(\nacc(C)\cap D)=\alpha$ for all $C\in\mathcal C_\alpha$.

Since $\lambda^\chi=\lambda$, the Engelking--Kar{\l}owicz Theorem provides a sequence $\langle f_j\mid j<\lambda\rangle$ 
of functions from $\lambda^+$ to $\lambda$ with the property that, for every $z\in[\lambda^+]^\chi$ and every function $f:z\rightarrow\lambda$,
there exists $j<\lambda$ with $f\s f_j$.
Consequently, we may fix a sequence $\langle f_j\mid j<\lambda\rangle$ 
of functions from $\lambda^+$ to $\lambda^+$ with the property that, for every $z\in[\lambda^+]^\chi$ and every \emph{regressive} function $f:z\rightarrow\lambda^+$,
there exists $j<\lambda$ with $f\s f_j$.

Suppose that $2^\lambda=\lambda^+$ and fix a bijection $\pi:\lambda^+\leftrightarrow{}^\lambda\lambda^+$.
For every $j<\lambda$, define a triangular array $\mathfrak x^j=\langle x^j_{\gamma,\beta}\mid \gamma<\beta<\kappa\rangle$ via:
$$x^j_{\gamma,\beta}:=
\begin{cases}\{\pi(f_j(\beta))(j),\beta\} \cap (\gamma,\beta],&\text{if }\beta\in\acc(\kappa);\\
\{\beta\},&\text{otherwise}.
\end{cases}$$
and then consider the corresponding $\acc$-preserving postprocessing function $\Phi_{\mathfrak x^j}$ given by Lemma~\ref{pp-from-micro}.
Denote $\mathcal C_\alpha^j:=\{\Phi_{\mathfrak x^j}(C)\mid C\in\mathcal C_\alpha\}$,
so that, for every $j<\lambda$, $\cvec C^j:=\langle \mathcal C_\alpha^j\mid \alpha<\kappa\rangle$ is yet again a 
$\p^-(\kappa, \mu^{\ind}, {\sq}, \ldots)$-sequence.

\begin{claim} There is $j<\lambda$ such that $\cvec C^j$ witnesses $\p^-(\kappa,\mu^{\ind},{\sq},1,\{E^\kappa_\chi\},\mu^+,1)$.
\end{claim}
\begin{proof} Suppose not. Then, for every $j<\lambda$,
there exists a cofinal subset $B_j\s\kappa$ such that, for all $\alpha\in E^\kappa_\chi$,
there exists $C\in\mathcal C_\alpha^j$ for which $\sup(\nacc(C)\cap B_j)<\alpha$.
Now, let $g:\kappa\rightarrow\kappa$ denote the unique function satisfying,
for every $\gamma<\kappa$:
$$g(\gamma)=\delta\iff\bigwedge_{j<\lambda}\bigl(\pi(\delta)(j)=\min(B_j\setminus(\gamma+1))\bigr).$$
Fix a club $E\s\acc(\kappa)$ such that, for all $\beta\in E$ and $\gamma<\beta$,
$$\sup\{g(\gamma),\min(B_j\setminus(\gamma+1))\mid j<\lambda\}<\beta.$$

By the choice of $\cvec C$, let us now fix $\alpha\in E^\kappa_\chi$ such that $\sup(\nacc(C)\cap E)=\alpha$ for all $C\in\mathcal C_\alpha$.
As $|\mathcal C_\alpha|\le\chi=\cf(\alpha)$, we may find a sequence $\langle z_C\mid C\in\mathcal C_\alpha\rangle$ of pairwise disjoint sets
such that, for all $C\in\mathcal C_\alpha$, $z_C$ is a cofinal subset of $\nacc(C)\cap E\setminus\{\min(C)\}$ with order-type $\chi$.
Now, let $z:=\biguplus_{C\in\mathcal C_\alpha}z_C$. Define a function $f:z\rightarrow\lambda^+$ as follows.
For all $\beta\in z$, find the unique $C\in\mathcal C_\alpha$ such that $\beta\in z_C$
and then let $f(\beta):=g(\gamma)$, for $\gamma:=\sup(C\cap\beta)$.
Note that as $\beta\in z\s \nacc(C)\cap E\setminus\{\min(C)\}$, $\gamma<\beta$ and hence also $f(\beta)<\beta$.
That is, $f$ is regressive.

Pick $j<\lambda$ such that $f\s f_j$.
By the choice of $B_j$, let us now pick $C^*\in\mathcal C_\alpha^j$ such that $\epsilon:=\sup(\nacc(C^*)\cap B_j)$ is smaller than $\alpha$.
Find $C\in\mathcal C_\alpha$ such that $C^*=\Phi_{\mathfrak x^j}(C)$.
Pick a large enough $\beta\in z_C$ for which $\gamma:=\sup(C\cap\beta)$ is greater than $\epsilon$.
Let $\eta:=\pi(f_j(\beta))(j)$.
As $\gamma\in\beta\in E$, we infer that 
$$\eta=\pi(f_j(\beta))(j)=\pi(g(\gamma))(j)=\min(B_j\setminus(\gamma+1))<\beta,$$
and hence 
$x^j_{\gamma,\beta}=\{\eta,\beta\}$. Recalling the definition of $\Phi_{\mathfrak x^j}$ from Lemma~\ref{pp-from-micro},
we see that $\eta\in\nacc(C^*)$, contradicting the fact that $\eta\in B_j\setminus(\epsilon+1)$.
\end{proof}

As $2^\lambda=\lambda^+$ and $\lambda^{\aleph_0}=\lambda$, \cite[Lemma~2.1]{MR485361} entails $\diamondsuit(\lambda^+)$.
So, by Fact~\ref{clubvsdiamond}, $\clubsuit(\kappa)$ holds.
Then, by the proof of Theorem~\ref{get-sigma-finite}(1), $\p^-(\kappa,\mu^{\ind},{\sq},1,\{E^\kappa_\chi\},\mu^+)$ holds, as well.
\end{proof}

\begin{fact}[{\cite[Theorem~3.4]{lh_lucke}}]\label{indexingit}
If $\square(\kappa)$ holds, then so does $\square^{\ind}(\kappa, \mu)$ for every $\mu\in\reg(\kappa)$.
\end{fact}

\begin{cor}\label{corindex} Suppose that $\kappa=\lambda^+=2^\lambda$ and $\square(\kappa)$ holds.
For every pair $\mu\le\chi$ of infinite regular cardinals with $\lambda^\chi=\lambda$, $\p^-(\kappa,\mu^{\ind},{\sq},1,\{E^\kappa_\chi\},\mu^+)$ holds.\qed
\end{cor}

\section{Departing from \texorpdfstring{$\diamondsuit$}{diamond}}\label{xboxdiamond}

As seen in Section~\ref{rightway},
instances of $\p^-(\kappa,\ldots)$ together with $\diamondsuit(\kappa)$ suffice for the construction of a $\kappa$-Souslin tree. 
For this, in \cite[Definition~1.6]{paper22}, we defined the principle $\p(\kappa, \mu, \mathcal R, \theta,\mathcal S,  \nu,\sigma)$
to assert both $\p^-(\kappa, \mu, \mathcal R, \theta, \mathcal S,  \nu,\sigma)$ and $\diamondsuit(\kappa)$,
and then, in that paper as well as other papers in this project, we presented a gallery of constructions of $\kappa$-Souslin trees having various additional features
as applications of the principle $\p(\kappa,\ldots)$.

The goal of this section is to present a principle $\p^\bullet(\kappa,\ldots)$
which is strong enough to still allow all of the said constructions, and weak enough to not rely on $\diamondsuit$. 
Note, however, that assuming $\kappa^{<\kappa}=\kappa$, $\p^-(\kappa,\infty,{\sq_\chi^*},1,\{\kappa\},\kappa,\sigma)$ with $\sigma\ge\omega$ implies $\diamondsuit(\kappa)$,
hence, here one should really only focus on the case $\sigma={<}\omega$.\footnote{For constructions of $\kappa$-Souslin trees that rely on instances with $\sigma=\omega$, see \cite{rinot20}.}
For this, we shall adopt Convention~\ref{omitsigma} at the outset, and will always omit the mention of $\sigma$.
We shall establish the following, which is the special case $\mathcal{R} := {\sq}$ of Corollary~\ref{Pbullet-equiv}:

\begin{cor}  For $\theta>0$, $\p_\xi^\bullet(\kappa, \mu, {\sq}, \theta, \mathcal S,  \nu)$ is equivalent to
$\p_\xi^-(\kappa, \mu, {\sq}, \theta, \mathcal S,  \nu)\wedge(\kappa^{<\kappa}=\kappa)$.\qed
\end{cor}

\begin{definition} Let $\mathcal F(\kappa):=\bigcup_{x\in\mathcal K(\kappa)}{}^xH_\kappa$ denote the collection of all functions from an element of $\mathcal K(\kappa)$ to $H_\kappa$.

For each $C\in\mathcal F(\kappa)$, denote $\dm{C}:=\dom(C)$ and $\alpha_C:=\sup(\dm{C})$.
\end{definition}
\begin{example} For any $C\in\mathcal F(\kappa)$, $\dm C$ is an element of $\mathcal K(\kappa)$.
Going in the other direction,
for every  sequence  $\langle A_\beta\mid\beta<\kappa\rangle$ of elements of $H_\kappa$,
for any $x\in\mathcal K(\kappa)$, $C_x:=\langle A_\beta\mid \beta\in x\rangle$ is an element of $\mathcal F(\kappa)$.
\end{example}

For a binary relation $\mathcal R$  over $\mathcal F(\kappa)$, and a nonempty collection $\mathcal S$ of stationary subsets of $\kappa$,
we shall define a principle $\p^\bullet_\xi(\kappa, \mu, \mathcal R, \theta, \mathcal S,  \nu)$  in two stages.
In the first stage, we focus on the first four parameters.

\begin{defn} We say that $\langle\mathcal C_\alpha\mid\alpha<\kappa\rangle$ is a $\p^\bullet_\xi(\kappa, \mu, \mathcal R, \ldots)$-sequence iff, for every $\alpha\in\acc(\kappa)$, all of the following hold:
\begin{itemize}
\item $\mathcal C_\alpha\s\{ C\in \mathcal F(\kappa)\mid \otp(\dm{C})\le\xi\ \&\ \alpha_C=\alpha\}$;
\item $0<|\mathcal C_\alpha|<\mu$;
\item for all $C \in \mathcal C_\alpha$ and $\bar\alpha \in \acc(\dm{C})$, there exists $D \in \mathcal C_{\bar\alpha}$ with $D \mathrel{\mathcal R} C$.
\end{itemize}
\end{defn}
\begin{conv}
We shall always assume
that $\mathcal C_0:=\{\emptyset\}$ and $\mathcal C_{\alpha+1}:=\{\{(\alpha,\emptyset)\}\}$ for all $\alpha<\kappa$.
Likewise, whenever we construct a  $\p^\bullet_\xi(\kappa, \mu, \mathcal R, \ldots)$-sequence $\langle \mathcal D_\alpha\mid\alpha<\kappa\rangle$,
we shall never bother to define $\mathcal D_0$ and $\mathcal D_{\alpha+1}$ for $\alpha<\kappa$.
We also adopt Conventions \ref{conventionxi}, \ref{convention-mu-infty}, and \ref{conv<theta}.
\end{conv}

\begin{example}\label{example53b}
The binary relations over $\mathcal F(\kappa)$ that fit as the parameter $\mathcal R$
should be understood as \emph{coherence} relations.
The basic example is the \emph{end-extension} relation, $\sq$,
where, for $C,D\in\mathcal F(\kappa)$, we define $C \sqsubseteq D$ iff  $C = D \restriction \alpha_C$.
More nuanced binary relations over $\mathcal F(\kappa)$ are obtained by modifying the $\sq$ relation as follows:
\begin{itemize}
\item We define $C \sq^* D$ iff there exists $\gamma < \alpha_C$ such that ${C\restriction(\dm{C}\setminus\gamma)} \sq {D\restriction(\dm{D}\setminus\gamma)}$;
\item For $\mathcal R \in \{ {\sq}, {\sq^*} \}$, we define
$C \mathrel{_{\chi}{\mathcal R}}D$ iff ((${C}\mathrel{\mathcal{R}}{D}$) or ($\cf(\alpha_C)<\chi$));
\item For $\mathcal{R} \in \{ {\sq}, {\sq^*} \}$,
we define $C \mathrel{\mathcal{R}_\chi} D$ iff ((${C} \mathrel{\mathcal{R}} {D}$) or 
($\otp(\dm D)<\chi$ and $\nacc(\dm D)$ consists only of successor ordinals));
\item For any binary relation $\mathcal R$ over $\mathcal F(\kappa)$ and any class $\Omega\s\on$, we define
$C\mathrel{^{\Omega}{\mathcal R}}D$ iff ((${C}\mathrel{\mathcal{R}}{D}$) and ($\alpha_C\notin\Omega$)).
\end{itemize}
\end{example}

The principle $\p^-(\kappa,\ldots)$ of the previous section (Definition~\ref{pminus}) dealt with hitting of arbitrary cofinal subsets of $\kappa$.
The new principle $\p^\bullet(\kappa,\ldots)$ focuses on hitting sets of the sort arising by Proposition~\ref{motivate} and by the following strengthening of Fact~\ref{def_Diamond_H_kappa}.
\begin{fact}[{\cite[Lemma~2.2]{paper22}}]\label{def_Diamond_H_kappa2}
$\diamondsuit(\kappa)$ is equivalent to the existence of a sequence 
$\langle  A_\beta \mid \beta < \kappa \rangle$ of elements of $H_\kappa$
and a partition $\langle B_\iota\mid\iota<\kappa\rangle$ of $\kappa$,
such that, for every subset $\Omega\subseteq H_\kappa$, every parameter $p\in H_{\kappa^{+}}$, and every $\iota<\kappa$, the following set is cofinal in $\kappa$:
$$B_\iota(\Omega,p):=\{ \beta\in B_\iota\mid \exists \mathcal M\prec H_{\kappa^+}(\mathcal M\cap\Omega= A_\beta\ \&\ p\in\mathcal M\ \&\ \mathcal M\cap\kappa=\beta)\}.$$
\end{fact}

As the reader by now probably expects, the cofinal sets considered by $\p^\bullet(\kappa,\ldots)$ take the following form.

\begin{definition}\label{defbsets} Given $B\s\kappa$, $\Omega\s H_\kappa$, $p\in H_{\kappa^+}$ and $C\in\mathcal F(\kappa)$,
we let $B(\Omega,p,C)$ denote the set of all $\beta\in B\cap\dm{C}$ such that there exists an elementary submodel $\mathcal M\prec H_{\kappa^+}$ satisfying:
\begin{itemize}
\item $p\in\mathcal M$;
\item $\mathcal M\cap\kappa=\beta$;
\item $\mathcal M\cap\Omega=C(\beta)$.
\end{itemize} 
\end{definition}

We now arrive at the second stage of the definition of this principle.

\begin{definition}\label{pbullet}
$\p_\xi^\bullet(\kappa, \mu, \mathcal R, \theta, \mathcal S,  \nu)$
asserts the existence 
of a $\p_\xi^\bullet(\kappa, \mu, \mathcal R, \ldots)$-sequence $\langle \mathcal C_\alpha \mid \alpha < \kappa \rangle$
and a partition $\langle B_\iota\mid \iota<\kappa\rangle$ of $\kappa$, satisfying the following.

For every sequence $\langle (\Omega_i,p_i,\iota_i) \mid i < \theta \rangle$ of elements of $\mathcal P(H_\kappa)\times H_{\kappa^+}\times\kappa$,
every $S \in \mathcal S$, and every $n<\omega$, there exist stationarily many $\alpha \in S$ such that:
\begin{enumerate}
\item $| \mathcal C_\alpha| < \nu$; and
\item for all $C \in \mathcal C_\alpha$ and $i < \min\{\alpha, \theta\}$,
\begin{equation}\label{hittingeqn2}\tag{$\star\star$}
\sup\{ \gamma\in \dm{C} \mid \suc_n (\dm{C} \setminus \gamma) \subseteq B_{\iota_i}(\Omega_i,p_i,C) \} = \alpha.
\end{equation}
\end{enumerate}
\end{definition}

Before we prove the main result of this section, let us point out that $\p^\bullet(\kappa,\ldots)$ is indeed a consequence of $\p(\kappa,\ldots)$.
\begin{prop}\label{P->Pbullet} Suppose that $\p_\xi^-(\kappa, \mu, \mathcal R, \theta,\mathcal S,  \nu)+\diamondsuit(\kappa)$ holds,
with $\mathcal R$ from Example~\ref{example53b}.
Then so does $\p_\xi^\bullet(\kappa, \mu, \mathcal R, \theta, \mathcal S,  \nu)$.
\end{prop}
\begin{proof} Recalling Convention~\ref{omitsigma}, fix a sequence $\langle\mathcal C_\alpha\mid\alpha<\kappa\rangle$ witnessing $\p_\xi^-(\kappa, \mu, \mathcal R,\allowbreak\theta, \mathcal S,  \nu, {<}\omega)$.
As $\diamondsuit(\kappa)$ holds, let $\langle  A_\beta \mid \beta < \kappa \rangle$ and $\langle B_\iota\mid\iota<\kappa\rangle$ be given by Fact~\ref{def_Diamond_H_kappa2}. 

For every $x\in\mathcal K(\kappa)$, let $D_x:=\langle A_\beta\mid \beta\in x\rangle$.
Then, for every $\alpha\in\acc(\kappa)$, let $\mathcal D_\alpha:=\{ D_x\mid x\in\mathcal C_\alpha\}$.
Evidently, $\langle\mathcal D_\alpha\mid\alpha<\kappa\rangle$ and $\langle B_\iota\mid\iota<\kappa\rangle$ together witness $\p_\xi^\bullet(\kappa, \mu, \mathcal R, \theta, \mathcal S,  \nu)$.
\end{proof}

\begin{prop}\label{Pbullet->CH} If $\p^\bullet_\xi(\kappa,\mu,\mathcal{R}, \theta, \mathcal{S}, \nu)$ holds with $\theta>0$, then $\kappa^{<\kappa}=\kappa$.
\end{prop}
\begin{proof} Let $\cvec{C} = \langle \mathcal C_\alpha \mid \alpha<\kappa \rangle$ be a sequence that,
together with some partition $\langle B_\iota \mid \iota<\kappa \rangle$ of $\kappa$, witnesses 
$\p^\bullet_\xi(\kappa,\mu,\mathcal{R}, \theta, \mathcal{S}, \nu)$ with $\theta>0$.
\begin{claim}
$H_\kappa = \bigcup_{\alpha\in\acc(\kappa)} \bigcup_{C\in\mathcal{C}_\alpha}\im(C).$
\end{claim}
\begin{proof} We focus on the nontrivial inclusion.
Let $A \in H_\kappa$.
Set $\Omega_0 := p_0 := A$, $\iota_0 := 0$, and $n:=1$.
Then by the hypothesis on the sequence $\cvec{C}$,
we can choose some $\alpha\in\acc(\kappa)$ such that for all $C\in\mathcal{C}_\alpha$,
Equation~\eqref{hittingeqn2} holds with $i:=0$.
Choose some $C\in\mathcal{C}_\alpha$ and some $\beta\in\nacc(\dm{C})\cap B_0(A,A,C)$.
Then we can fix an elementary submodel $\mathcal{M}\prec H_{\kappa^+}$ such that
$A\in\mathcal{M}$, $\mathcal{M}\cap\kappa=\beta$, and $\mathcal{M}\cap A = C(\beta)$.
But $|A|<\kappa$ and $|A|\in\mathcal{M}$ by elementarity, so that $|A| \in \mathcal{M}\cap\kappa = \beta$,
and we infer that $A \s \mathcal{M}$.
Thus, $A = A\cap\mathcal{M} = C(\beta)\in\im(C)$.
\end{proof}

In particular, $|H_\kappa| =  \kappa$.
\end{proof}

\begin{lemma}\label{Pbullet->Pminus} Suppose $\mathcal{R}$ is taken from Example~\ref{example53b}.
Then $\p_\xi^\bullet(\kappa, \mu, \mathcal R, \theta, \mathcal S,  \nu)$ implies 
$\p_\xi^-(\kappa, \mu, \mathcal R, \theta, \mathcal S,  \nu)$.
\end{lemma}
\begin{proof} For every $C\in\mathcal{F}(\kappa)$, let $x_C := \im(g_C)$, where the function $g_C : \dm{C}\to\alpha_C$ is defined by setting
\[
g_C(\beta) := \begin{cases}
\min\Bigl(\bigl(C(\beta)\cup\{\beta\}\bigr)\cap\bigl(\sup(\dm{C}\cap\beta),\beta\bigr]\Bigr), 
&\text{if } \beta\in\nacc(\dm{C})\cap\acc(\kappa); \\
\beta, &\text{if } \beta\in\acc(\dm{C})\cup\nacc(\kappa).
\end{cases}
\]

\begin{claim}\label{coherence-in-Pbullet->Pminus}
For all $C, C' \in \mathcal{F}(\kappa)$:
\begin{enumerate}
\item $x_C$ is a club in $\alpha_C$ with $\acc(x_C) = \acc(\dm{C})$.
\item If $C \mathrel{\mathcal{R}} C'$ then $x_C \mathrel{\mathcal{R}} x_{C'}$.
\end{enumerate}
\end{claim}
\begin{proof} Left to the reader (cf.~Lemma~\ref{pp-from-micro}).
\end{proof}

Fix a sequence $\cvec{C} = \langle \mathcal{C}_\alpha \mid \alpha<\kappa \rangle$ and a partition
$\vec{B} = \langle B_\iota \mid \iota<\kappa \rangle$ of $\kappa$
together witnessing $\p_\xi^\bullet(\kappa, \mu, \mathcal R, \theta, \mathcal S,  \nu)$.
Then, for every $\alpha\in\acc(\kappa)$, let $\mathcal{D}_\alpha := \{x_C \mid C \in \mathcal{C}_\alpha\}$.
It follows from Claim~\ref{coherence-in-Pbullet->Pminus} that
$\cvec{D} := \langle \mathcal{D}_\alpha \mid \alpha<\kappa \rangle$ is a 
$\p_\xi^-(\kappa, \mu, \mathcal R, \dots)$-sequence.
To show that $\cvec{D}$ witnesses $\p_\xi^-(\kappa, \mu, \mathcal R, \theta, \mathcal S,  \nu)$,
we shall now verify that it satisfies the hitting feature of Definition~\ref{pminus},
recalling Convention~\ref{omitsigma}.

\begin{claim} Suppose $S\in\mathcal{S}$, $\langle A_i \mid i<\theta \rangle$ is a sequence of cofinal subsets of $\kappa$, and $n<\omega$.
Then there exist stationarily many $\alpha \in S$ such that $|\mathcal{D}_\alpha|<\nu$ and, for all $x \in \mathcal{D}_\alpha$ and all $i<\min\{\alpha,\theta\}$,
$\sup\{ \delta\in x \mid \suc_n (x \setminus \delta) \subseteq A_i \} = \alpha$.
\end{claim}
\begin{proof} For every $i<\theta$, let $\Omega_i := p_i := A_i$ and $\iota_i := 0$.
By the choice of $\cvec{C}$ and $\vec{B}$, there are stationarily many $\alpha \in S\cap\acc(\kappa)$ such that
$|\mathcal{C}_\alpha|<\nu$ and, for all $C \in \mathcal{C}_\alpha$ and all $i<\min\{\alpha,\theta\}$,
$\sup\{ \gamma\in \dm{C} \mid \suc_{n+1} (\dm{C} \setminus \gamma) \subseteq B_{0}(A_i,A_i,C) \} = \alpha$.
Consider any such $\alpha$.
Clearly, $|\mathcal{D}_\alpha| \leq |\mathcal{C}_\alpha| <\nu$.
Now, consider any given $x \in \mathcal{D}_\alpha$, $i<\min\{\alpha,\theta\}$, and $\varepsilon<\alpha$.
By definition of $\mathcal{D}_\alpha$, we can fix some $C \in \mathcal{C}_\alpha$ such that $x = x_C$.
By our choice of $\alpha$, we can fix $\gamma \in \dm{C}$ with $\varepsilon<\gamma<\alpha$ such that
$\suc_{n+1} (\dm{C} \setminus \gamma) \subseteq B_{0}(A_i,A_i,C)$.

Consider each $j<n+1$.
Set $\beta_j := (\dm{C}\setminus\gamma)(j+1)$.
Then $\beta_j \in B_0(A_i,A_i,C)$, meaning that we can fix an elementary submodel
$\mathcal{M}_j \prec H_{\kappa^+}$ satisfying $A_i \in \mathcal{M}_j$, $\mathcal{M}_j\cap\kappa = \beta_j$,
and $\mathcal{M}_j \cap A_i = C(\beta_j)$.
As $A_i$ is a cofinal subset of $\kappa$, it follows by elementarity that
$\mathcal{M}_j \models ``A_i \text{ is a cofinal subset of }\kappa"$,
so that in fact $C(\beta_j)$ is a cofinal subset of $\beta_j$.
Of course, $\beta_j\in\acc(\kappa)$.
Thus, as $\beta_j \in \nacc(\dm{C})$,
we infer from the definition of $g_C$ that $g_C(\beta_j) \in C(\beta_j) \s A_i$.

It is clear that $\gamma, \beta_0, \beta_1, \dots, \beta_n$ are n+2 consecutive points of $\dm{C}$.
Altogether, we infer from the definition of $x_C$ that 
$g_C(\beta_0), g_C(\beta_1), \dots, g_C(\beta_n)$ are n+1 consecutive points of $x_C$, all above $\gamma$.
Letting $\delta := g_C(\beta_0)$, we obtain $\varepsilon<\gamma<\delta$, $\delta \in x$, and
$\suc_n(x\setminus\delta)\s A_i$, as sought.
\end{proof}

It follows that $\cvec{D}$ is as sought.
\end{proof}

We now improve Proposition~\ref{P->Pbullet} in two ways. First, we reduce $\diamondsuit(\kappa)$ down to $\kappa^{<\kappa}=\kappa$.
Second, and more surprisingly, we let our reader choose the partition of~$\kappa$.

\begin{thm}\label{mainpbullet} Suppose that $\kappa^{<\kappa}=\kappa$ and $\p_\xi^-(\kappa, \mu, \mathcal R, \theta, \mathcal S,  \nu)$ holds 
with $\mathcal R$ from Example~\ref{example53b}.
Let $\vec B=\langle B_\iota\mid \iota<\kappa\rangle$ be a given partition of $\kappa$ into stationary sets.
Then there exists a sequence $\langle\mathcal C_\alpha\mid\alpha<\kappa\rangle$ that, together with $\vec B$, witnesses $\p_\xi^\bullet(\kappa, \mu, \mathcal R, \theta, \mathcal S,  \nu)$.
\end{thm}
\begin{proof} Using $\kappa^{<\kappa}=\kappa$, let $\lhd$ be some well-ordering of $H_\kappa$ of order-type $\kappa$. 
Also, fix a sequence $\vec A=\langle A_\gamma\mid \gamma<\kappa\rangle$ of elements of $H_\kappa$ such that:
\begin{itemize}
\item  for each $\gamma\in\nacc(\kappa)$, $A_\gamma=\emptyset$, and 
\item  for all $\iota<\kappa$ and $A\in H_\kappa$, $\{ \gamma\in B_\iota\mid A_\gamma=A\}$ is stationary in $\kappa$.
\end{itemize}
For every $\gamma<\kappa$, let $\beta_\gamma:=\otp(A_\gamma,{\lhd})$.
\begin{claim}\label{claim111}
Let $\iota<\kappa$, $\Omega \in[H_\kappa]^\kappa$ and $p\in H_{\kappa^{+}}$.
There exists $G\s\kappa$ such that $\{\beta_\gamma\mid\gamma\in G\}$ is cofinal in $\kappa$, and, for every $\gamma\in G$,
there exists an elementary submodel $\mathcal M\prec H_{\kappa^{+}}$ such that:
$$\begin{array}{cccc}
p\in\mathcal M,&
\mathcal M\cap\kappa=\beta_\gamma,&
\mathcal M\cap \Omega=A_{\gamma},\text{ and}&
\beta_\gamma\in B_\iota\cap\gamma.
\end{array}$$
\end{claim}
\begin{proof} Let $\epsilon<\kappa$ be arbitrary; we shall exhibit the existence of an ordinal $\gamma<\kappa$ with $\beta_\gamma>\epsilon$ satisfying all of the four requirements.

As $B_\iota$ is stationary, we may pick $\mathcal M\prec H_{\kappa^+}$ with $p,{\lhd},\Omega\in\mathcal M$ such that $\mathcal M\cap\kappa\in B_\iota\setminus(\epsilon+1)$.
Denote $\beta:=\mathcal M\cap\kappa$. As $\lhd\in\mathcal M$, we infer that $|\mathcal M\cap H_\kappa|=|\beta|<\kappa$,
so that $\mathcal M\cap\Omega\in H_\kappa$.
Thus, by the choice of $\vec A$, let us fix $\gamma>\beta$ such that $\mathcal M\cap\Omega=A_\gamma$.
Finally, as $H_{\kappa^+}\models \otp(\Omega,{\lhd})=\kappa$, we infer that $\beta_\gamma=\otp( A_\gamma,{\lhd})=\otp(\mathcal M\cap\Omega,{\lhd})=\kappa^{\mathcal M}=\beta$.
\end{proof}

Define $c:[\kappa]^2\rightarrow\kappa$ by letting, for all $\delta<\gamma<\kappa$,
$$c(\delta,\gamma):=\begin{cases}
\beta_\gamma,&\text{if }\beta_\gamma\in(\delta,\gamma];\\
\gamma,&\text{otherwise}.
\end{cases}
$$

Let $x\in\mathcal K(\kappa)$ be arbitrary. We define an element $C_x$ of $\mathcal F(\kappa)$, as follows:
\begin{itemize}
\item $\dom(C_x):=\acc(x)\cup\{c(\sup(x\cap\gamma),\gamma)\mid \gamma\in\nacc(x)\ \&\ \gamma\neq\min(x)\}$;
\item for all $\beta\in\dom(C_x)$, we let $C_x(\beta):=A_{\min(x\setminus\beta)}$.
\end{itemize}
\begin{claim} For all $x,y\in\mathcal K(\kappa)$:
\begin{enumerate}
\item $\dom(C_x)$ is a club in $\sup(x)$ with $\acc(\dom(C_x)) = \acc(x)$;
\item if $x\mathrel{\mathcal R}y$, then $C_x\mathrel{\mathcal R}C_y$.
\end{enumerate}
\end{claim}
\begin{proof} Left to the reader.
\end{proof}

Recalling Convention~\ref{omitsigma}, fix a sequence $\cvec D=\langle\mathcal D_\alpha\mid\alpha<\kappa\rangle$ witnessing $\p_\xi^-(\kappa, \mu,\allowbreak \mathcal R,\allowbreak\theta, \mathcal S,  \nu, {<}\omega)$.
Then, for every $\alpha\in\acc(\kappa)$, let $\mathcal C_\alpha:=\{ C_{x}\mid x\in\mathcal D_\alpha\}$.

\begin{claim} Suppose that $S\in\mathcal S$ and $\langle (\Omega_i,p_i,\iota_i) \mid i < \theta \rangle$ is a sequence of elements of $\mathcal P(H_\kappa)\times H_{\kappa^+}\times \kappa$,
and $n$ is a positive integer. Then there exist stationarily many $\alpha \in S$ 
such that  $| \mathcal C_\alpha| < \nu$, and, for every $C \in \mathcal C_\alpha$ and $i < \min\{\alpha, \theta\}$,
$$\sup\{ \beta \in \dm{C} \mid \suc_{n} (\dm{C} \setminus \beta) \subseteq B_{\iota_i}(\Omega_i,p_i,C) \} = \alpha.$$
\end{claim}
\begin{proof} Let $i<\theta$ be arbitrary.

$\br$ If $|\Omega_i|=\kappa$, then, by Claim~\ref{claim111}, we may fix a cofinal subset $G_i\s\kappa$ such that, for every $\gamma\in G_i$,
there exists an elementary submodel $\mathcal M\prec H_{\kappa^{+}}$, such that:
$$\begin{array}{ccccc}
p_i\in\mathcal M,&
\mathcal M\cap\kappa=\beta_\gamma,&
\mathcal M\cap \Omega_i=A_{\gamma},&
\beta_\gamma\in B_{\iota_i}\cap\gamma,\text{ and}&
G_i\cap\gamma=G_i\cap\beta_\gamma.
\end{array}$$

$\br$ If $|\Omega_i|<\kappa$, then as $\{\gamma\in B_{\iota_i}\mid A_\gamma=\Omega_i\}$ is stationary, we may fix a cofinal subset $G_i\s\kappa$ such that, 
for every $\gamma\in G_i$, there exists an elementary submodel $\mathcal M\prec H_{\kappa^{+}}$, such that:
$$\begin{array}{cccc}
p_i,{\lhd},\Omega_i\in\mathcal M,&
\mathcal M\cap\kappa=\gamma,&
\Omega_i=A_{\gamma},\text{ and}&
\gamma\in B_{\iota_i}.
\end{array}$$

Next, fix a club $E\s\kappa$ with the property that, for every $\alpha\in E$ and $i<\min\{\alpha,\theta\}$,
if $|\Omega_i|<\kappa$, then $\otp(\Omega_i,{\lhd})<\alpha$.
By the choice of $\cvec D$,
we may find a stationary $S'\s S\cap E$ such that, for every $\alpha\in S'$,
$| \mathcal D_\alpha| < \nu$, and,
for all $x \in \mathcal D_\alpha$ and $i < \min\{\alpha, \theta\}$:
$$\sup\{ \gamma\in x \mid \suc_{n+1} (x \setminus \gamma) \subseteq G_i \} = \alpha.$$

Let $\alpha\in S'$ be arbitrary. Clearly, $|\mathcal C_\alpha|\le|\mathcal D_\alpha|<\nu$. 
Fix $x\in\mathcal D_\alpha$ and $i<\min\{\alpha,\theta\}$; 
we need to prove that for every $\epsilon<\alpha$, there exists $\beta\in \dm C_x\setminus\epsilon$ with
$$\suc_{n}(\dm{C}_x \setminus \beta) \subseteq B_{\iota_i}(\Omega_i,p_i,C_x).$$

As $\alpha\in S'$, let us fix a large enough $\gamma\in x\setminus\epsilon$ such that $\suc_{n+1} (x \setminus \gamma) \subseteq G_i$.
If $|\Omega_i|<\kappa$, then we also require that $\otp(\Omega_i,{\lhd})\le\gamma$,
which is possible since $\alpha \in E$.

Let $\{ \gamma_j\mid j<n+1\}$ denote the increasing enumeration of $\suc_{n+1}(x\setminus\gamma)$.
Set $\beta:=c(\gamma,\gamma_0)$, so that $\beta\in\dm{C}_x\setminus\epsilon$ and $\suc_{n}(\dm{C}_x\setminus \beta)=\{c(\gamma_j,\gamma_{j+1})\mid j<n\}$.

Fix an arbitrary $j<n$, and we shall show that $c(\gamma_j,\gamma_{j+1})\in B_{\iota_i}(\Omega_i,p_i,C_x)$. 

$\br$ If $|\Omega_i|=\kappa$, then as $\gamma_{j+1}\in G_i$, pick  $\mathcal M\prec H_{\kappa^{+}}$ such that:
\begin{itemize}
\item $p_i\in\mathcal M$;
\item $\mathcal M\cap\kappa=\beta_{\gamma_{j+1}}$;
\item $\mathcal M\cap \Omega_i=A_{\gamma_{j+1}}$;
\item $\beta_{\gamma_{j+1}}\in B_{\iota_i}\cap\gamma_{j+1}$;
\item $G_i\cap\gamma_{j+1}=G_i\cap\beta_{\gamma_{j+1}}$.
\end{itemize}

By the last two bullets, $\gamma_j<\beta_{\gamma_{j+1}}<\gamma_{j+1}$, so that $c(\gamma_j,\gamma_{j+1})=\beta_{\gamma_{j+1}}$.
Furthermore, $\beta_{\gamma_{j+1}}$ is the unique element of the interval $\dm{C}_x\cap(\gamma_j,\gamma_{j+1})$,
so that $C_x(\beta_{\gamma_{j+1}})=A_{\min(x\setminus\beta_{\gamma_{j+1}})}=A_{\gamma_{j+1}}=\mathcal M\cap\Omega_i$. 
Altogether, $\mathcal M$ witnesses that $c(\gamma_j,\gamma_{j+1})\in B_{\iota_i}(\Omega_i,p_i,C_x)$.

$\br$ If $|\Omega_i|<\kappa$, then as $\gamma_{j+1}\in  G_i$, pick $\mathcal M\prec H_{\kappa^{+}}$ such that:
\begin{itemize}
\item $p_i,{\lhd},\Omega_i\in\mathcal M$;
\item $\mathcal M\cap\kappa=\gamma_{j+1}$;
\item $\Omega_i=A_{\gamma_{j+1}}$;
\item $\gamma_{j+1}\in B_{\iota_i}$.
\end{itemize}

As $\beta_{\gamma_{j+1}}=\otp(A_{\gamma_{j+1}},{\lhd})=\otp(\Omega_i,{\lhd})\le\gamma<\gamma_j<\gamma_{j+1}$, we infer that $\mathcal M\cap\Omega_i=\Omega_i=A_{\gamma_{j+1}}$ and $c(\gamma_j,\gamma_{j+1})=\gamma_{j+1}$.
In particular, $C_x(\gamma_{j+1})=A_{\min(x\setminus\gamma_{j+1})}=A_{\gamma_{j+1}}=\mathcal M\cap\Omega_i$. 
Altogether, $\mathcal M$ witnesses that $c(\gamma_j,\gamma_{j+1})\in B_{\iota_i}(\Omega_i,p_i,C_x)$.
\end{proof}
It follows that $\langle\mathcal C_\alpha\mid\alpha<\kappa\rangle$ is as sought.
\end{proof}

Putting the last three results together, we obtain:
\begin{cor}\label{Pbullet-equiv}
For $\theta>0$, and $\mathcal{R}$ from Example~\ref{example53b}, the following are equivalent:
\begin{enumerate}
\item $\p_\xi^-(\kappa, \mu, \mathcal{R}, \theta, \mathcal S,  \nu)\wedge(\kappa^{<\kappa}=\kappa)$;
\item $\p_\xi^\bullet(\kappa, \mu, \mathcal{R}, \theta, \mathcal S,  \nu)$;
\item For every partition $\vec B=\langle B_\iota\mid \iota<\kappa\rangle$ of $\kappa$ into stationary sets,
there exists a sequence $\cvec{C}$ that, together with $\vec B$, witnesses $\p_\xi^\bullet(\kappa, \mu, \mathcal R, \theta, \mathcal S,  \nu)$. \qed
\end{enumerate}
\end{cor}

The following combines Definition~\ref{indexedP} with Definition~\ref{pbullet}:

\begin{definition}\label{indexedpb}
$\p_\xi^\bullet(\kappa, \mu^{\ind}, {\sq}, \theta, \mathcal S,  \nu)$
asserts the existence of sequences $\cvec C=\langle \mathcal C_\alpha\mid \alpha<\kappa\rangle$, $\vec B=\langle B_\iota\mid\iota<\kappa\rangle$,
and $\langle i(\alpha)\mid \alpha<\kappa\rangle$ such that
$\cvec C$ and $\vec B$ together witness $\p_\xi^\bullet(\kappa, \mu^+,\allowbreak {\sq}, \theta, \mathcal S,  \nu)$,
and for every $\alpha\in\acc(\kappa)$, all of the following hold:
\begin{itemize}
\item there exists a canonical enumeration $\langle C_{\alpha,i}\mid i(\alpha)\le i<\mu\rangle$ of $\mathcal C_\alpha$;
\item for all $i\in[i(\alpha),\mu)$ and $\bar\alpha\in\acc(\dm{C}_{\alpha,i})$, $i\ge i(\bar\alpha)$ and $C_{\bar\alpha,i}\sq C_{\alpha,i}$;
\item $\langle \acc(\dm{C}_{\alpha,i})\mid i(\alpha)\le i<\mu\rangle$
is $\s$-increasing with $\acc(\alpha) = \bigcup_{i\in[i(\alpha),\mu)}\acc(\dm C_{\alpha,i})$.
\end{itemize}
\end{definition}

The proofs in this section make clear that the following holds as well.
\begin{cor}\label{upshotindx} Suppose that $\theta>0$ and $\mu<\kappa$.
Then $\p_\xi^\bullet(\kappa, \mu^{\ind}, {\sq}, \theta, \mathcal S,  \nu)$ is equivalent to
$\p_\xi^-(\kappa, \mu^{\ind}, {\sq}, \theta, \mathcal S,  \nu)\wedge(\kappa^{<\kappa}=\kappa)$.\qed
\end{cor}

\section{Tree constructions}\label{constructions-section}
In this section, we present various constructions from instances of the proxy principle $\p^\bullet(\kappa,\dots)$.
The next table summarizes the kind of $\kappa$-Souslin trees we obtain from the instance $\p^\bullet(\kappa,\mu,\mathcal R,\theta,\{E^\kappa_{\ge\partial}\},\nu)$.
Of course, for the $\chi$-complete trees, we must also assume that $\kappa$ is $({<}\chi)$-closed,
or we can simply assume that $\chi=\aleph_0$.
\begin{figure}[H]\label{table1}
$$\begin{array}{c|l|l|l|c|l|l}
\text{Theorem} &\mu&\hspace{6pt}\mathcal R&\theta&\partial&\nu&\text{Type of $\kappa$-Souslin tree}\\ \hline\hline
\ref{basicthm}&\kappa&{\sqx^*}&1&{\chi}&\kappa&\chi\text{-complete}\\ \hline
\ref{souslinwithascentpath}&\mu^{\ind}&{\hspace{6pt}\sq}&1&{\chi}&\kappa&\chi\text{-complete}\text{ with a }\mu\text{-ascent path}\\ \hline
\ref{omitpath}&\kappa&{\sqx^*}&1&{\max\{\chi,\lambda\}}&2&\begin{aligned}&\chi\text{-complete,}\\&\text{with no ascending path of width} <\lambda\end{aligned}\\ \hline
\ref{rigid}&\kappa&{\sqx}&1&{\max\{\chi,\lambda\}}&2&\begin{aligned}&\chi\text{-complete, rigid,}\\&\text{with no ascending path of width} <\lambda\end{aligned}\\ \hline
\ref{free}&\kappa&{\sqx}&\kappa&{\chi}&2&\chi\text{-complete}, \chi\text{-free}\\ \hline
\ref{corhom}&\kappa&{\sqx^*}&\kappa&{\chi}&\kappa&\chi\text{-complete, uniformly homogeneous}\\ \hline
\ref{prop25}&2&{\hspace{6pt}\sq}&\kappa&0&2&\text{slim, uniformly coherent} \\\hline
\end{array}$$
\caption{Relationship between the vector of parameters and the characteristics of trees obtained.}
\end{figure}

By \cite{ShZa:610}, every uniformly coherent $\omega_1$-Souslin tree is the product of two free $\omega_1$-Souslin trees.
In addition, it is not hard to see that for any cardinal $\lambda$, any $\lambda$-free $\lambda^+$-Souslin tree is specializable.
Finally, by a result of Baumgartner that we mentioned in the introduction, there consistently exist $\omega_2$-Souslin trees which are not specializable.
This suggests that uniformly coherent $>$ free $>$ specializable $>$ plain, and this claim is in fact supported by the results summarized in the above table.
Indeed, for a uniformly coherent $\kappa$-Souslin tree, we assume $\mu=\nu=2$ and $\theta=\kappa$;
for a free $\kappa$-Souslin tree, we (allow $\mu=\kappa$ but) assume $\nu=2$ and $\theta=\kappa$;
for a $\kappa$-Souslin tree omitting a narrow ascending path (which is a generalization of specializable),
we assume $\nu=2$ and $\theta=1$,
whereas, for a plain $\kappa$-Souslin tree, $\nu=\kappa$ and $\theta=1$ will do.

\medskip

The next list demonstrates well the utility of the proxy principle 
as a device that provides a disconnection between the tree constructions and the study of the combinatorial hypotheses.\footnote{Keep in mind
the monotonicity properties of the proxy principle, as described in Remark~\ref{monotonicity}.}

\begin{thm}\label{thm61}
\begin{enumerate}
\item Assuming $\lambda=\cf(\lambda)\ge\aleph_0$,  $\diamondsuit(E^{\lambda^+}_\lambda)$ entails $\p^\bullet_\lambda(\lambda^+,2,\sqleft{\lambda},\allowbreak\lambda^+,\{E^{\lambda^+}_\lambda\},2)$.
In particular, $\diamondsuit(\omega_1)$ entails $\p^\bullet_\omega(\omega_1,2,{\sq},\omega_1,\{\omega_1\},2)$.
\item Assuming $\lambda=\cf(\lambda)\ge\aleph_0$, $\diamondsuit^*(E^{\lambda^+}_\lambda)$ entails
$\p^\bullet_\lambda(\lambda^+, 2, {\sql}, \lambda^+, {\ns^+_{\lambda^+} \restriction E^{\lambda^+}_\lambda}, 2)$. 

\item Assuming $\lambda\ge\aleph_0$, $\sd_\lambda$ entails $\p^\bullet_\lambda(\lambda^+,2,{\sq},\lambda^+,\{\lambda^+\},2)$.

\item Assuming $\lambda\ge\aleph_1$,
$\square_\lambda+\ch_\lambda$ entails $\p^\bullet_\lambda(\lambda^+,2,{\sq},{<}\lambda,\{E^{\lambda^+}_\chi\},2)$
for every $\chi\in\reg(\lambda)$,
as well as $\p^\bullet(\lambda^+,2,{\sq^*},1,\{E^{\lambda^+}_{\lambda}\},2)$ for $\lambda$ regular.
\item Assuming $\lambda\geq\beth_\omega$,
$\square(\lambda^+)+\ch_\lambda$ entails $\p^\bullet(\lambda^+,2,{\sq},1,\{\lambda^+\},2)$,
as well as $\p^\bullet(\lambda^+,2,{\sq^*},1,\{S\},2)$ for every stationary $S\s\lambda^+$.
\item Assuming $\lambda$ is singular, ${\square(\lambda^+)}+{\gch}$ entails $\p^\bullet(\lambda^+,2,{\sq},\lambda^+,\{\lambda^+\},2)$.
\item Assuming $\lambda\ge\aleph_1$,
$\square(\lambda^+)+\gch$ entails $\p^\bullet(\lambda^+,2,{\sq},1,\{E^{\lambda^+}_\chi\},2)$
for every $\chi \in \reg(\lambda)$.
\item Assuming $\lambda\ge\aleph_1$ and $1\le\mu<\cf(\lambda)$,
$\square(\lambda^+,\mu)+\gch$ entails $\p^\bullet(\lambda^+,\mu^+,{\sq},1,\allowbreak\{E^{\lambda^+}_{\geq\chi}\},\mu^+)$
for every $\chi\in\reg(\cf(\lambda))$,\footnote{For a finite cardinal $\mu$, $\mu^+$ stands for $\mu+1$.}
as well as $\p^\bullet(\lambda^+,\mu^+,{\sq^*},1, \{S\},\mu^+)$ for every stationary $S\s\lambda^+$.
\item Assuming $\aleph_0\le\mu<\cf(\lambda)$, $\square^\ind(\lambda^+,\mu)+\gch$ entails $\p^\bullet(\lambda^+,\mu^\ind,{\sq},\allowbreak 1,\allowbreak\{E^{\lambda^+}_{\ge\chi}\},\mu^+)$
for every $\chi\in\reg(\cf(\lambda))$.
\item For $\kappa\ge\aleph_2$ and a stationary $E\subset\kappa$,
$\square(E)+\diamondsuit(E)$ entails $\p^\bullet(\kappa,2,{\sq^*},1,\{S\},2)$ for every stationary $S\s\kappa$.
\item Assuming $\lambda^{<\lambda}=\lambda\ge\aleph_0$,
after forcing to add a single Cohen subset of $\lambda$,
$\p^\bullet_\lambda(\lambda^+,2,{\mathcal R},\lambda^+,{\ns^+_{\lambda^+}\restriction E^{\lambda^+}_\lambda},2)$ holds
with $\mathcal R={\sql}$. If $\square_\lambda$ holds in the ground model, then the conclusion holds with $\mathcal R={\sq}$.
\item Assuming $\lambda^{<\lambda}=\lambda\ge\aleph_0$
and $\kappa>\lambda$ strongly inaccessible,
after forcing by a $({<}\lambda)$-distributive $\kappa$\nobreakdash-cc notion of forcing collapsing $\kappa$ to $\lambda^+$,
$\p^\bullet_\lambda(\lambda^+,\infty,{\sq},\lambda^+,\allowbreak{\ns^+_\kappa\restriction E^{\lambda^+}_\lambda},2)$ holds.
\item Assuming $\lambda^{<\lambda}=\lambda\ge\aleph_1$ and $\ch_\lambda$,
after forcing with a $\lambda^+$-cc notion of forcing of size $\le\lambda^+$ that preserves the regularity of $\lambda$ and is not ${}^\lambda\lambda$-bounding,
$\p^\bullet_\lambda(\lambda^+,\infty,{\sq},\lambda^+, {\ns^+_{\lambda^+} \restriction E^{\lambda^+}_\lambda}, 2)$ holds.
\item Assuming $\lambda^{<\lambda}=\lambda\ge\aleph_1$ and $\ch_\lambda$,
after forcing with a $\lambda^+$-cc notion of forcing of size $\le\lambda^+$ that forces $\cf(\lambda)<|\lambda|$ (e.g., Prikry, Magidor, and Radin forcing),
$\p^\bullet_\lambda(\kappa,\infty,{\sq},\kappa,{\ns^+_\kappa\restriction T},2)$ holds,
where $\kappa := \lambda^+$ and $T := E^\kappa_\lambda$ are defined in the ground model.
\item For infinite regular cardinals $\theta<\lambda$ satisfying $\lambda^{<\theta}=\lambda$ and $\ch_\lambda$,
after L\'evy-collapsing $\lambda$ to $\theta$,
$\p^\bullet_\theta(\kappa,\infty,{\sq},\kappa,{\ns^+_\kappa\restriction T},2)$ holds,
where $\kappa := \lambda^+$ and $T := E^\kappa_\lambda$ are defined in the ground model.
\item
If $\lambda=\theta^+$ for a regular cardinal $\theta$, and $\ns\restriction E^{\lambda}_{\theta}$ is saturated,
then $\ch_\lambda$ entails $\p^\bullet(\lambda^+,2,{\sql}^*,\theta,\{E^{\lambda^+}_{\lambda}\},2)$.
\item  Assuming $\lambda^{<\lambda}=\lambda\ge\aleph_1$ and $\ch_\lambda$,
if there exists a nonreflecting stationary subset of $E^{\lambda^+}_{\neq\lambda}$,
then $\p^\bullet_\lambda(\lambda^+,\lambda^+,{\sq},{<}\lambda,\{\lambda^+\},2)$ holds,
and so does $\p^\bullet(\lambda^+,\lambda^+,{\sq^*},1,\{E^{\lambda^+}_\lambda\},2)$.
\item Assuming $\lambda=2^{<\lambda}$ is singular and ${\square^*_\lambda}+{\ch_\lambda}$, if there exists a nonreflecting stationary subset of $E^{\lambda^+}_{\neq\cf(\lambda)}$,
then $\p^\bullet_{\lambda^2}(\lambda^+,\lambda^+,{\sq},\lambda^+,\{\lambda^+\},2)$ holds.
\item For $\kappa$ strongly inaccessible, if there exists a sequence $\langle A_\alpha \mid \alpha\in E \rangle$
satisfying the hypothesis of Theorem~B,
then $\p^\bullet(\kappa,\kappa,{\sqleftup{E}},1,\{E\},2)$ holds.
\item For $\kappa$ strongly inaccessible, if $\diamondsuit(E)$ holds over some nonreflecting stationary $E\s\kappa$, 
then $\p^\bullet(\kappa,\kappa,{\sqleftup{E}},\kappa,\{E\},2)$ holds,
and so does $\p^\bullet(\kappa,\kappa,{\sq^*},1,\{S\},2)$ for every nonreflecting stationary $S\s\kappa$.
\item If $V=L$  and $\kappa$ is not weakly compact, then
$\p^\bullet(\kappa,2,{\sq},\kappa,\mathcal{S},2)$ holds for
$\mathcal{S} := \{E^\kappa_{\geq\chi} \mid \chi\in\reg(\kappa) \text{ and $\kappa$ is $({<}\chi)$-closed} \}$.
\end{enumerate}
\end{thm}
\begin{proof} 
Each statement asserts that an instance of $\p^\bullet(\kappa,\dots)$ follows from hypotheses that include
$\kappa^{<\kappa}=\kappa$.
Thus, by Theorem~\ref{mainpbullet} (or Corollary~\ref{upshotindx}),
it suffices to prove $\p^-(\kappa,\dots)$ in each case.
Furthermore, for statements where the hypotheses include $\diamondsuit(\kappa)$,
we may appeal to Theorem~\ref{get-sigma-finite},
so that it suffices to prove $\p^-(\kappa,\dots,1)$.

(1) By \cite[Theorem~3.6]{paper22} and \cite[Corollary~1.12]{paper22}.
(2) By Theorem~\ref{diamond*-to-P*}(1) followed by Lemma~\ref{lemma916}.
(3) By \cite[Theorem~3.6]{paper22}.
(4) By \cite[Corollary~3.9]{paper22} and \cite[Corollary~4.13]{paper24}.
(5) By \cite[Corollaries 4.7 and~4.13]{paper24}.
(6) By \cite[Corollary~4.22]{paper29} using $\chi:=\aleph_0$, since under \gch\ every singular cardinal is strong limit.
(7) By \cite[Corollary~4.5]{paper24}.
(8) By the same proof of Theorem~\ref{mainindex}, and then \cite[Lemma~3.8]{paper32}.
(9) By Theorem~\ref{mainindex}.
(10) This is Corollary~\ref{corollaryalajensen}(2).
(11) By Lemma \ref{lemma916} above and either Theorem~5.7 or Theorem~4.2(2) of \cite{paper22}.
Note that the cited theorems from \cite{paper22} require $\lambda$ to be uncountable just to be able to verify $\diamondsuit(\lambda^+)$,
whereas, here, we settle for $\p^-(\lambda^+,\ldots) \land\ch_\lambda$. See the proof of \cite[Theorem~2.3]{rinot12} for the exact details.
(12) By \cite[Proposition~3.10]{paper26} (which relies on Theorem~\ref{diamond*-to-P*}(2) above) 
followed by Lemma~\ref{lemma916}.
(13) By \cite[Proposition~B(1) and Theorem~3.4]{paper26} followed by Lemma~\ref{lemma916} above 
(cf.~Remark~\ref{remark-P*}).
(14) By \cite[Proposition~B(2) and Theorem~3.4]{paper26} followed by Lemma~\ref{lemma916} above 
(cf.~Remark~\ref{remark-P*}).
(15) By \cite[Proposition~3.9]{paper26} followed by Lemma~\ref{lemma916} above 
(cf.~Remark~\ref{remark-P*}).
(16) By \cite[Theorem~6.4]{paper22}.
(17) By \cite[Theorem~A]{paper32}.
(18) By \cite[Theorem~B]{paper32}.
(19) This is Theorem~\ref{cor430}.
(20) By Theorem~\ref{thm43} and Corollary~\ref{cor29}.
(21) By \cite[Corollary~1.10(5)]{paper22} or \cite[Corollary~4.14]{paper29}.
\end{proof}

\subsection{Basic characteristics of trees}
Examining our construction of the $\kappa$-Souslin tree in the proof of Proposition~\ref{prop23},
we notice that all of the hard work took place when constructing the nonzero limit levels of the tree.
It was at those levels that we balanced the normality requirement with the need to bound the size of $T_\alpha$ 
and to seal antichains.

In contrast, the only constraint we addressed at successor levels was the requirement that the tree be ever-branching,
in consideration of Lemma~\ref{enough-no-antichains},
and we did that by assigning two immediate successors to every node from the previous level.\footnote{For two nodes $x,y$ in a streamlined tree $T$,
we say that $y$ is an \emph{immediate successor} of $x$ iff $x\stree y$ and $\dom(y)=\dom(x)+1$.}
But here we have the flexibility to impose some additional features,
without affecting the most important global properties of the tree.

\begin{definition}
A streamlined $\kappa$-tree $T$ is said to be:
\begin{itemize}
\item \emph{binary} iff $T \s {}^{<\kappa}2$;
\item \emph{$\varsigma$-splitting} (for an ordinal $\varsigma<\kappa$) iff 
every node in $T$ admits at least $\varsigma$ many immediate successors;
\item \emph{splitting} iff it is 2-splitting;
\item \emph{prolific} iff, for all $\alpha<\kappa$ and $t\in T_\alpha$,
$\{ \conc{t}{i} \mid i<\max \{\omega, \alpha\}\}\s T_{\alpha+1}$.
\end{itemize}
\end{definition}

While a $\varsigma$-splitting tree, for any value $\varsigma>2$, cannot be binary,
we observe the following implications between properties of a streamlined $\kappa$-tree:
\begin{gather*}
\text{prolific} \implies \omega\text{-splitting} \implies \text{splitting} \implies \text{ever-branching}.
\end{gather*}

Referring back to our construction in the proof of Proposition~\ref{prop23},
we note that the tree constructed there was binary.
However, we can easily tweak the construction of the successor levels in order to ensure that
the resulting tree ends up being (no longer binary, but rather) 
prolific and/or $\varsigma$-splitting for some value of $\varsigma<\kappa$,
without affecting the validity of any other aspects of the proof.
Precisely, for an ordinal $\varsigma<\kappa$,
to obtain a $\kappa$-Souslin tree that is prolific and $\varsigma$-splitting,
we modify the successor level in the proof of Proposition~\ref{prop23},
setting,
for every $\alpha<\kappa$,
$$T_{\alpha+1}:=\{ \conc{t}{i}\mid t\in T_\alpha, i<\max\{\alpha,\varsigma,\omega\}\}.$$
Of course, it is no longer true that 
each $T_\alpha$ is a subset of ${}^\alpha2$ of size $\le\max\{\aleph_0,|\alpha|\}$
as claimed in the original proof,
but that was simply a matter of preference.
The important constraint, to be maintained for all $\alpha<\kappa$ throughout the recursive construction, is that
$T_\alpha$ is a subset of ${}^\alpha\kappa$ of size $<\kappa$;
this follows at successor levels from the fact that $\varsigma<\kappa$,
and at limit levels from regularity of $\kappa$ together with the property $(*)_\alpha$ of the construction.
Furthermore, $\wo$ must be chosen at the outset to be a well-ordering of $^{<\kappa}\kappa$ (or some larger set) instead of ${}^{<\kappa}2$.

As we present the more involved Souslin-tree constructions throughout the rest of this paper,
the reader should have no trouble adapting them between binary and prolific/$\varsigma$-splitting, as desired.
Note also that an abstract translation between various kinds of trees is offered in the appendix of \cite{rinot20}.

\medskip

But even at the limit levels, there is some degree of flexibility, as witnessed by the two alternatives of the following definition.
\begin{definition}\label{def63} A streamlined tree $T$ is said to be:
\begin{itemize}
\item \emph{slim} if $|T_\alpha| \leq \max \{\left|\alpha\right|, \aleph_0\}$ for every ordinal $\alpha$.
\item  \emph{$\chi$-complete} if,
for any $\stree$-increasing sequence $\eta$, of length $<\chi$, of elements of $T$,
the limit of the sequence,  $\bigcup\im(\eta)$, is also in $T$.
\end{itemize}
\end{definition}

Notice that the $\kappa$-Souslin tree constructed in the proof of Proposition~\ref{prop23} is slim.
This is a result of adhering to property~$(*)_\alpha$ in the definition of levels $T_\alpha$ for every $\alpha\in\acc(\kappa)$.
Recalling the discussion in Subsection~\ref{subsection:completing+sealing},
there needs to be some stationary $\Gamma\s\kappa$ on which, for every $\alpha\in\Gamma$,
not every $\alpha$-branch through $T\restriction\alpha$ will have its limit placed into $T_\alpha$.
In the proof of Proposition~\ref{prop23}, 
we took the simplest approach by setting $\Gamma := \acc(\kappa)$.\footnote{As a result of taking this simplest approach,
the tree also satisfies the property of being \emph{club-regressive} (in addition to being slim), 
as explored in~\cite[Proposition~2.3]{paper22}.}
A much more complicated approach is taken in \cite[\S5]{rinot20}.
In the upcoming treatment, we shall focus on a relatively simple form of a set $\Gamma$,
namely, $\Gamma:=E^\kappa_{\ge\chi}$ for some $\chi\in\reg(\kappa)$. 
Note that this means that the resulting $\kappa$-Souslin tree would be $\chi$-complete.
Nevertheless, a minor tweaking would facilitate obtaining slim $\kappa$-Souslin trees;
simply set $\chi:=\omega$ and make sure to impose $\omega_1$ as the value of the second parameter of the proxy principle.\footnote{The notion of slimness is more prevalent 
in the context of $\kappa$-Kurepa trees as a property that rules out some trivial examples. 
Our interest in slim $\kappa$-Souslin trees comes from \cite{rinot20},
where we constructed $\kappa$-Souslin trees whose reduced powers are $\kappa$-Kurepa.}

\subsection{The underlying setup}
Throughout the rest of this section, we fix a well-ordering $\lhd_\kappa$ of $H_\kappa$.
All of the constructions in this section will use hypotheses of the form $\p^\bullet(\kappa,\ldots)$
that imply $\kappa^{<\kappa}=\kappa$. Thus, we shall moreover assume that $\otp(H_\kappa,{\lhd_\kappa})=\kappa$.

\begin{definition}
For every $T\in  H_\kappa$, denote $\beta(T):=0$ unless there is $\beta<\kappa$
such that $T\s{}^{\le\beta}H_\kappa$ and $T \nsubseteq {}^{<\beta}H_\kappa$,
in which case, we let $\beta(T):=\beta$ for this unique $\beta$.
We also let $T_{\beta(T)}:=\{ x\in T\mid \dom(x)=\beta(T)\}$.\footnote{If $T$ is a streamlined tree,
then this notation coheres with the fourth bullet of Lemma~\ref{streamlined-basics}, but, in general,
$T_{\beta(T)}$ may well be empty.}
\end{definition}
We collect here a gallery of \emph{actions} which we will use throughout the constructions of this section.
The reader may skip this definition at the moment, and come back to each of its clauses upon its use.

\begin{definition}\label{actions}
\begin{enumerate}
\item The default extension function, $\defaultaction:(H_\kappa)^2\rightarrow H_\kappa$, is defined as follows.
Let $\defaultaction(x,T):=x$, unless $\bar Q:=\{ z\in T_{\beta(T)}\mid x\s z\}$ is nonempty,
in which case, we let $\defaultaction(x,T):=\min(\bar Q, {\lhd_\kappa})$.

\item The function for sealing antichains, $\sealantichain:(H_\kappa)^3\rightarrow H_\kappa$,
is defined as follows.
Let $\sealantichain(x,T,\mho):=\defaultaction(x,T)$,
unless $$Q:=\{ z\in T_{\beta(T)}\mid \exists y\in\mho( x\cup y\s z)\}$$ is nonempty,
in which case, we let $\sealantichain(x,T,\mho):=\min(Q, {\lhd_\kappa})$.

\item\label{actionsealpath}
The function for sealing ascending paths,
$\sealpath : (H_\kappa)^3\rightarrow H_\kappa$, is defined as follows.
Given $(x,T,\mho)\in(H_\kappa)^3$, if the set
\[
Q := \{ z \in T_{\beta(T)}\mid x\s z\ \&\ \forall y\in \mho[\dom(y)=\dom(x)+1\implies y\nsubseteq z]\}
\]
is nonempty, then let $\sealpath(x,T,\mho) := \min (Q, {\lhd_\kappa})$.
Otherwise, let $\sealpath(x,T,\mho) := \defaultaction(x,T)$.
\item\label{autoaction} The function for sealing automorphisms,
$\sealautomorphism: (H_\kappa)^4\rightarrow H_\kappa$, is defined as follows.
Given $(x,T,b,\mho)$:
If $x \in T \restriction (\beta(T))$,
$b$ is a partial function from $T \restriction (\beta(T))$ to $T \restriction (\beta(T))$,
$\mho$ is an automorphism of $T \restriction (\beta(T))$, and the set
\[
Q := \{ x_0 \in \dom(b) \mid \mho(x_0) \neq x_0 \}
\]
is nonempty, then let $x_0 := \min(Q, {\lhd_\kappa})$, $\tilde{x}_0 := \defaultaction(b(x_0),T)$,
$\tilde{y}_0 := \bigcup_{\gamma<\beta(T)} \mho(\tilde{x}_0 \restriction\gamma)$,
and $\sealautomorphism(x,T,b,\mho) := \defaultaction(x, T \setminus\{\tilde{y}_0\})$.
Otherwise, let $\sealautomorphism(x,T,b,\mho):=\defaultaction(x,T)$.

\item\label{actionfree} The function for sealing antichains in derived trees,
$\free:(H_\kappa)^4\rightarrow H_\kappa$, is defined as follows.\footnote{The notation $(z)_\xi$ and $T(\vec{b})$ will be introduced in Subsection~\ref{subsection:free} below.}
Given $(x,T,\vec{b},\mho)\in (H_\kappa)^4$: If $T$ is a streamlined tree, $\vec b\in{}^{<\kappa}T$ and $\vec b\neq\emptyset$,
then put $z:=\sealantichain(\emptyset,T(\vec b),\mho)$, and consider the following options:

$\br$ If 
$z \in {}^{\beta(T)}({}^{\dom(\vec{b})}H_\kappa)$ and
there exists some $\xi<\dom(\vec b)$ such that $x\cup (z)_\xi$ is in $T$, 
then let $\free(x,T,\vec b,\mho):=(z)_\xi$ for the least such $\xi$.

$\br$ Otherwise, let $\free(x,T,\vec b,\mho):=\defaultaction(x,T)$.
\end{enumerate}
\end{definition}

The following is obvious.

\begin{lemma}[Extension Lemma]\label{extendfact}
Suppose $x\in T\in  H_\kappa$, $\mho,b, \vec{b} \in H_\kappa$, and $T$ is a normal subtree of ${}^{\le\beta(T)}\kappa$.
Then $\defaultaction(x,T)$, $\sealantichain(x,T,\mho)$, $\sealpath(x,T,\mho)$, and $\free(x,T,\vec b,\mho)$
are elements of $T_{\beta(T)}$ extending $x$.
If $T\restriction\beta(T)$ is moreover ever-branching,
then $\sealautomorphism(x,T,b,\mho)$ is an element of $T_{\beta(T)}$ extending $x$, as well.
\qed
\end{lemma}

\subsection{The prototype construction}
\label{subsection:construction-from-weakest}

In this subsection, we present a construction of a $\kappa$-Souslin tree from the weakest useful instance of the proxy principle.
\begin{cor}
If $\p^\bullet(\kappa,\kappa,{\sq^*},1,\{\kappa\},\kappa)$ holds, then there exists a $\kappa$-Souslin tree.
\end{cor}
\begin{proof} Appeal to Theorem~\ref{basicthm} below with $\chi:=\aleph_0$.
\end{proof}

All of the subsequent constructions in this section, with the exception of the construction of a uniformly homogeneous tree in Subsection~\ref{homogeneoussection}, will be modeled after the following construction.

\begin{thm}\label{basicthm}
Suppose that $\kappa$ is $({<}\chi)$-closed for a given $\chi\in\reg(\kappa)$. Let $\varsigma<\kappa$.
If $\p^\bullet(\kappa,\kappa,{\sqx^*},1,\{E^\kappa_{\geq\chi}\},\kappa)$ holds, then there exists a
normal, prolific, $\varsigma$-splitting, $\chi$-complete $\kappa$-Souslin tree.
\end{thm}
\begin{proof}
The proof is very similar in spirit to the one from Proposition~\ref{prop23};
one just needs to be a little bit more careful.

Suppose $\p^\bullet(\kappa,\kappa,{\sqx^*},1,\{E^\kappa_{\geq\chi}\},\kappa)$ holds,
as witnessed by $\cvec{C} = \langle \mathcal C_\alpha\mid \alpha<\kappa\rangle$ and
$\vec B=\langle B_\iota \mid \iota< \kappa \rangle$.
We shall recursively construct a sequence $\langle T_\alpha\mid \alpha<\kappa\rangle$ of levels
whose union will ultimately be the desired tree $T$.
In this construction,
we shall ensure that for each $\alpha<\kappa$,
the level $T_\alpha$ will be a subset of ${}^\alpha\kappa$ of size $<\kappa$.

Let $T_0:=\{\emptyset\}$, and for all $\alpha<\kappa$,
let $T_{\alpha+1}:=\{ \conc{t}{i}\mid t\in T_\alpha, i<\max\{\alpha,\varsigma,\omega\}\}$.

Next, suppose that $\alpha<\kappa$ is a nonzero limit ordinal,
and that $\langle T_\beta\mid \beta<\alpha\rangle$ has already been defined.
Constructing the level $T_\alpha$ involves deciding which $\alpha$-branches
through $T \restriction \alpha$ will have their limits placed into the tree.
Denote $\Gamma:=E^\kappa_{\ge\chi}$.
The construction splits into two cases, depending on whether $\alpha\in\Gamma$:

$\br$ Suppose $\alpha\notin\Gamma$. Then we let $T_\alpha$ consist of the limits of all $\alpha$-branches through $T\restriction\alpha$.
This construction ensures that the tree will be $\chi$-complete,
and as any $\alpha$-branch through $T\restriction\alpha$ is determined by a subset of $T\restriction\alpha$
of cardinality $\cf(\alpha)<\chi$,
the $({<}\chi)$-closedness of $\kappa$ together with $|T\restriction\alpha|<\kappa$ ensures that $\left|T_\alpha\right| < \kappa$ at these levels.
Normality at level $T_\alpha$ is verified by induction:
Fixing a closed sequence of ordinals of minimal order-type converging to $\alpha$
enables us to find, for any given $x \in T \restriction \alpha$,
an $\alpha$-branch through $T\restriction\alpha$ containing $x$,
and the limit of such an $\alpha$-branch will necessarily be in $T_\alpha$.

$\br$ Now suppose $\alpha\in\Gamma$.
The idea for ensuring normality at level $T_\alpha$ is to attach to each
 $C \in \mathcal C_\alpha$ and $x \in T \restriction \dm{C}$
some node $\mathbf{b}^C_x \in {}^\alpha\kappa$ above $x$,
and then let
\[
T_\alpha := \{ \mathbf{b}^C_x \mid C \in \mathcal C_\alpha, x \in T \restriction \dm{C}\}.\tag*{$(**)_\alpha$}
\]

By the induction hypothesis, $|T_\beta|<\kappa$ for all $\beta<\alpha$,
and by the choice of $\cvec{C}$ we have $|\mathcal C_\alpha|<\kappa$,
so that by regularity of $\kappa$ we are guaranteed to end up with $|T_\alpha|<\kappa$.

As for every $C \in \mathcal C_\alpha$ and $x\in T\restriction \dm{C}$,
$\mathbf{b}^C_x$ will be the limit of some $\alpha$-branch through $T\restriction\alpha$ that contains $x$,
we opt to describe $\mathbf{b}^C_x$ as the limit $\bigcup\im(b^C_x)$ of a sequence
$b^C_x\in\prod_{\beta\in\dm{C}\setminus\dom(x)}T_\beta$ such that:
\begin{itemize}
\item $b^C_x(\dom(x))=x$;
\item $b_x^C(\beta')\stree b_x^C(\beta)$ for any pair $\beta'<\beta$ of ordinals from $(\dm{C}\setminus\dom(x))$;
\item $b^C_x(\beta)=\bigcup\im(b^C_x\restriction\beta)$ for all $\beta\in\acc(\dm{C}\setminus\dom(x))$.
\end{itemize}

Let $C \in \mathcal C_\alpha$ be arbitrary.
By recursion over $\beta \in \dm{C}$, we shall assign a value $b^C_x (\beta)$ in $T_\beta$
for all $x \in T \restriction (\dm{C}\cap(\beta+1))$.

Fix $\beta \in \dm{C}$, and assume that for every $x \in T \restriction(\dm{C}\cap\beta)$ we have already defined
$b^C_x \restriction \beta$.
We must define the value of $b^C_x(\beta)$ for all $x \in T \restriction (\dm{C}\cap (\beta+1))$.

\begin{enumerate}
\item For every $x \in T_\beta$, let $b^C_x(\beta) := x$.
We take care of these nodes separately, because for any such node $x$,
the sequence $b^C_x$ is just starting here.

\item Next, let $x \in T \restriction (\dm{C}\cap \beta)$ be arbitrary. In particular, assume that $\dm{C}\cap\beta\neq\emptyset$.

\begin{enumerate}
\item If $\beta\in\nacc(\dm{C})$, then let $\beta^-:=\sup(\dm{C}\cap\beta)$ denote the predecessor of $\beta$ in $\dm{C}$,
and then set
$$b_x^C(\beta):=
\begin{cases}
\sealantichain(b^C_x(\beta^-),T\restriction(\beta+1),C(\beta)),&\text{if }\beta\in B_0;\\
\defaultaction(b^C_x(\beta^-),T\restriction(\beta+1)),&\text{otherwise}.
\end{cases}$$

\item If $\beta \in \acc(\dm{C})$,
then we let $b^C_x(\beta):=\bigcup\im(b^C_x\restriction\beta)$, as promised.
\end{enumerate}
\end{enumerate}

The following is obvious, and is aligned with the microscopic perspective described in requirement (2) of Subsection~\ref{subsection:coherent}.

\begin{dependencies}\label{depends651}
For any two consecutive points $\beta^-<\beta$ of $\dom(b^C_x)$,
the value of $b^C_x(\beta)$ is completely determined by
$b^C_x(\beta^-)$, $T\restriction(\beta+1)$ and $C(\beta)$.
\end{dependencies}

In the case $\beta\in\nacc(\dm{C})$, since $b^C_x(\beta^-)$ belongs to the normal tree $T \restriction (\beta+1)$,
we infer from the Extension Lemma (Lemma~\ref{extendfact}) that $b_x^C(\beta)$ is an element of $T_\beta$ extending $b^C_x(\beta^-)$.
In the case $\beta\in\acc(\dm{C})$, the fact that $b^C_x(\beta)$ belongs to $T_\beta$ requires an argument:

\begin{claim}\label{coherent-from-weakest} Let $\beta\in\acc(\dm C)$. 
Then $b^C_x (\beta) \in T_\beta$.
\end{claim}
\begin{proof}  If $\beta\notin\Gamma$, then $T_\beta$ was constructed to consist of
the limits of all $\beta$-branches through $T\restriction\beta$,
including the limit of the $\beta$-branch determined by $b^C_x\restriction\beta$, which is $b^C_x (\beta)$.
Thus, we may assume that $\beta\in\Gamma$.

Since $\beta \in \acc(\dm C)$,
let us fix $D\in\mathcal C_\beta$ such that $D\sqleft{\chi}^* C$.
As $\Gamma=E^\kappa_{\ge\chi}$, in fact, $D\sq^* C$.
Fix $\gamma\in\dm C$ such that 
${D\restriction(\dm{D}\setminus\gamma)}\sq{C\restriction(\dm{C}\setminus\gamma)}$.
Set $d:=\dm{D}\setminus\gamma$ and $y:=b_x^C(\gamma)$.
We now prove, by induction on $\delta \in d$, that $b^C_x(\delta)=b^C_y(\delta)=b^D_y(\delta)$.

\begin{itemize}
\item Clearly, $b^C_x(\min(d)) = y = b^C_y(\min(d))=b^D_y(\min(d))$.
\item Suppose $\delta^-<\delta$ are successive points of $d$,
and $b^C_x(\delta^-)=b^C_y(\delta^-)=b^D_y(\delta^-)$.
Then, by Dependencies~\ref{depends651}, also $b^C_x(\delta)=b^C_y(\delta)=b^D_y(\delta)$.
\item For $\delta \in \acc(d)$:
If the sequences are identical up to $\delta$, then their limits must be identical.
\end{itemize}

It follows that $b^C_x(\beta)=\bigcup_{\delta\in d}b^C_x(\delta)=\bigcup_{\delta\in d}b^D_y(\delta)=\mathbf{b}^D_y$, and by the induction hypothesis $(**)_\beta$, the latter is in $T_\beta$.
So $b^C_x(\beta)\in T_\beta$, as sought.
\end{proof}

This completes the definition of the sequence $b^C_x$, and thus of its limit $\mathbf{b}^C_x$,
for each $C \in \mathcal C_\alpha$ and each $x\in T\restriction\dm{C}$.
Consequently, the level $T_\alpha$ is defined as promised in $(**)_\alpha$.

\begin{claim}\label{tailwithsame} Let $\alpha\in\Gamma$ and $t\in T_\alpha$.
Then there exists $C\in\mathcal C_\alpha$ such that $\{\delta\in\dm C\mid t=\mathbf{b}_{t\restriction \delta}^C\}$ is a final segment of $\dm C$.
\end{claim}
\begin{proof} By the same analysis from the proof of Claim~\ref{coherent-from-weakest}.
\end{proof}

Having constructed all levels of the tree, we then let
$T := \bigcup_{\alpha < \kappa} T_\alpha$.
It is clear from the construction that
$T$ is a normal, prolific, $\varsigma$-splitting, $\chi$-complete streamlined $\kappa$-tree.
By Lemma~\ref{enough-no-antichains}, to prove that  $T$ is $\kappa$-Souslin,
it suffices to show that it has no $\kappa$-sized antichains.
By Lemma~\ref{equiv-no-antichains}, it suffices to prove the following.

\begin{claim}\label{naantileft} Let $A\s T$ be a maximal antichain.
Then there exists $\alpha<\kappa$ such that every node of $T_\alpha$ extends some element of $A \cap (T\restriction\alpha)$.
\end{claim}
\begin{proof}
Set $\Omega:=A$ and $p:=\{T,A\}$.
By our choice of $\cvec{C}$ and $\vec{B}$,
and recalling Definition~\ref{pbullet},
we now fix $\alpha\in E^\kappa_{\ge\chi}$ such that,  for all $C \in \mathcal C_\alpha$,
$$\sup\{ \gamma\in \dm{C} \mid \suc_1(\dm{C} \setminus \gamma) \subseteq B_{0}(\Omega,p,C) \} = \alpha.$$

Let $t \in T_\alpha$ be arbitrary.
As $\alpha\in E^\kappa_{\ge\chi}=\Gamma$, 
we appeal to Claim~\ref{tailwithsame}, 
to find $C\in\mathcal C_\alpha$ and $x \in T \restriction \dm C$ such that $t = \mathbf{b}^C_x$.
By our choice of $\alpha$, fix $\gamma\in \dm{C}\setminus\dom(x)$ with $\suc_1(\dm{C} \setminus \gamma) \subseteq B_{0}(\Omega,p,C)$.
Let $\beta:=\min(\dm{C}\setminus(\gamma+1))$, so that $\beta\in B_0(\Omega,p,C)$.
Recalling Definition~\ref{defbsets} and Proposition~\ref{motivate}(2),
we infer that $C(\beta)=A\cap(T\restriction\beta)$ and the latter is a maximal antichain in $T\restriction\beta$.

As $\gamma$ is the predecessor of $\beta$ in $\dm C$, and as $\beta\in B_0$,
we infer that $$b_x^C(\beta):=
\sealantichain(b^C_x(\gamma),T\restriction(\beta+1),C(\beta)).$$

Write $\bar x:=b^C_x(\gamma)$, $\bar T:=T\restriction(\beta+1)$ and $\mho:=C(\beta)$.
Recalling Definition~\ref{actions}(2), we consider the set:
$$Q:=\{ z\in \bar T_{\beta(\bar{T})}\mid \exists y\in\mho( \bar x\cup y\s z)\}.$$
By now, we know that
$$Q=\{ z\in T_{\beta}\mid \exists y\in A\cap (T\restriction\beta)( b^C_x(\gamma)\cup y\s z)\}.$$

As $A\cap (T\restriction\beta)$ is a maximal antichain in $T\restriction\beta$,
we infer from the normality of $T$ that $Q$ is nonempty, meaning that $b_x^C(\beta)$ extends some element of $A\cap (T\restriction\beta)$.
In particular, $t$ extends some element of $A\cap (T\restriction\alpha)$.
\end{proof}
This completes the proof.
\end{proof}

\subsection{A tree with an ascent path}\label{indexedtree} 

In \cite{rinot20}, a gallery of constructions of $\kappa$-Souslin trees with ascent paths was presented.
Each of those constructions assumed an instance of the form
$\p(\kappa,\ldots,\sigma)$ with $\sigma=\omega$, which, by Proposition~\ref{nonreflectcons}(2),
requires the existence of a nonreflecting stationary subset of $E^\kappa_\omega$.
In this section, we present a construction from a weaker instance of the proxy principle which is compatible with reflection,
thereby improving an old theorem of Baumgartner (see \cite{MR0732661}).
Note, however, that the objects we obtain here are not as complex and flexible as the ones obtained from the stronger instances in \cite{rinot20}.

\begin{definition}\label{defascentpath} Suppose that $T\s{}^{<\kappa}H_\kappa$ is a streamlined $\kappa$-tree.
We say that $\vec f =\langle f_\alpha\mid \alpha<\kappa\rangle$
is a \emph{$\mu$-ascent path} through $T$
iff for every pair $\alpha<\beta<\kappa$:
\begin{enumerate}
\item $f_\alpha$ is a function from $\mu$ to $T_\alpha$;
\item $\{ i<\mu \mid f_\alpha (i) \stree f_\beta (i) \}$ is co-bounded in $\mu$.
\end{enumerate}
\end{definition}
\begin{remark} In the general language of \cite[Definition~1.2]{rinot20}, the above is called an \emph{$\mathcal F^\bd_\mu$-ascent path}.
\end{remark}

\begin{thm}\label{souslinwithascentpath}
Suppose that $\kappa$ is $({<}\chi)$-closed for a given $\chi\in\reg(\kappa)$, and $\aleph_0\leq\mu<\kappa$. Let $\varsigma<\kappa$.
If $\p^\bullet(\kappa,\mu^{\ind},{\sq},1,\{E^\kappa_{\geq\chi}\},\kappa)$ holds, then there exists a
normal, prolific, $\varsigma$-splitting, $\chi$-complete $\kappa$-Souslin tree
admitting a $\mu$-ascent path.
\end{thm}
\begin{proof} 
Recalling Definition~\ref{indexedpb},
we fix sequences $\cvec C=\langle \mathcal C_\alpha\mid \alpha<\kappa\rangle$ and $\vec B=\langle B_\iota\mid\iota<\kappa\rangle$
together witnessing $\p^\bullet(\kappa, \mu^+,{\sq}, 1,\{E^\kappa_{\geq\chi}\},\kappa)$,
along with
$\langle i(\alpha)\mid \alpha<\kappa\rangle$ satisfying for every $\alpha\in\acc(\kappa)$:
\begin{enumerate}
\item there exists a canonical enumeration $\langle C_{\alpha,i}\mid i(\alpha)\le i<\mu\rangle$ of $\mathcal C_\alpha$;
\item for all $i\in[i(\alpha),\mu)$ and $\bar\alpha\in\acc(\dm{C}_{\alpha,i})$, $i\ge i(\bar\alpha)$ and $C_{\bar\alpha,i}\sq C_{\alpha,i}$;
\item $\langle \acc(\dm{C}_{\alpha,i})\mid i(\alpha)\le i<\mu\rangle$
is $\s$-increasing with $\acc(\alpha) = \bigcup_{i\in[i(\alpha),\mu)}\acc(\dm C_{\alpha,i})$.
\setcounter{condition}{\value{enumi}}
\end{enumerate}
Without loss of generality, we may also assume that:
\begin{enumerate}
\setcounter{enumi}{\value{condition}}
\item $0\in\dm{C}$ for every $C\in\bigcup_{\alpha\in\acc(\kappa)}\mathcal C_\alpha$.
\end{enumerate}

Now, the construction of the tree $T$ is \emph{identical} to that in Theorem~\ref{basicthm},
using $\cvec C$ and $\vec B$.
We are left with demonstrating that $T$ admits a $\mu$-ascent path.

Referring to the construction, recall that $\Gamma := E^\kappa_{\geq\chi}$ is stationary in $\kappa$,
and that for every $\alpha'\in\Gamma$, $C\in\mathcal{C}_{\alpha'}$ and $x \in T\restriction\dm{C}$,
$\mathbf{b}^C_x$ is an element of $T_{\alpha'}$ extending $x$.

For every $\alpha<\kappa$, let $\alpha':=\min(\Gamma\setminus\alpha)$,
and then define $f_\alpha:\mu\rightarrow T_\alpha$ via
$$f_\alpha(i):=\mathbf{b}^{C_{\alpha',\max\{i,i(\alpha')\}}}_\emptyset\restriction\alpha.$$
\begin{claim} $\langle f_\alpha\mid\alpha<\kappa\rangle$ forms a $\mu$-ascent path through $T$.
\end{claim}
\begin{proof}
Fix an arbitrary pair $\alpha<\beta$ of ordinals in $\kappa$.
Write $\alpha':=\min(\Gamma\setminus\alpha)$ and $\beta':=\min(\Gamma\setminus\beta)$.

$\br$ If $\alpha'=\beta'$, then we trivially get that $f_{\alpha}(i) \stree f_\beta (i)$ for all $i<\mu$.

$\br$ If $\alpha'<\beta'$, then $\alpha'\in\Gamma\cap\beta' \s \acc(\beta')$.
By Clause~(3), find a large enough $j<\mu$ such that $\alpha'\in\acc(\dm{C}_{\beta',i})$ for all $i\in[j,\mu)$.
For any such $i$, we have $C_{\alpha',i}\sq C_{\beta',i}$, so that,
by Dependencies~\ref{depends651}, $\mathbf{b}^{C_{\alpha',i}}_\emptyset=\mathbf{b}^{C_{\beta',i}}_\emptyset\restriction\alpha'$.
In particular, $f_{\alpha}(i) \stree f_\beta(i)$ on a tail of $i$'s.
\end{proof}
This completes the proof.
\end{proof}

Theorem~A is the case $(\lambda, \mu) := (\aleph_1, \aleph_0)$ of the following:

\begin{cor}\label{cor713} Suppose that $\lambda$ is an uncountable cardinal,
and $\square(\lambda^+)+\ch_\lambda$ holds.
For every $\mu\in\reg(\lambda)$ satisfying $\lambda^\mu=\lambda$, there exists a $\lambda^+$-Souslin tree admitting a $\mu$-ascent path.
\end{cor}
\begin{proof} By Corollary~\ref{corindex}, Corollary~\ref{upshotindx} and Theorem~\ref{souslinwithascentpath},
setting $\chi := \mu$.
\end{proof}

\begin{cor} In the Harrington--Shelah model \cite[Theorem~A]{MR783595}
in which every stationary subset of $E^{\aleph_2}_{\aleph_0}$ reflects,
there also exists an $\aleph_2$-Souslin tree with an $\omega$-ascent path.
\end{cor}
\begin{proof} The Harrington--Shelah model is a forcing extension of $L$ in which $\gch$ holds,
and the first Mahlo cardinal in $L$ becomes $\aleph_2$.
By standard results in inner model theory,
$\square(\aleph_2)$ holds, and hence Corollary~\ref{cor713} applies.
\end{proof}

\subsection{Omitting an ascending path}\label{subsection:omit}

In \cite[Definition~1.3]{MR3667758}, L\"ucke considered a weakening of a $\mu$-ascent path which he 
calls an \emph{ascending path of width $\mu$}.  
This is obtained by replacing Clause~(2) of Definition~\ref{defascentpath} by:
\begin{itemize}
\item[($2'$)] there are $i,j<\mu$ such that $f_\alpha(i) \stree f_\beta(j)$.
\end{itemize}

L\"ucke proved (see \cite[Theorem~1.9 and subsequent comment]{MR3667758}) that, assuming $\lambda^{<\lambda}=\lambda$, 
for every $\lambda^+$-Aronszajn tree $T$, the following are equivalent:
\begin{itemize}
\item for every ${\Lambda}<\lambda$, there is no ascending path of width $\Lambda$ through $T$;
\item $T$  is \emph{specializable}, that is, there exists a notion of forcing $\mathbb P$ that does not change the cardinal structure up to and including $\lambda^+$,
and such that, in $V^{\mathbb P}$, $T$ is the union of $\lambda$ many antichains.
\end{itemize}

We next show that by decreasing the value of $\nu$ in Theorem~\ref{basicthm}, we can get a $\kappa$-Souslin tree satisfying the additional property of omitting an ascending path.

\begin{thm}\label{omitpath}
Suppose that $\kappa$ is $({<}\chi)$-closed for a given $\chi\in\reg(\kappa)$. Let $\varsigma<\kappa$.
Given an infinite cardinal $\lambda<\kappa$, 
put $\mathcal S := \{ E^\kappa_{\geq\chi} \cap E^\kappa_{>\Lambda} \mid \Lambda<\lambda \}$.

If  $\p^\bullet(\kappa,\kappa,{\sqx^*},1, \mathcal S, 2)$ holds,
then there exists a normal, prolific, $\varsigma$-splitting, $\chi$\nobreakdash-complete, $\kappa$-Souslin tree $T$
such that, for all $\Lambda < \lambda$, there is no ascending path of width $\Lambda$ through $T$.
\end{thm}
\begin{proof}
Suppose $\p^\bullet(\kappa,\kappa,{\sqx^*},1, \mathcal S, 2)$ holds,
as witnessed by 
$\cvec{C} = \langle \mathcal C_\alpha\mid \alpha<\kappa\rangle$ and
$\vec B=\langle B_\iota \mid \iota< \kappa \rangle$.
We recursively construct a sequence $\langle T_\alpha\mid \alpha<\kappa\rangle$ of levels
whose union will ultimately be the desired tree $T$, as in the proof of Theorem~\ref{basicthm},
using $\Gamma:=E^\kappa_{\ge\chi}$.
In particular, for every $\alpha\in\Gamma$, we shall have:
\[
T_\alpha := \{ \mathbf{b}^C_x \mid C \in \mathcal C_\alpha, x \in T \restriction \dm{C}\}.
\]
The only difference is in the definition of $b^C_x$ in the case $\alpha\in\Gamma$, 
where stage (2)(a) is now done as follows.
\[
b_x^C(\beta) :=
\begin{cases}
\sealantichain(b_x^C(\beta^-),T\restriction (\beta+1),C(\beta)), &\text{if } \beta\in B_0;\\
\sealpath(b_x^C(\beta^-),T\restriction (\beta+1),C(\beta)), &\text{if } \beta\in B_1;\\
\defaultaction(b^C_x(\beta^-),T\restriction(\beta+1)),&\text{otherwise}.
\end{cases}
\]
In all cases, since $b^C_x(\beta^-)$ belongs to the normal tree $T \restriction (\beta+1)$,
we infer from the Extension Lemma (Lemma~\ref{extendfact}) that
$b_x^C(\beta)$ is an element of $T_\beta$ extending $b^C_x(\beta^-)$.
Now, it is clear that Dependencies~\ref{depends651} and
Claims \ref{coherent-from-weakest}, \ref{tailwithsame} and \ref{naantileft} are all valid,
so that 
$T := \bigcup_{\alpha < \kappa} T_\alpha$
 is a normal, prolific, $\varsigma$-splitting, $\chi$-complete streamlined $\kappa$-Souslin tree.
Thus, we are left with verifying the following.

\begin{claim}\label{claim4163}
For every ${\Lambda} < \lambda$, $T$ admits no ascending path of width $\Lambda$.
\end{claim}
\begin{proof}
Fix a nonzero ${\Lambda}<\lambda$, and suppose toward a contradiction that
$\vec f = \langle f_\alpha : \Lambda \to T_\alpha \mid \alpha<\kappa \rangle$ is an
ascending path of width $\Lambda$ through $T$.
Let $p:=\vec f$ and $\Omega:=\{ f_\alpha(\xi)\mid \xi<\Lambda, \alpha<\kappa\}$,
so that $p\in H_{\kappa^+}$ and $\Omega\s H_\kappa$.
By our choice of $\cvec{C}$ and $\vec{B}$,
and recalling Definition~\ref{pbullet},
we now fix $\alpha\in E^\kappa_{\ge\chi}\cap E^\kappa_{>\Lambda}$ above $\lambda$ such that
$\mathcal C_\alpha$ is a singleton, say, $\mathcal C_\alpha=\{C_\alpha\}$, and
$$\sup\{ \gamma\in \dm{C}_\alpha \mid \suc_1(\dm{C}_\alpha \setminus \gamma) \subseteq B_1(\Omega,p,C_\alpha) \} = \alpha.$$

For every $\xi<\Lambda$, since $f_\alpha(\xi) \in T_\alpha$ and $\alpha\in\Gamma$,
by the construction of the level $T_\alpha$,
we can fix $x_\xi \in T \restriction\dm{C}_\alpha$
such that $f_\alpha(\xi) = \mathbf{b}^{C_\alpha}_{x_\xi}$.
Now let
\[
\delta := \max\{\Lambda^+,\sup\{ \dom(x_\xi) \mid \xi<\Lambda \}\}.
\]
Since $\dom(x_\xi)<\alpha$ for every $\xi<\Lambda$,
and $\Lambda^+\leq\lambda<\alpha$,
it follows from $\cf(\alpha)>\Lambda$  that $\delta<\alpha$.
Thus, we may find a large enough $\beta\in\nacc(\dm{C}_\alpha)\cap B_1(\Omega,p,C_\alpha)$ 
such that $\beta^-:=\sup(\dm{C}_\alpha\cap\beta)$ is greater than $\delta$.
Denote $\mho:=C_\alpha(\beta)$.
It follows that, for all $\xi<\Lambda$:
\begin{itemize}
\item $\beta^-\in \dm{C}_\alpha$;
\item $x_\xi \in T \restriction (\dm{C}_\alpha\cap\beta^-)$;
\item $f_\alpha(\xi) = \mathbf{b}^{C_\alpha}_{x_\xi}$;
\item $C_\alpha(\beta)=\mho$.
\end{itemize}

Since $\beta^- +1 \leq\beta<\alpha$ and we have assumed that $\vec f$ is an ascending path of width $\Lambda$ through $T$,
we can fix $\bar\xi<\Lambda$ such that, for some $\xi<\Lambda$, $f_{\beta^- +1}(\xi)\stree f_\alpha(\bar\xi)$.
As $\dom(x_{\bar\xi}) \in\dm C_\alpha \cap \beta$,
we obtain $b^{C_\alpha}_{x_{\bar\xi}}(\beta)\stree \mathbf{b}^{C_\alpha}_{x_{\bar\xi}} = f_\alpha(\bar\xi)$.
Let us examine how $b^{C_\alpha}_{x_{\bar\xi}}(\beta)$ was chosen.

As $\beta\in\nacc(\dm{C}_\alpha)\cap B_1$ and $\beta^- = \sup(\dm C_\alpha\cap\beta)$, we have
$$b^{C_{\alpha}}_{x_{\bar\xi}}(\beta) = \sealpath(b_{x_{\bar\xi}}^{C_{\alpha}}(\beta^-),T\restriction (\beta+1),C_\alpha(\beta)).$$
Returning to Definition~\ref{actions}\eqref{actionsealpath},
we consider the following set:
\[
Q := \{ z \in T_\beta\mid b_{x_{\bar\xi}}^{C_\alpha}(\beta^-)\s z \ \&\ \forall y \in \mho[\dom(y)=\beta^-+1\implies y\nsubseteq z]\}.
\]

Recalling Definition~\ref{defbsets},
let us fix an elementary submodel $\mathcal M\prec H_{\kappa^+}$ satisfying:
$$\begin{array}{ccc}
p\in\mathcal M,&
\mathcal M\cap\kappa=\beta,\text{ and}&
\mathcal M\cap\Omega=\mho.
\end{array}$$

Clearly, $\beta$ is a limit ordinal, and  $\mho = \{ f_\gamma(\xi) \mid \xi<\Lambda, \gamma<\beta \}$.
In particular, $\{y\in \mho\mid \dom(y)=\beta^- +1\}= \{f_{\beta^- +1}(\xi) \mid \xi<\Lambda\}$.
Since $T$ is a prolific tree and $\beta^- >\delta \ge \Lambda^+$,
the node $b_{x_{\bar\xi}}^{C_\alpha}(\beta^-)$ must have at least
$\Lambda^+$ many immediate successors in $T_{\beta^-+1}$.
Thus we can fix $w \in T_{\beta^-+1}$ extending $b_{x_{\bar\xi}}^{C_\alpha}(\beta^-)$
that is distinct from $f_{\beta^- +1}(\xi)$ for every $\xi<\Lambda$.
By normality, we can extend $w$ to some element $z \in T_\beta$.
It follows that $Q$ is nonempty,
so that $b^{C_\alpha}_{x_{\bar\xi}}(\beta)$ was chosen from $Q$,
and in particular is incomparable with $f_{\beta^- +1}(\xi)$ for every $\xi<\Lambda$.
It follows that $f_\alpha(\bar\xi)$ is incomparable with $f_{\beta^- +1}(\xi)$ for every $\xi<\Lambda$,
contradicting our choice of $\bar\xi$ and completing the proof.
\end{proof}
This completes the proof.
\end{proof}

\subsection{Rigid} 
Compared to the previous subsection, here we strengthen $\mathcal R$ from ${\sqx^*}$ to ${\sqx}$, and gain rigidity as a result.

\begin{definition}[cf.~{\cite[\S V.1]{devlin1974souslin}}]\label{automorphism}
An \emph{automorphism} of a streamlined tree $T$ is a bijection $\Omega:T\leftrightarrow T$ satisfying $(x\stree y\iff\Omega(x)\stree\Omega(y))$.
The identity map is known as the \emph{trivial automorphism}.
A streamlined tree is said to be \emph{rigid} iff its only automorphism is the trivial one.
\end{definition}

Note that any automorphism $\Omega$ of a streamlined tree $T$ is \emph{level-preserving}, that is,
$\Omega[T_\beta] = T_\beta$ for every ordinal $\beta$.

\begin{cor}
If $\p^\bullet(\kappa,\kappa,{\sq},1,\{E^\kappa_{\ge\lambda}\},2)$ holds, then there exists a rigid $\kappa$-Souslin tree
admitting no ascending path of width $\Lambda$ for every $\Lambda<\lambda$.
\end{cor}
\begin{proof} Appeal to the next theorem with $\chi:=\aleph_0$.
\end{proof}

\begin{thm}\label{rigid}
Suppose that $\kappa$ is $({<}\chi)$-closed for a given $\chi\in\reg(\kappa)$. Let $\varsigma<\kappa$.
Given an infinite cardinal $\lambda<\kappa$, 
put $\mathcal S := \{ E^\kappa_{\geq\chi} \cap E^\kappa_{>\Lambda} \mid \Lambda<\lambda \}$.

If  $\p^\bullet(\kappa,\kappa,{\sqx},1, \mathcal S, 2)$ holds,
then there exists a normal, prolific, $\varsigma$-splitting, $\chi$-complete, rigid $\kappa$-Souslin tree $T$
such that, for all $\Lambda < \lambda$, there is no ascending path of width $\Lambda$ through $T$.
\end{thm}
\begin{proof}
Suppose $\p^\bullet(\kappa,\kappa,{\sqx},1, \mathcal S, 2)$ holds.
Fix a partition of $\kappa$ into stationary sets, $\vec B=\langle B_\iota \mid \iota< \kappa \rangle$,
such that $B_2\cap\nacc(\kappa)=\emptyset$. 
By Corollary~\ref{Pbullet-equiv}, then, fix $\cvec{C} = \langle \mathcal C_\alpha\mid \alpha<\kappa\rangle$ 
that, together with $\vec B$, witnesses $\p^\bullet(\kappa,\kappa,{\sqx},1, \mathcal S, 2)$.
We recursively construct a sequence $\langle T_\alpha\mid \alpha<\kappa\rangle$ of levels
whose union will ultimately be the desired tree $T$, as in the proof of Theorem~\ref{basicthm},
using $\Gamma:=E^\kappa_{\ge\chi}$.
In particular, for every $\alpha\in\Gamma$, we shall have:
\[
T_\alpha := \{ \mathbf{b}^C_x \mid C \in \mathcal C_\alpha, x \in T \restriction \dm{C}\}.
\]
The only difference is in the definition of $b^C_x$ in the case $\alpha\in\Gamma$, 
where stage (2)(a) is now done as follows.
\[
b_x^C(\beta) :=
\begin{cases}
\sealantichain(b_x^C(\beta^-),T\restriction (\beta+1),C(\beta)), &\text{if } \beta\in B_0;\\
\sealpath(b_x^C(\beta^-),T\restriction (\beta+1),C(\beta)), &\text{if } \beta\in B_1;\\
\sealautomorphism(b_x^C(\beta^-),T\restriction (\beta+1),
\langle b_{y}^C(\beta^-)\mid y\in T\restriction(\dm{C}\cap\beta)\rangle,C(\beta)), &\text{if } \beta\in B_2;\\
\defaultaction(b^C_x(\beta^-),T\restriction(\beta+1)),&\text{otherwise}.
\end{cases}
\]
In all cases, since $b^C_x(\beta^-)$ belongs to the normal tree $T \restriction (\beta+1)$ and $T\restriction\beta$ is ever-branching,
we infer from the Extension Lemma (Lemma~\ref{extendfact}) that
$b_x^C(\beta)$ is an element of $T_\beta$ extending $b^C_x(\beta^-)$.

The following is obvious.

\begin{dependencies}\label{depends6141}
For any two consecutive points $\beta^-<\beta$ of $\dom(b^C_x)$,
the value of $b^C_x(\beta)$ is completely determined by
$b^C_x(\beta^-)$, $T\restriction(\beta+1)$, $C(\beta)$,
and the map $y\mapsto b_y^C(\beta^-)$ over $T\restriction(\dm{C}\cap\beta)$.
\end{dependencies}

Note that, unlike Dependencies~\ref{depends651}, which only involves the value of $C$ at $\beta$,
here we have a dependency on $C\restriction\beta$ as well. The price we pay for this added dependency is that,
in order to carry out the proof of Claim~\ref{coherent-from-weakest} in this context,
we require $\cvec{C}$ to satisfy the stronger $\sqx$-coherence, rather than merely $\sqx^*$-coherence.

Claims \ref{coherent-from-weakest}, \ref{tailwithsame} and \ref{naantileft} remain valid,
so that 
$T := \bigcup_{\alpha < \kappa} T_\alpha$
 is a normal, prolific, $\varsigma$-splitting, $\chi$-complete streamlined $\kappa$-Souslin tree.
Also, Claim~\ref{claim4163} remains valid, so that 
for all $\Lambda < \lambda$, there is no ascending path of width $\Lambda$ through $T$.
Thus, we are left with proving the following.

\begin{claim} $T$ is rigid.
\end{claim}
\begin{proof}
Suppose $\Omega: T \to T$ is a nontrivial automorphism, and we will derive a contradiction.
Let $p:=\{\Omega,T\}$, so that $p\in H_{\kappa^+}$ and $\Omega\s H_\kappa$.
By our choice of $\vec{B}$ and $\cvec{C}$,
and recalling Definition~\ref{pbullet},
we now fix $\alpha\in E^\kappa_{\ge\chi}$ such that $\mathcal C_\alpha$ is a singleton, say, $\mathcal C_\alpha=\{C_\alpha\}$, and
$$\sup\{ \gamma\in \dm{C}_\alpha \mid \suc_1(\dm{C}_\alpha \setminus \gamma) \subseteq B_2(\Omega,p,C_\alpha) \} = \alpha.$$

By Proposition~\ref{motivate}(3), for any $\beta\in B_2(\Omega,p,C_\alpha)$,
$C_\alpha(\beta)=\Omega\restriction(T\restriction\beta)$ and the latter is a nontrivial automorphism of $T\restriction\beta$.
In particular, $\Omega\restriction (T\restriction\alpha)$ is a nontrivial automorphism of $T\restriction\alpha$.
Thus, by normality, we can let $x_0$ be the $\lhd_\kappa$-least element of $T \restriction\dm{C}_\alpha$
such that $\Omega (x_0) \neq x_0$.

As $\alpha\in\Gamma$, by the nature of the construction of $T_\alpha$,
it makes sense to consider the particular node ${\mathbf{b}}^{C_\alpha}_{x_0}$.
Since ${\mathbf{b}}^{C_\alpha}_{x_0} \in T_\alpha$ and $\Omega$ is an automorphism of $T$, we have $\Omega({\mathbf{b}}^{C_\alpha}_{x_0})\in T_\alpha$,
so that we may choose some $x \in T \restriction \dm{C}_\alpha$ such that $\Omega({\mathbf{b}}^{C_\alpha}_{x_0})={\mathbf{b}}^{C_\alpha}_x$.\footnote{Recall that $C_\alpha$ is the sole element of $\mathcal C_\alpha$.}

Fix a large enough $\beta \in \nacc (\dm{C}_\alpha) \cap B_2(\Omega,p,C_\alpha)$ such that $\sup(\dm{C}_\alpha\cap\beta)>\max\{ \dom(x_0), \dom(x) \}$.
Then $b^{C_\alpha}_{x_0}(\beta) = \mathbf{b}^{C_\alpha}_{x_0} \restriction\beta$
and $b^{C_\alpha}_{x}(\beta) = \mathbf{b}^{C_\alpha}_{x} \restriction\beta$.
Since $\Omega$ is an automorphism, it follows that
\[
\Omega(b^{C_\alpha}_{x_0}(\beta)) = \Omega(\mathbf{b}^{C_\alpha}_{x_0} \restriction\beta) =
\Omega(\mathbf{b}^{C_\alpha}_{x_0}) \restriction\beta = \mathbf{b}^{C_\alpha}_{x} \restriction\beta =
b^{C_\alpha}_{x}(\beta).
\]

Write $\beta^-:=\sup(\dm{C}_\alpha\cap\beta)$, $\bar T:=T\restriction (\beta+1)$, $\bar b:=\langle b_{y}^{C_\alpha}(\beta^-)\mid y\in T\restriction(\dm{C}_\alpha\cap\beta)\rangle $ and $\mho:=C_\alpha(\beta)$.
As $\beta\in\nacc(\dm{C}_\alpha)\cap B_2$, we infer:
\begin{itemize}
\item $b^{C_\alpha}_{x_0}(\beta)=\sealautomorphism(b_{x_0}^{C_\alpha}(\beta^-),\bar T,\bar b,\mho)$;
\item $b^{C_\alpha}_x(\beta)=\sealautomorphism(b_x^{C_\alpha}(\beta^-),\bar T,\bar b,\mho)$.
\end{itemize}

Consider an arbitrary $z\in\{b_{x_0}^{C_\alpha}(\beta^-),b_{x}^{C_\alpha}(\beta^-)\}$.
Recall Definition~\ref{actions}\eqref{autoaction}, and let us analyze $\sealautomorphism(z,\bar T,\bar b,\mho)$.
We have $z\in\bar T$, and $\dom(z)=\beta^-<\beta=\beta(\bar T)$.
Since $\beta\in B_2(\Omega,p,C_\alpha)$, $\beta$ is a limit ordinal
and $\mho$ is the automorphism $\Omega\restriction(\bar T\restriction\beta)$ of $\bar{T}\restriction\beta$. 
Also, $\bar b$ is a partial function from $\bar T \restriction(\beta(\bar T))$ to $\bar T \restriction(\beta(\bar T))$.
Since $x_0\in\dom(\bar b)$ and $\Omega(x_0)\neq x_0$,
the set $Q:=\{ x_0\in \dom(\bar b)\mid \mho(x_0)\neq x_0\}$ is nonempty.
By the very choice of $x_0$, moreover, $x_0=\min(Q,{\lhd_\kappa})$.
Write $\tilde x_0:=\defaultaction( \bar b(x_0),\bar T)$ and
$\tilde y_0: = \bigcup_{\gamma < \beta(\bar T)} \mho(\tilde x_0 \restriction \gamma)$.
Since $\beta(\bar T)=\beta$ is a limit ordinal and $\mho$ is the automorphism $\Omega\restriction(T\restriction\beta)$,
we obtain $\tilde y_0=\Omega(\tilde x_0)$.
Altogether:
\begin{itemize}
\item $b^{C_\alpha}_{x_0}(\beta)=\sealautomorphism(b_{x_0}^{C_\alpha}(\beta^-),\bar T,\bar b,\mho)=\defaultaction(b_{x_0}^{C_\alpha}(\beta^-),\bar T\setminus\{\tilde{y}_0\})$.
\item $b^{C_\alpha}_x(\beta)=\sealautomorphism(b_{x}^{C_\alpha}(\beta^-),\bar T,\bar b,\mho)=\defaultaction(b_{x}^{C_\alpha}(\beta^-),\bar T\setminus\{\tilde{y}_0\})$.
\end{itemize}

Since  $x_0\stree b_{x_0}^{C_\alpha}(\beta^-) = \bar b(x_0)\stree \defaultaction(\bar b(x_0),\bar T) = \tilde{x}_0$, we have $\Omega(x_0)\stree\Omega(\tilde{x}_0) = \tilde{y}_0$.
Thus, by 
the fact that $\Omega(x_0)$ and $x_0$ are two distinct elements of the same level, we infer that $b_{x_0}^{C_\alpha}(\beta^-)$ is incomparable with $\tilde{y}_0$.
Consequently,
\begin{align*}
b^{C_\alpha}_{x_0}(\beta)&=\defaultaction(b_{x_0}^{C_\alpha}(\beta^-),\bar T\setminus\{\tilde{y}_0\}) \\&=\defaultaction(b_{x_0}^{C_\alpha}(\beta^-),\bar T)=\defaultaction(\bar b(x_0),\bar T) = \tilde{x}_0,
\end{align*}
so that
$$\Omega(b^{C_\alpha}_{x_0}(\beta)) = \Omega(\tilde{x}_0) =\tilde{y}_0.$$
As
$b^{C_\alpha}_x(\beta)=\defaultaction(b_{x}^{C_\alpha}(\beta^-),\bar T\setminus\{\tilde{y}_0\})$,
we obtain $b^{C_\alpha}_x(\beta)\in \bar T\setminus\{\tilde{y}_0\}$. So
$$\Omega(b^{C_\alpha}_{x_0}(\beta))\neq b^{C_\alpha}_x(\beta).$$

This is a contradiction.
\end{proof}
This completes the proof.
\end{proof}

Theorem~B is the case $\lambda:=\aleph_1$ of the following corollary.
\begin{cor} Suppose that $\kappa$ is a strongly inaccessible cardinal, $\lambda<\kappa$ is regular and uncountable,
and there exists a sequence $\langle A_\alpha\mid \alpha\in S\rangle$ such that:
\begin{itemize}
\item $S$ is a nonreflecting stationary subset of $E^\kappa_{\ge\lambda}$;
\item For every $\alpha\in S$, $A_\alpha$ is a cofinal subset of $\alpha$;
\item For every cofinal $B\s\kappa$, there exists $\alpha\in S$
for which 
$$\{\delta<\alpha\mid \min(A_\alpha\setminus(\delta+1))\in B\}$$
is stationary in $\alpha$.
\end{itemize}

Then there is a rigid $\kappa$-Souslin tree admitting no ascending path of width $<\lambda$.
\end{cor}
\begin{proof} By Theorem~\ref{cor430}, in particular, $\p^-(\kappa,\kappa,{\sq},1,\{E^\kappa_{\ge\lambda}\},2)$ holds.
As $\kappa$ is strongly inaccessible, it follows from Theorem~\ref{mainpbullet} that $\p^\bullet(\kappa,\kappa,{\sq},1,\{E^\kappa_{\ge\lambda}\},2)$ holds, as well.
Now appeal to Theorem~\ref{rigid} with $\chi:=\aleph_0$.
\end{proof}

\subsection{Free}
\label{subsection:free}

Recall that the square of a tree $(T,{<_T})$ is the poset $({\hat T},{<_{\hat T}})$, where
\begin{itemize}
\item ${\hat T}:=\{ (t,t')\mid t,t'\in T, \height(t)=\height(t')\}$,
\item $(s,s')<_{\hat T}(t,t')$ iff $s<_Tt$ and $s'<_Tt'$.
\end{itemize}

By a theorem of Kurepa (see \cite[Lemma~14.14, Theorem~14.15]{just1997discovering}),
the square of an ever-branching tree of cardinality $\kappa$ has an antichain of size $\kappa$.
In particular:
\begin{fact}[Kurepa, \cite{MR0049973}]\label{square-not-Souslin}  The square of a $\kappa$-Souslin tree is not $\kappa$-Souslin.
\end{fact}
In contrast, as we shall soon see, the product of two $\kappa$-Souslin trees may still be Souslin.
In fact, this is also true for longer products:

\begin{definition}\label{def-classical-product-tree} 
For a sequence of trees $\langle \mathcal{T}^i \mid i<\tau \rangle$ 
with $\mathcal T^i = (T^i, {<_{T^i}})$ for each $i<\tau$,
the product
$\bigotimes_{i<\tau} \mathcal{T}^i$
is defined to be the tree $(\hat{T}, {<_{\hat{T}}})$,
where:
\begin{itemize}
\item
$\hat{T}$ consists of all $\vec t\in \prod_{i<\tau}T^i$
that are \emph{level sequences}, that is, 
$i\mapsto\height_{\mathcal T^i}(\vec t(i))$ is constant over $\tau$;
\item $\vec{s} <_{\hat{T}} \vec{t}$ iff $\vec s(i) <_{T^i} \vec t(i)$ for every $i<\tau$.
\end{itemize}
\end{definition}

Now, consider a fixed tree $\mathcal T=(T, {<_T})$.
For a node $s \in T$, let us denote $T(s):=\{x\in T\mid x\text{ is comparable with }s\}$.
That is, $T(s)=s_\downarrow\cup\{s\}\cup\cone s$.
Then, for any level sequence $\vec{s} = \langle s_i \mid i<\tau \rangle$ in $\bigotimes_{i<\tau}\mathcal T$,
we can consider the \emph{derived tree} $\bigotimes_{i<\tau}(T(s_i),{<_T})$ and ask whether it is $\kappa$-Souslin.

In order to bring the notion of derived trees within our framework of streamlined trees,
we work with a slightly different definition, as follows.

\begin{notation}[$i^\text{th}$-component]\label{notationcomp}
For every function $x:\alpha\rightarrow{}^\tau H_\kappa$ and every $i<\tau$,
let $(x)_i:\alpha\rightarrow H_\kappa$ denote the sequence $\langle x(\varepsilon)(i)\mid \varepsilon<\alpha\rangle$.
\end{notation}

\begin{defn}[Derived tree, {\cite[Definition~4.4]{paper32}}]\label{derived} 
Suppose that $T \s {}^{<\kappa}H_\kappa$ is a streamlined tree,
and $\vec{s}=\langle s_i \mid i < \tau \rangle$ is a nonempty sequence of elements of $T$.
We let $T(\vec{s})$ stand for the collection
of all $x:\alpha\rightarrow {}^\tau H_\kappa$ such that:
\begin{itemize}
\item $\alpha$ is an ordinal; and
\item for all $i<\tau$, $(x)_i\cup s_i$ is in $T$.
\end{itemize}
\end{defn}

Note that for a given level sequence $\vec{s} = \langle s_i \mid i<\tau \rangle$ 
of elements of a streamlined tree $T$,
the tree $(T(\vec{s}),{\stree})$ is order-isomorphic to 
$\bigotimes_{i<\tau}(T(s_i),{\stree})$.\footnote{Recall the last bullet of Lemma~\ref{streamlined-basics}.}

\begin{fact}[cf.~{\cite[Lemma~4.7]{paper32}}]\label{derivedtree} 
Suppose that $T \s {}^{<\kappa}H_\kappa$ is a streamlined tree.
For every $\chi<\kappa$ and every nonempty sequence $\vec s\in{}^{<\chi}T$:
\begin{enumerate}
\item $T(\vec s) \s {}^{<\kappa}H_\kappa$ is a streamlined tree;
\item If $T$ is normal, then so is $T(\vec s)$;
\item If $T$ is a normal $\kappa$-tree and $\kappa$ is $({<}\chi)$-closed, then $T(\vec s)$ is a $\kappa$-tree;
\item If $T$ is ever-branching, then so is $T(\vec s)$;
\item If $\vec{s} = \langle \emptyset \rangle$, then $(T(\vec{s}),{\stree})$ is order-isomorphic to $(T,{\stree})$
via the isomorphism $x \mapsto (x)_0$.
Likewise, if $\vec{s} = \langle \emptyset, \emptyset \rangle$, 
then $(T(\vec{s}),{\stree})$ is order-isomorphic to the square of $(T,{\stree})$.
\end{enumerate}
\end{fact}

By Fact~\ref{square-not-Souslin},
if a level-sequence $\vec{s}$ is not injective, then the derived tree $T(\vec{s})$ will not be Souslin.
Altogether, this leads us to the following definition.

\begin{defn}[{\cite[Definition~4.5]{paper32}}] 
A streamlined $\kappa$-Souslin tree $T$ is said to be \emph{$\chi$-free} if for every $\delta < \kappa$ and every nonempty injection
$\vec s\in {}^{<\chi} T_\delta$, the derived tree $T(\vec s)$ is again a $\kappa$-Souslin tree.
\end{defn}

\begin{remarks}
\begin{enumerate}
\item By \cite[Lemma~A.7(1)]{rinot20}, if there exists a $\chi$-free $\kappa$-Souslin tree, then $\kappa$ is $({<}\chi)$-closed.
\item Instead of $\aleph_0$-free, we simply say \emph{free}.
\end{enumerate}
\end{remarks}

\begin{cor}
If $\p^\bullet(\kappa,\kappa,{\sq},\kappa,\{\kappa\}, 2)$ holds,
then there is a free $\kappa$-Souslin tree.
\end{cor}
\begin{proof} Appeal to the next theorem with $\chi:=\aleph_0$.
\end{proof}

The main theorem of this subsection reads as follows.
\begin{thm}\label{free}
Suppose that $\kappa$ is $({<}\chi)$-closed for a given $\chi\in\reg(\kappa)$. Let $\varsigma<\kappa$.

If $\p^\bullet(\kappa,\kappa,{\sqx},\kappa,\allowbreak\{E^\kappa_{\geq\chi}\}, 2)$ holds,
then there exists a
normal, prolific, $\varsigma$-splitting, $\chi$-complete, $\chi$-free $\kappa$-Souslin tree.
\end{thm}
\begin{proof}
Suppose $\p^\bullet(\kappa,\kappa,{\sqx},\kappa, \{E^\kappa_{\geq\chi}\}, 2)$ holds,
as witnessed by $\langle \mathcal C_\alpha\mid \alpha<\kappa\rangle$
and $\langle B_\iota \mid \iota< \kappa \rangle$.
Let $\pi:\kappa\rightarrow\kappa$ be such that $\beta\in B_{\pi(\beta)}$ for all $\beta<\kappa$.
Let $\phi:\kappa\leftrightarrow H_\kappa$ witness the isomorphism $(\kappa,\in)\cong(H_\kappa,{\lhd_\kappa})$.
Put $\psi:=\phi\circ\pi$.

We recursively construct a sequence $\langle T_\alpha\mid \alpha<\kappa\rangle$ of levels
whose union will ultimately be the desired tree $T$, as in the proof of Theorem~\ref{basicthm},
using $\Gamma:=E^\kappa_{\ge\chi}$.
In particular, for every $\alpha\in\Gamma$, we shall have:
\[
T_\alpha := \{ \mathbf{b}^C_x \mid C \in \mathcal C_\alpha, x \in T \restriction \dm{C}\}.
\]
The only difference is in the definition of $b^C_x$ in the case $\alpha\in\Gamma$, 
where stage (2)(a) is now done as follows.
$$b_x^C(\beta):=
\begin{cases}
\free(b^C_x(\beta^-),T\restriction(\beta+1),\langle b^C_{\psi(\beta)(\xi)}(\beta^-)\mid\xi\in\dom(\psi(\beta))\rangle,C(\beta)),\\
\makebox[195pt][l]\hfill\text{if }\psi(\beta) \in {}^{<\chi}(T\restriction (\dm{C}\cap \beta));\\
\defaultaction(b^C_x(\beta^-),T\restriction(\beta+1)),\\
\hfill\text{otherwise}.
\end{cases}$$

In all cases, since $b^C_x(\beta^-)$ belongs to the normal tree $T \restriction (\beta+1)$,
we infer from the Extension Lemma (Lemma~\ref{extendfact}) that
$b_x^C(\beta)$ is an element of $T_\beta$ extending $b^C_x(\beta^-)$.

The following is obvious.

\begin{dependencies}
For any two consecutive points $\beta^-<\beta$ of $\dom(b^C_x)$,
the value of $b^C_x(\beta)$ is completely determined by
$b^C_x(\beta^-)$, $T\restriction(\beta+1)$, $C(\beta)$, $\psi(\beta)$,
and the map $y\mapsto b_y^C(\beta^-)$ over $T\restriction(\dm{C}\cap\beta)$.
\end{dependencies}

It follows that Claims \ref{coherent-from-weakest} and \ref{tailwithsame} remain valid,
where, as in the proof of Theorem~\ref{rigid}, 
we appeal to the stronger $\sqx^*$-coherence of $\cvec{C}$ in order to prove Claim~\ref{coherent-from-weakest}
despite the increased dependency.
It follows that $T := \bigcup_{\alpha < \kappa} T_\alpha$
 is a normal, prolific, $\varsigma$-splitting, $\chi$-complete streamlined $\kappa$-tree.
Thus, we are left with verifying the following.

\begin{claim} $T$ is a $\chi$-free $\kappa$-Souslin tree.
\end{claim}
\begin{proof}
We rely throughout on the various clauses of Fact~\ref{derivedtree}.
$T$ is $\kappa$-Souslin iff $T(\langle\emptyset\rangle)$ is $\kappa$-Souslin,
and hence it suffices to prove $\chi$-freeness.
Fix an arbitrary nonzero ordinal
$\tau < \chi$, some $\delta < \kappa$, and an injection
$\vec w=\langle w_\xi \mid \xi < \tau \rangle \in {}^\tau T_\delta.$
We need to show that the derived tree $\hat T:=T(\vec w)$ is a $\kappa$-Souslin tree.
For this, fix an arbitrary maximal antichain $A$ in $\hat T$. 
Let $\Omega:=A$ and $p:=\{A,T\}$.
Fix a transversal $\langle C_\alpha\mid\alpha<\kappa\rangle\in\prod_{\alpha<\kappa}\mathcal C_\alpha$.
Recalling Definition~\ref{pbullet}, we now infer that the following set is stationary:
\[
W: = \{ \alpha \in E^\kappa_{\geq\chi} \mid \mathcal C_\alpha=\{C_\alpha\} \text{ and }
    \forall i < \alpha [\sup ( \nacc (\dm C_\alpha ) \cap B_i(\Omega,p, C_\alpha) ) = \alpha ] \}.
\]

Since $\{ \alpha \in E^\kappa_{\ge\chi} \mid {}^{<\chi} (T \restriction \alpha) \subseteq \phi[\alpha]\}$ is a club relative to $E^\kappa_{\ge\chi}$,
let us fix an ordinal $\alpha \in W\setminus(\delta+1)$ such that ${}^{<\chi} (T \restriction \alpha) \subseteq \phi[\alpha]$.

Let $v\in\hat{T}_\alpha$ be arbitrary, and we shall show that it extends some element of $A\cap(\hat{T}\restriction\alpha)$.

Let $\xi<\tau$ be arbitrary. As $(v)_\xi\in T_\alpha$ and $\alpha\in W \s\Gamma$,
by Claim~\ref{tailwithsame}, there exists some $x_\xi \in T \restriction \dm{C}_\alpha$ such that $(v)_\xi = \mathbf{b}^{C_\alpha}_{x_\xi}$.
Furthermore, by $\tau<\chi\le\cf(\alpha)$
and Claim~\ref{tailwithsame}, we may fix a large enough $\gamma\in \dm{C}_\alpha\setminus\delta$ along with $\vec x := \langle x_\xi \mid \xi < \tau \rangle\in {}^{\tau}T_\gamma$
such that $(v)_\xi = \mathbf{b}^{C_\alpha}_{x_\xi}$ for all $\xi<\tau$.
Of course, $\vec x \in {}^{<\chi} (T \restriction \alpha)\s \phi[\alpha]$.
In particular, $i := \phi^{-1} (\vec x)$ is $< \alpha$,
so that our choice of $\alpha$ guarantees that $\sup (\nacc(\dm C_\alpha) \cap B_i(\Omega,p, C_\alpha) ) = \alpha$.

Fix a large enough $\beta \in \nacc(\dm C_\alpha) \cap B_i(\Omega,p, C_\alpha) $ with
$\gamma \leq \sup(\dm C_\alpha\cap\beta)<\beta < \alpha$.
Write $\bar T:=T\restriction(\beta+1)$, $\beta^-:=\sup(\dm C_\alpha\cap\beta)$,
and $\vec b:=\langle b^{C_\alpha}_{x_\xi}(\beta^-)\mid \xi<\tau\rangle$.
As $\beta\in B_i$, we infer
\[
\psi(\beta) = \phi(\pi(\beta)) = \phi(i)= \vec x \in {}^\tau(T_\gamma) \s {}^{<\chi}(T\restriction (\dm{C}_\alpha \cap \beta)),
\]
so that, for all $x\in T\restriction(\dm{C}_\alpha\cap\beta)$,
\[
b_x^{C_\alpha}(\beta) = \free( b_x^{C_\alpha}(\beta^-), \bar T, \vec b,C_\alpha(\beta)).
\]

Recalling Definition~\ref{actions}\eqref{actionfree}, we must analyze $z:=\sealantichain(\emptyset,\bar T(\vec b),C_\alpha(\beta))$.
Clearly, $\beta(\bar T(\vec b))=\beta(\bar T)=\beta$ and $\bar T(\vec b)=T(\vec b)\restriction (\beta+1)$. Thus, consider the set
$$Q:=\{ z\in T(\vec b)_\beta\mid \exists y\in C_\alpha(\beta)~( y\s z)\}.$$

Let $a:\beta^-\rightarrow{}^\tau H_\kappa$ be the unique function satisfying $(a)_\xi=\vec b(\xi)$ for all $\xi<\tau$.
For all $\xi<\tau$, we have $x_\xi\s b^{C_\alpha}_{x_\xi}(\beta^-)= (a)_\xi \stree \mathbf{b}_{x_\xi}^{C_\alpha}=(v)_\xi$, so that $a = v \restriction \beta^-$ is in $\hat T \restriction\beta$.
Since $\beta \in B_i(A,\{A,T\}, C_\alpha)$, Proposition~\ref{motivate}(2) entails that $C_\alpha(\beta) = A \cap (\hat T\restriction\beta)$ is a maximal antichain in $\hat T\restriction\beta$,
and we may pick some $y\in C_\alpha(\beta)$ such that $a\cup y\in\hat T\restriction\beta$. Since $T$ is normal, $\hat T$ is normal,
and hence there exists some $z'\in \hat T_\beta$ such that $a\cup y\s z'$. 
As $a\s z'$, we obtain $z'\in T(\vec b)$. Therefore, $Q$ is nonempty,
and $z=\min(Q,{\lhd_\kappa})$.

Notice that $\beta>\gamma\geq\delta$.
Since $\{ w_\xi\mid \xi<\tau\}$ is an antichain and $w_\xi \stree (z)_\xi$ for every $\xi<\tau$,
it follows that $\{(z)_\xi \mid \xi<\tau\}$ is also an antichain.  Thus,
we infer that, for all $\xi<\tau$, $$\{ \xi'<\tau\mid b_{x_\xi}^{C_\alpha}(\beta^-)\cup(z)_{\xi'}\in T\}=\{\xi\}.$$
Consequently, for all $\xi<\tau$, $b_{x_\xi}^{C_\alpha}(\beta) = \free( b_{x_\xi}^{C_\alpha}(\beta^-), \bar T, \vec b,C_\alpha(\beta))=(z)_\xi$.

Let $y \in A \cap(\hat T\restriction\beta)$ be a witness to the choice of $z$.
Then for all $\xi < \tau$ we infer
\[(y)_\xi \stree (z)_\xi=  b^{C_\alpha}_{x_\xi} (\beta) \stree \mathbf{b}^{C_\alpha}_{x_\xi} = (v)_\xi.\qedhere\]
\end{proof}
This completes the proof.
\end{proof}

Of course, there are cases in which one wants to construct free trees that are slim (or, moreover, regressive)
rather than complete.
For this, note that the above construction would be equally successful had we set $\Gamma := \acc(\kappa)$,
abandoning the goal of completeness,
provided that we assume the stronger coherence relation $\sq$ instead of $\sqx$.
In addition, if we require $|\mathcal{C}_\alpha|<\aleph_1$ for every $\alpha<\kappa$,
then the tree will satisfy $|T_\alpha|\leq|\alpha|$ for all infinite $\alpha<\kappa$.
Altogether, we have the following.\footnote{For more results in this direction, see \cite[\S6]{rinot20}.}

\begin{thm}
Suppose that $\kappa$ is $({<}\chi)$-closed for a given $\chi\in\reg(\kappa)$. 

If $\p^\bullet(\kappa,\aleph_1,{\sq},\kappa,\allowbreak\{E^\kappa_{\geq\chi}\}, 2)$ holds,
then there exists a
normal, prolific, slim, $\chi$-free $\kappa$-Souslin tree.\qed
\end{thm}

\subsection{Uniformly homogeneous and uniformly coherent}\label{homogeneoussection}
\begin{definition} For two elements $x,y$ of $H_\kappa$, we define $x*y$ to be the empty set,
unless $x,y\in{}^{<\kappa}H_\kappa$ with $\dom(x)<\dom(y)$, in which case $x*y:\dom(y)\rightarrow H_\kappa$ is defined by stipulating:
$$(x*y)(\varepsilon):=\begin{cases}x(\varepsilon),&\text{if }\varepsilon\in\dom(x);\\
y(\varepsilon),&\text{otherwise.}\end{cases}$$
\end{definition}

\begin{definition} A streamlined $\kappa$-tree $T$ is said to be:
\begin{itemize}
\item \emph{coherent} iff $\{ \varepsilon\in\dom(x)\cap\dom(y)\mid x(\varepsilon)\neq y(\varepsilon)\}$ is finite
for all $x,y\in T$;
\item \emph{uniformly homogeneous} iff for all $y\in T$
and $x\in T\restriction\dom(y)$, $x*y$ is in $T$;
\item \emph{uniformly coherent} iff it is coherent and uniformly homogeneous.
\end{itemize}
\end{definition}

\begin{cor}
If $\p^\bullet(\kappa,\kappa,{\sq},\kappa, \{\kappa\}, \kappa)$ holds,
then there exists a uniformly homogeneous $\kappa$-Souslin tree.
\end{cor}
\begin{proof} Appeal to the next theorem with $\chi:=\aleph_0$.
\end{proof}

\begin{thm}\label{homogeneous}
Suppose that $\kappa$ is $({<}\chi)$-closed for a given $\chi\in\reg(\kappa)$. Let $\varsigma<\kappa$.

If $\p^\bullet(\kappa,\kappa,{\sqx},\kappa, \{E^\kappa_{\geq\chi}\}, \kappa)$ holds,
then there exists a normal, prolific, $\varsigma$-splitting, $\chi$-complete, uniformly homogeneous $\kappa$-Souslin tree.
\end{thm}
\begin{proof} 

Suppose $\p^\bullet(\kappa,\kappa,{\sqx},\kappa,\{E^\kappa_{\geq\chi}\},\kappa)$ holds,
as witnessed by $\langle \mathcal C_\alpha\mid \alpha<\kappa\rangle$ and $\langle B_\iota\mid \iota<\kappa\rangle$.
We may assume that $0\in\dm{C}$ for every $C\in\bigcup_{\alpha\in\acc(\kappa)}\mathcal C_\alpha$.
Let $\pi:\kappa\rightarrow\kappa$ be such that $\beta\in B_{\pi(\beta)}$ for all $\beta<\kappa$.
Let $\phi:\kappa\leftrightarrow H_\kappa$ witness the isomorphism $(\kappa,\in)\cong(H_\kappa,{\lhd_\kappa})$.
Put $\psi:=\phi\circ\pi$.

We recursively construct a sequence $\langle T_\alpha\mid \alpha<\kappa\rangle$ of levels
whose union will ultimately be the desired tree $T$.
Denote $\Gamma:=E^\kappa_{\ge\chi}$.
For every $\alpha\in \Gamma$ and $C\in\mathcal C_\alpha$, we shall define an $\alpha$-branch $\mathbf{b}^C$ through $T\restriction\alpha$;
then, we will ensure that, for every $\alpha\in\acc(\kappa)$:
\begin{equation}\tag*{$(\otimes)_\alpha$}
T_\alpha=\begin{cases}\{ x* \mathbf{b}^C\mid C\in\mathcal C_\alpha, x\in T\restriction\alpha\},&\text{if }\alpha\in\Gamma;\\
\{t\in{}^\alpha H_\kappa\mid\forall\bar\alpha<\alpha(t\restriction\bar\alpha\in T_{\bar\alpha})\},&\text{otherwise.}
\end{cases}
\end{equation}

Here we go. 
Let $T_0:=\{\emptyset\}$, and for all $\alpha<\kappa$,
let $T_{\alpha+1}:=\{ \conc{t}{i}\mid t\in T_\alpha,\allowbreak i<\max\{\alpha,\varsigma,\omega\}\}$.
Next, suppose that $\alpha\in\acc(\kappa)$ and that $\langle T_\beta\mid \beta<\alpha\rangle$ has already been defined.
The construction splits into two cases:

$\br$ If $\alpha\in\acc(\kappa)\setminus\Gamma$, then $\cf(\alpha)<\chi$, and we let $T_\alpha$ consist of the limits of all $\alpha$-branches through $T\restriction\alpha$.
This coheres with $(\otimes)_\alpha$ and ensures that the outcome tree will be $\chi$-complete.

$\br$ Now suppose $\alpha\in\Gamma$.
For every $C\in\mathcal C_\alpha$, we shall define a node $\mathbf{b}^C$, and then we would let:
$$T_\alpha:=\{ x* \mathbf{b}^C\mid C\in\mathcal C_\alpha, x\in T\restriction\alpha\}.$$

To obtain $\mathbf{b}^C$, we define a sequence $b^C\in\prod_{\beta\in\dm{C}}T_\beta$ by recursion.
Let $b^C(0):=\emptyset$.
Next, suppose $\beta^-<\beta$ are successive points of $\dm{C}$, and $b^C(\beta^-)$ has already been defined.
Let
$$b^C(\beta):=
\begin{cases}
b^C(\beta^-)*\sealantichain(\psi(\beta)*b^C(\beta^-),T\restriction(\beta+1),C(\beta)),&\text{if }\psi(\beta) \in T\restriction\beta^-;\\
\defaultaction(b^C(\beta^-),T\restriction(\beta+1)),&\text{otherwise}.
\end{cases}$$

If $\psi(\beta) \in T\restriction\beta^-$, 
then, by the induction hypothesis $(\otimes)_{\beta}$,
$\psi(\beta)*b^C(\beta^-)$ belongs to the normal tree $T \restriction (\beta+1)$,
so that the Extension Lemma (Lemma~\ref{extendfact}) together with $(\otimes)_{\beta}$ imply $b^C(\beta)$ is an element of $T_\beta$ extending $b^C(\beta^-)$.
If $\psi(\beta) \notin T\restriction\beta^-$, then again the Extension Lemma implies that $b^C(\beta)$ is in $T_\beta$ and extends $b^C(\beta^-)$.

The following is obvious.
\begin{dependencies}\label{dep7311}
For any two consecutive points $\beta^-<\beta$ of $\dom(b^C)$,
the value of $b^C(\beta)$ is completely determined by
$b^C(\beta^-)$, $\psi(\beta)$,$T\restriction(\beta+1)$ and $C(\beta)$.
\end{dependencies}

In the case $\beta\in\acc(\dm{C})$, we let $b^C(\beta):=\bigcup(\im(b^C\restriction\beta))$. The fact that the latter belongs to $T_\beta$ follows from Dependencies~\ref{dep7311}, $(\otimes)_\beta$ and $\sqx$-coherence.
 
This completes the definition of the level $T_\alpha$.

Having constructed all levels of the tree, we then let $T := \bigcup_{\alpha < \kappa} T_\alpha$.
It is clear that $T$ is a normal, prolific, $\varsigma$-splitting, $\chi$-complete and uniformly homogeneous
streamlined $\kappa$-tree.

\begin{claim} Let $A\s T$ be a maximal antichain.
Then there exists $\alpha<\kappa$ such that every node of $T_\alpha$ extends some element of $A \cap (T\restriction\alpha)$.
\end{claim}
\begin{proof} Let $\Omega:=A$ and $p:=\{T,A\}$.
Consider the club $D:=\{ \alpha\in\acc(\kappa)\mid T\restriction\alpha\s\phi[\alpha]\}$,
and then pick $\alpha\in D\cap\Gamma$ such that, for all $C\in\mathcal C_\alpha$ and $i < \alpha$,
$$\sup(\nacc(\dm{C})\cap B_{i}(\Omega,p,C))=\alpha.$$

Now, let $v\in T_\alpha$ be arbitrary, and we shall show that $v$ extends some element of $A\cap(T\restriction\alpha)$.
As $\alpha\in\Gamma$, let us fix $C\in\mathcal C_\alpha$ and $x \in T \restriction \alpha$ such that $v = x* \mathbf{b}^C$.
As $\alpha\in D$ and $x\in T\restriction\alpha$, we may also find $i<\alpha$ such that $\phi(i)=x$.

Fix $\beta \in \nacc(\dm{C}) \cap B_{i}(\Omega,p,C)$ 
large enough to ensure that $\beta^->\dom(x)$ for $\beta^-:=\sup(\dm C\cap\beta)$.
As $\beta\in B_i$, we obtain $\psi(\beta)=\phi(\pi(\beta))=\phi(i)=x$.
Thus, by writing $\bar x:=x*b^C(\beta^-)$, $\bar T:=T\restriction(\beta+1)$ and $\mho:=C(\beta)$,
we get that
$$b^C(\beta):=b^C(\beta^-)*\sealantichain(\bar x,\bar T,\mho).$$

Recalling Definition~\ref{actions}(2), we consider the set:
$$Q:=\{ z\in \bar T_{\beta(\bar{T})}\mid \exists y\in\mho( \bar x\cup y\s z)\}.$$
By now, we know that
$$Q=\{ z\in T_{\beta}\mid \exists y\in A\cap (T\restriction\beta)( (x*b^C(\beta^-))\cup y\s z)\}.$$
Since $\beta \in B_i(A,\{T,A\},C)$, we infer from Proposition~\ref{motivate}(2)
that $C(\beta) = A \cap (T \restriction \beta)$ is a maximal antichain in $T \restriction \beta$.
So $Q$ is nonempty, and $b^C(\beta)=b^C(\beta^-)*z$, for $z=\min(Q,{\lhd_\kappa})$.
Pick $y\in A\cap(T\restriction\beta)$ witnessing that $z\in Q$. Then:
$$v = x*\mathbf{b}^C \supsetneq =x* b^C(\beta)=x*(b^C(\beta^-)*z)=z\supsetneq y,$$
as sought.
\end{proof}
This completes the proof.
\end{proof}

\begin{thm}\label{corhom} If $\kappa$ is $({<}\chi)$-closed,
then $\p^\bullet(\kappa,\kappa,{\sqx^*},\kappa, \{E^\kappa_{\geq\chi}\}, \kappa)$ implies the existence of a
normal, prolific, $\chi$-complete, uniformly homogeneous $\kappa$-Souslin tree.
\end{thm}
\begin{proof} By Theorem~\ref{homogeneous}, using Theorem~\ref{get-full-coherence-from-sqleft*}
\end{proof}

\begin{definition} For a streamlined $\kappa$-tree $T$ and a subset $\Delta\s\kappa$,
we say that $T$ is $\Delta$-similar iff, for all $\delta\in\Delta\cap\acc(\kappa)$ and $x,y\in T_\delta$,
$\sup\{ \varepsilon<\delta\mid x(\varepsilon)\neq y(\varepsilon)\}<\delta$.
\end{definition}

It is easy to see that a streamlined $\kappa$-tree is coherent iff it is $\kappa$-similar.
It is also easy to see that in the construction of Theorem~\ref{homogeneous},
the outcome tree $T$ is $\Delta$-similar for the set $\Delta:=\{\delta\in\Gamma\mid |\mathcal C_\delta|=1\}$.
Therefore, we get:

\begin{thm}\label{prop25} Suppose $\p^\bullet(\kappa,2,{\sq},\kappa,\{\kappa\},2)$ holds. 
Then there exists a normal, slim, prolific, club-regressive, uniformly coherent $\kappa$-Souslin tree.\qed
\end{thm}
\begin{remark} The definition of \emph{club-regressive} trees may be found in \cite[\S2]{paper22}. For further analysis, see \cite[Proposition~2.5]{paper22}.
\end{remark}

\section{Some open problems}

It follows from Theorem~\ref{thm61}(11) that after forcing to add a single Cohen real over a model of ${\ch}+{\neg\diamondsuit(\aleph_1)}$,
a very strong instance of $\p^\bullet(\aleph_1,\ldots)$ holds, while $\diamondsuit(\aleph_1)$ fails.
We conjecture that Radin forcing can be used to answer the following question in the affirmative:
\begin{q} Is a conjunction of the form ${\p^\bullet(\kappa,\ldots)}+{\neg \diamondsuit(\kappa)}$ consistent for $\kappa$ inaccessible? 
\end{q}

Corollaries \ref{corollaryalajensen} and \ref{get-full-coherence} leave the following question open:
\begin{q} Assume ${\square(\kappa)}+{\diamondsuit(\kappa)}$.
\begin{enumerate}[(i)]
\item  Does $\p^-(\kappa,2,{\sq},1,\{\kappa\},2)$ hold?
\item Does (the weaker instance) $\p^-(\kappa,\kappa,{\sq},1,\{\kappa\},\kappa)$ hold?
\end{enumerate}
\end{q}

By Clauses (1) and (5) of Theorem~\ref{thm61}, the answer to (i) is affirmative for $\kappa=\aleph_1$
and $\kappa=\lambda^+$ such that $\lambda\ge\beth_\omega$.
By \cite[Theorem~4.3 and Corollary~2.5]{paper24}, the answer to (i) is also affirmative for $\kappa=\lambda^+$ such that $\lambda^{\aleph_0}=\lambda$.
Thus, the remaining cases are $\kappa$ inaccessible or $\aleph_1<\kappa<\beth_\omega$.
For $\kappa$ strongly inaccessible, a sufficient condition for (ii) is given by Theorem~\ref{cor430}.

\medskip

Our next question has to do with the parameter $\theta$ of the proxy principle.
\begin{q}\label{q73} For $\mathcal R\in\{ {\sqleftboth{\Omega}{\chi}}, {\sqstarleftboth{\Omega}{\chi}} , {\sqleftup{\Omega}_\chi}, {^\Omega{\sq}_\chi^*}\}$:
\begin{enumerate}[(i)]
\item Does $\p^-(\kappa,\mu,\mathcal R,1,\{\kappa\},\nu)$ entail $\p^-(\kappa,\mu,\mathcal R,\kappa,\{\kappa\},\nu)$?
\item Does $\p^\bullet(\kappa,\mu,\mathcal R,1,\{\kappa\},\nu)$ entail $\p^\bullet(\kappa,\mu,\mathcal R,\kappa,\{\kappa\},\nu)$?
\end{enumerate}
\end{q}

Some partial answers to (ii) include \cite[Lemma~3.7]{paper29}, \cite[Lemma~3.20]{paper32} and \cite[Lemmas 3.8 and 3.9]{paper28},
but there are more findings which haven't yet been published.
Motivated by \cite[Lemma~2.13]{paper28}, we hope to prove that, for $\kappa$ inaccessible, $\p^\bullet(\kappa,2,{\sq},1,\{\kappa\},2)$
entails the existence of graph of size and chromatic number $\kappa$ all of whose smaller subgraphs are countably chromatic.

\begin{q}\label{q74} Assume $\gch$. For a regular uncountable cardinal $\lambda$:
\begin{enumerate}[(i)] 
\item  Does  $\square_\lambda$ entail the existence of a uniformly coherent $\lambda^+$-Souslin tree?
\item Does the existence of a nonreflecting stationary subset of $E^{\lambda^+}_{<\lambda}$ entail the existence of a $\lambda$-free $\lambda^+$-Souslin tree?
\end{enumerate}
\end{q}

By \cite[Theorem~C]{paper29} and \cite[Theorem~B]{paper32}, the answer is affirmative for the analogous questions concerning $\lambda$ singular.
Also, note that an affirmative answer to Question~\ref{q73}(ii) would provide an affirmative answer to Question~\ref{q74}.

\begin{q} 
\begin{enumerate}[(i)] 
\item May $\clubsuit(\kappa)$ be waived from the hypothesis of Theorem~\ref{get-sigma-finite}? 
\item May the hypothesis $\p_\xi^-(\kappa, \mu, \mathcal R, \theta, \mathcal S,  \nu)$ of Theorem~\ref{mainpbullet} be weakened to $\p_\xi^-(\kappa, \mu, \mathcal R, \theta, \mathcal S,  \nu,1)$?\footnote{Recall Convention~\ref{omitsigma}.}
\end{enumerate}
\end{q}

\section{Acknowledgements}
The first author was supported by the Center for Absorption in Science,
Ministry of Aliyah and Integration, State of Israel.
The second author was partially supported by the European Research Council (grant agreement ERC-2018-StG 802756) and by the Israel Science Foundation (grant agreement 2066/18).

This paper is submitted to the proceedings of \emph{50 Years of Set Theory in Toronto}.
The two authors first met at the Toronto Set Theory Seminar, 
when Brodsky was a Ph.D. student of Stevo Todorcevic 
and Rinot was a Fields-Ontario postdoctoral fellow of Ilijas Farah, Stevo Todorcevic and Bill Weiss.

We thank James Cummings, Moti Gitik, Yair Hayut, Adi Jarden, Menachem Kojman, Chris Lambie-Hanson, Philipp L\"{u}cke, Matti Rubin (of blessed memory), Stevo Todorcevic, and Bill Weiss,
with whom we had stimulating discussions on the subject matter of this project throughout the years.

\bibliographystyle{alpha}

\begin{thebibliography}{MHD04}

\bibitem[AS93]{AbSh:403}
Uri Abraham and Saharon Shelah.
\newblock A $\Delta ^2_2$ well-order of the reals and incompactness of $L(Q^{MM})$.
\newblock {\em Ann. Pure Appl. Logic}, 59(1):1--32, 1993.

\bibitem[Alv99]{MR1711574}
Carlos Alvarez.
\newblock On the history of {S}ouslin's problem.
\newblock {\em Arch. Hist. Exact Sci.}, 54(3):181--242, 1999.

\bibitem[BR17a]{paper22}
Ari~Meir Brodsky and Assaf Rinot.
\newblock A microscopic approach to {S}ouslin-tree constructions. {P}art {I}.
\newblock {\em Ann. Pure Appl. Logic}, 168(11):1949--2007, 2017.

\bibitem[BR17b]{rinot20}
Ari~Meir Brodsky and Assaf Rinot.
\newblock Reduced powers of {S}ouslin trees.
\newblock {\em Forum Math. Sigma}, 5(e2):1--82, 2017.

\bibitem[BR19a]{paper29}
Ari~Meir Brodsky and Assaf Rinot.
\newblock Distributive {A}ronszajn trees.
\newblock {\em Fund. Math.}, 245(3):217--291, 2019.

\bibitem[BR19b]{paper26}
Ari~Meir Brodsky and Assaf Rinot.
\newblock More notions of forcing add a {S}ouslin tree.
\newblock {\em Notre Dame J. Form. Log.}, 60(3):437--455, 2019.

\bibitem[BR19c]{paper32}
Ari~Meir Brodsky and Assaf Rinot.
\newblock A remark on {S}chimmerling's question.
\newblock {\em Order}, 36(3):525--561, 2019.

\bibitem[BS86]{MR836425}
Shai Ben-David and Saharon Shelah.
\newblock Souslin trees and successors of singular cardinals.
\newblock {\em Ann. Pure Appl. Logic}, 30(3):207--217, 1986.

\bibitem[Cum97]{MR1376756}
James Cummings.
\newblock Souslin trees which are hard to specialise.
\newblock {\em Proc. Amer. Math. Soc.}, 125(8):2435--2441, 1997.

\bibitem[Dev78]{MR0491194}
Keith~J. Devlin.
\newblock A note on the combinatorial principles {$\diamondsuit (E)$}.
\newblock {\em Proc. Amer. Math. Soc.}, 72(1):163--165, 1978.

\bibitem[Dev83]{MR0732661}
Keith~J. Devlin.
\newblock Reduced powers of {$\aleph _{2}$}-trees.
\newblock {\em Fund. Math.}, 118(2):129--134, 1983.

\bibitem[Dev84]{MR750828}
Keith~J. Devlin.
\newblock {\em Constructibility}.
\newblock Perspectives in Mathematical Logic. Springer-Verlag, Berlin, 1984.

\bibitem[DJ74]{devlin1974souslin}
Keith~J. Devlin and H{\.a}vard Johnsbr{\.a}ten.
\newblock {\em The {S}ouslin problem}.
\newblock Lecture Notes in Mathematics, Vol. 405. Springer-Verlag, Berlin, 1974.

\bibitem[Dra74]{MR3408725}
Frank~R. Drake.
\newblock {\em Set theory---an introduction to large cardinals}, volume~76 of
  {\em Studies in Logic and the Foundations of Mathematics}.
\newblock North-Holland Publishing Co., Amsterdam, 1974.

\bibitem[Gre76]{MR485361}
John Gregory.
\newblock Higher {S}ouslin trees and the generalized continuum hypothesis.
\newblock {\em J. Symbolic Logic}, 41(3):663--671, 1976.

\bibitem[HJ99]{MR1697766}
Karel Hrbacek and Thomas Jech.
\newblock {\em Introduction to set theory}, volume 220 of {\em Monographs and
  Textbooks in Pure and Applied Mathematics}.
\newblock Marcel Dekker, Inc., New York, third edition, 1999.

\bibitem[HS85]{MR783595}
Leo Harrington and Saharon Shelah.
\newblock Some exact equiconsistency results in set theory.
\newblock {\em Notre Dame J. Formal Logic}, 26(2):178--188, 1985.

\bibitem[HSW10]{holz2010introduction}
Michael Holz, Karsten Steffens, and Edmund Weitz.
\newblock {\em Introduction to cardinal arithmetic}.
\newblock Modern Birkh\"{a}user Classics. Birkh\"{a}user Verlag, Basel, 2010.
\newblock Reprint of the 1999 original.

\bibitem[Jec67]{jech1967non}
Tom{\'a}{\v{s}} Jech.
\newblock Non-provability of {S}ouslin's hypothesis.
\newblock {\em Comment. Math. Univ. Carolinae}, 8(2):291--305, 1967.

\bibitem[Jen68]{jensen1968souslin}
Ronald~B Jensen.
\newblock Souslin's hypothesis is incompatible with {V}={L}.
\newblock {\em Notices Amer. Math. Soc.}, 15(6):935, 1968.

\bibitem[Jen72]{jensen1972fine}
R.~Bj{\"o}rn Jensen.
\newblock The fine structure of the constructible hierarchy.
\newblock {\em Ann. Math. Logic}, 4:229--308; erratum, ibid. 4:443, 1972.
\newblock With a section by Jack Silver.

\bibitem[Jus01]{just2001like}
Winfried Just.
\newblock {$\clubsuit$}-like principles under {CH}.
\newblock {\em Fund. Math.}, 170(3):247--256, 2001.

\bibitem[JW97]{just1997discovering}
Winfried Just and Martin Weese.
\newblock {\em Discovering modern set theory. {II}}, volume~18 of {\em Graduate  Studies in Mathematics}.
\newblock American Mathematical Society, Providence, RI, 1997.
\newblock Set-theoretic tools for every mathematician.

\bibitem[Kan11]{kanamori2011historical}
Akihiro Kanamori.
\newblock Historical remarks on {S}uslin's problem.
\newblock In {\em Set theory, arithmetic, and foundations of mathematics:
  theorems, philosophies}, volume~36 of {\em Lect. Notes Log.}, pages 1--12.
  Assoc. Symbol. Logic, La Jolla, CA, 2011.

\bibitem[KLY07]{MR2320769}
Bernhard K{\"o}nig, Paul Larson, and Yasuo Yoshinobu.
\newblock Guessing clubs in the generalized club filter.
\newblock {\em Fund. Math.}, 195(2):177--191, 2007.

\bibitem[KS93]{KjSh:449}
Menachem Kojman and Saharon Shelah.
\newblock {$\mu $-complete Suslin trees on $\mu ^+$}.
\newblock {\em Archive for Mathematical Logic}, 32:195--201, 1993.

\bibitem[Kun80]{MR597342}
Kenneth Kunen.
\newblock {\em Set theory}, volume 102 of {\em Studies in Logic and the
  Foundations of Mathematics}.
\newblock North-Holland Publishing Co., Amsterdam-New York, 1980.
\newblock An introduction to independence proofs.

\bibitem[Kur35]{kurepa1935ensembles}
Georges ({\DJ}uro) Kurepa.
\newblock Ensembles ordonn{\'e}s et ramifi{\'e}s.
\newblock {\em Th\`eses de l'entre-deux-guerres}, 175, 1935. 
\verb"http://www.numdam.org/item/THESE_1935__175__1_0/"
\newblock {\em Publications de l'Institut Math\'ematique Beograd}, 4:1--138, 1935.
	Included (without the Appendix) in~\cite[pp.~12--114]{kurepa1996selected}.

\bibitem[Kur38]{Kurepa-Aronszajn}
Georges Kurepa.
\newblock Ensembles lin{\'e}aires et une classe de tableaux ramifi{\'e}s
  (tableaux ramifi{\'e}s de {M}. {A}ronszajn).
\newblock {\em Publ. Math. Univ. Belgrade}, 6--7:129--160, 1938.
	Available (excluding pp.~135--138) at \verb"http://elib.mi.sanu.ac.rs/files/journals/publ/6/14.pdf".
	Included in~\cite[pp.~115--142]{kurepa1996selected}.

\bibitem[Kur52]{MR0049973}
{\DJ}uro Kurepa.
\newblock Sur une propri\'et\'e caract\'eristique du continu lin\'eaire et le
  probl\`eme de {S}uslin.
\newblock {\em Acad. Serbe Sci. Publ. Inst. Math.}, 4:97--108, 1952.
\verb"http://elib.mi.sanu.ac.rs/files/journals/publ/10/10.pdf".
	Included in~\cite[pp.~329--339]{kurepa1996selected}.

\bibitem[Kur96]{kurepa1996selected}
{\DJ}uro Kurepa.
\newblock {\em Selected papers of {\DJ}uro {K}urepa}.
\newblock Matemati\v{c}ki Institut SANU, Belgrade, 1996.
\newblock Edited and with commentaries by Aleksandar Ivi\'{c}, Zlatko
  Mamuzi\'{c}, \v{Z}arko Mijajlovi\'{c} and Stevo Todor\v{c}evi\'{c}.

\bibitem[Lar99]{MR1683897}
Paul Larson.
\newblock An {${\bf S}_{\rm max}$} variation for one {S}ouslin tree.
\newblock {\em J. Symbolic Logic}, 64(1):81--98, 1999.

\bibitem[Lam17a]{lambie2017aronszajn}
Chris Lambie-Hanson.
\newblock Aronszajn trees, square principles, and stationary reflection.
\newblock {\em Mathematical Logic Quarterly}, 63(3-4):265--281, 2017.

\bibitem[Lam17b]{narrow_systems}
Chris Lambie-Hanson.
\newblock Squares and narrow systems.
\newblock {\em J. Symb. Log.}, 82(3):834--859, 2017.

\bibitem[Lev02]{MR1924429}
Azriel Levy.
\newblock {\em Basic set theory}.
\newblock Dover Publications, Inc., Mineola, NY, 2002.

\bibitem[LL18]{lh_lucke}
Chris Lambie-Hanson and Philipp L\"{u}cke.
\newblock Squares, ascent paths, and chain conditions.
\newblock {\em J. Symb. Log.}, 83(4):1512--1538, 2018.

\bibitem[LR19]{paper28}
Chris Lambie-Hanson and Assaf Rinot.
\newblock Reflection on the coloring and chromatic numbers.
\newblock {\em Combinatorica}, 39(1):165--214, 2019.

\bibitem[L{\"u}c17]{MR3667758}
Philipp L{\"u}cke.
\newblock Ascending paths and forcings that specialize higher {A}ronszajn
  trees.
\newblock {\em Fund. Math.}, 239(1):51--84, 2017.

\bibitem[Mal96]{malykhin1996suslin}
Vyacheslav~Ivanovich Malykhin.
\newblock The Suslin hypothesis and its significance for set-theoretic mathematics.
\newblock {\em Russian Mathematical Surveys}, 51(3):419--437, 1996.

\bibitem[Mil43]{miller1943note}
Edwin~W. Miller.
\newblock A note on {S}ouslin's problem.
\newblock {\em Amer. J. Math.}, 65(4):673--678, 1943.

\bibitem[Ost76]{ostaszewski1976countably}
A.~J. Ostaszewski.
\newblock On countably compact, perfectly normal spaces.
\newblock {\em J. London Math. Soc. (2)}, 14(3):505--516, 1976.

\bibitem[Rin11a]{rinot_s01}
Assaf Rinot.
\newblock {J}ensen's diamond principle and its relatives.
\newblock In {\em Set theory and its applications}, volume 533 of {\em Contemp.
  Math.}, pages 125--156. Amer. Math. Soc., Providence, RI, 2011.


\bibitem[Rin11b]{rinot09}
Assaf Rinot.
\newblock On guessing generalized clubs at the successors of regulars.
\newblock {\em Ann. Pure Appl. Logic}, 162(7):566--577, 2011.


\bibitem[Rin14]{rinot11}
Assaf Rinot.
\newblock The {O}staszewski square, and homogeneous {S}ouslin trees.
\newblock {\em Israel J. Math.}, 199(2):975--1012, 2014.

\bibitem[Rin15]{rinot12}
Assaf Rinot.
\newblock Chromatic numbers of graphs - large gaps.
\newblock {\em Combinatorica}, 35(2):215--233, 2015.

\bibitem[Rin17]{paper24}
Assaf Rinot.
\newblock Higher {S}ouslin trees and the {GCH}, revisited.
\newblock {\em Adv. Math.}, 311(C):510--531, 2017.

\bibitem[Rin19]{paper37}
Assaf Rinot.
\newblock {S}ouslin trees at successors of regular cardinals.
\newblock {\em Math. Log. Q.}, 65(2): 200-204, 2019.

\bibitem[RZ20]{paper45}
Assaf Rinot and Jing Zhang.
\newblock Strongest colorings.
\newblock In preparation, 2020.

\bibitem[Roi90]{MR1028781}
Judith Roitman.
\newblock {\em Introduction to modern set theory}.
\newblock Pure and Applied Mathematics (New York). John Wiley \& Sons, Inc.,
  New York, 1990.
\newblock A Wiley-Interscience Publication.

\bibitem[Rud69]{rudin1969souslin}
Mary~Ellen Rudin.
\newblock Souslin's conjecture.
\newblock {\em Amer. Math. Monthly}, 76(10):1113--1119, 1969.

\bibitem[Sch14]{MR3243739}
Ralf Schindler.
\newblock {\em Set theory}.
\newblock Exploring independence and truth.
\newblock Universitext. Springer, Cham, 2014.

\bibitem[She84a]{sh:e4}
Saharon Shelah.
\newblock An $\aleph _2$ {S}ouslin tree from a strange hypothesis.
\newblock {\em Abs. Amer. Math. Soc.}, 160:198, 1984.

\bibitem[She84b]{Sh:176}
Saharon Shelah.
\newblock {Can you take Solovay's inaccessible away?}
\newblock {\em Israel Journal of Mathematics}, 48:1--47, 1984.

\bibitem[She99]{Sh:624}
Saharon Shelah.
\newblock {On full Suslin trees}.
\newblock {\em Colloquium Mathematicum}, 79:1--7, 1999.

\bibitem[SZ99]{ShZa:610}
Saharon Shelah and Jindrich Zapletal.
\newblock {Canonical models for $\aleph_1$ combinatorics}.
\newblock {\em Annals of Pure and Applied Logic}, 98:217--259, 1999.

\bibitem[ST71]{solovay1971iterated}
R.~M. Solovay and S.~Tennenbaum.
\newblock Iterated {C}ohen extensions and {S}ouslin's problem.
\newblock {\em Ann. of Math. (2)}, 94(2):201--245, 1971.

\bibitem[Sou20]{Souslin}
Mikhail~Yakovlevich Souslin.
\newblock Probl\`eme 3.
\newblock {\em Fundamenta Mathematicae}, 1(1):223, 1920.

\bibitem[SS78]{steen1978counterexamples}
Lynn~Arthur Steen and J.~Arthur Seebach, Jr.
\newblock {\em Counterexamples in topology}.
\newblock Springer-Verlag, New York-Heidelberg, second edition, 1978.

\bibitem[Ten68]{tennenbaum1968souslin}
Stanley Tennenbaum.
\newblock Souslin's problem.
\newblock {\em Proc. Nat. Acad. Sci. U.S.A.}, 59(1):60--63, 1968.

\bibitem[Tod84]{todorcevic1984trees}
Stevo Todor{\v{c}}evi{\'c}.
\newblock Trees and linearly ordered sets.
\newblock In {\em Handbook of set-theoretic topology}, pages 235--293.
  North-Holland, Amsterdam, 1984.

\bibitem[Vel86]{MR830071}
Boban Veli{\v{c}}kovi{\'c}.
\newblock Jensen's {$\square$} principles and the {N}ov\'ak number of partially
  ordered sets.
\newblock {\em J. Symbolic Logic}, 51(1):47--58, 1986.

\end{thebibliography}

\end{document}